\documentclass[12pt]{amsart}

\usepackage[all]{xy}
\usepackage{tikz}
\usepackage{amsmath,amssymb,amscd,
                     amsthm,latexsym,graphics,
                     enumerate,engord}

\theoremstyle{plain}
    \newtheorem{thm}{Theorem}[section]
    \newtheorem{lem}[thm]   {Lemma}
    \newtheorem{cor}[thm]   {Corollary}
    \newtheorem{prop}[thm]  {Proposition}

\theoremstyle{definition}

    \newtheorem{ex}[thm]{Example}
    
    \newtheorem{rem}[thm]{Remark}

\def\cat{{\mathrm{cat}\hskip1pt}}
\def\secat{{\mathrm{secat}\hskip1pt}}

\def\TC{{\mathrm{TC}\hskip1pt}}

\def\zcl{{\mathrm{zcl}\hskip1pt}}
\def\RP{{\mathbb{R}\mathrm{P}\hskip1pt}}

\def\Conf{{\mathrm{Conf}\hskip1pt}}
\def\TCs{{\mathrm{TC}_s\hskip1pt}}

\title[Cohomology of configurations on projective spaces]{The cohomology ring away from 2 of configuration spaces on real projective spaces}

\author[J.~Gonz\'alez]{Jes\'us Gonz\'alez\textsuperscript{*}}
\author[A.~Guzm\'an-S\'aenz]{Aldo Guzm\'an-S\'aenz\textsuperscript{\dag}}
\author[M.~Xicot\'encatl]{Miguel Xicot\'encatl\textsuperscript{\ddag}}
\thanks{\textsuperscript{*}~~Supported by Conacyt Research Grant 221221.}
\thanks{\textsuperscript{\dag}~~Supported by Conacyt Ph.D.~Scholarship 316605.}
\thanks{\textsuperscript{\ddag}~~Supported by Conacyt Research Grant 168349.}

\address{Departamento de Matem\'aticas, Centro de Investigaci\'on y de Estudios Avanzados del IPN, Av.~IPN 2508, Zacatenco, M\'exico City 07000, M\'exico}
\email{jesus@math.cinvestav.mx}
\email{aldo@math.cinvestav.mx}
\email{xico@math.cinvestav.mx}

\begin{document}

\begin{abstract}
Let $R$ be a commutative ring containing $1/2$. We compute the $R$-cohomology ring of the configuration space $\operatorname{Conf}(\mathbb{R}\mathrm{P}^m,k)$ of $k$ ordered points in the $m$-dimensional real projective space $\mathbb{R}\mathrm{P}^m$. The method uses the observation that the orbit configuration space of $k$ ordered points in the $m$-dimensional sphere (with respect to the antipodal action) is a $2^k$-fold covering of $\operatorname{Conf}(\mathbb{R}\mathrm{P}^m,k)$. This implies that, for odd $m$, the Leray spectral sequence for the inclusion $\operatorname{Conf}(\mathbb{R}\mathrm{P}^m,k)\subset(\mathbb{R}\mathrm{P}^m)^k$ collapses after its first non-trivial differential, just as it does when $\mathbb{R}\mathrm{P}^m$ is replaced by a complex projective variety. The method also allows us to handle the $R$-cohomology ring of the configuration space of $k$ ordered points in the punctured manifold $\mathbb{R}\mathrm{P}^m-\star$. Lastly, we compute the Lusternik-Schnirelmann category and all of the higher topological complexities of some of the auxiliary orbit configuration spaces.
\end{abstract}

\maketitle


\bigskip\bigskip\begin{center}
Draft version - \today
\end{center}\bigskip

\tableofcontents

\section{Introduction and motivation}
\label{intro}
We give an explicit presentation for the cohomology ring of configuration spaces of ordered points in real projective spaces. Before presenting our work, we review the theory where our results insert.

\bigskip
Since the work of Arnol'd, Fadell, and Fadell-Neuwirth~(\cite{arnold,fasolo,fane}) in the 1960's, a good understanding of the algebraic topology of ordered configuration spaces 
$$
\Conf(X,k)=\{(x_1,\cdots,x_k)\in X^k \;|\; x_i\neq x_j \mbox { for } i\neq j\}
$$
and their unordered analogues\footnote{The symmetric group on $k$ letters, $\Sigma_k$, acts freely on the right of $\Conf(X,k)$ by permutation of coordinates.} $B(X,k)=\Conf(X,k)/\Sigma_k$ has been a central need for many areas in mathematics, as well as in other disciplines such as physics, chemistry and computer science. Specifically, although the algebraic topology of configuration spaces turns out to be a key issue for many problems in algebraic geometry, knot theory, differential topology and homotopy theory, except for a few standard situations (configurations on Euclidean spaces and spheres), there is a rather surprisingly lack of explicit descriptions of the cohomology ring of configuration spaces. It is the aim of this paper to help mend such a situation by providing fully detailed descriptions in the case of ordered configurations on real projective spaces, as well as on punctured real projective spaces. In fact, rather than our explicit results, it is our methods (using orbit configuration spaces) that might be more interesting and could help shed light in other situations, specially when combined with already existing methods (some of which are mentioned later in the paper).

\medskip
Taking one step deeper into motivation requires reviewing a key construction, namely the space of finite ``linear'' combinations of points in a space $X$ with ``coefficients'' in a based space $(L,\star)$, where $\star$ is regarded as zero. The underlying set is
$$
C(X,L):= \left.\left(\;\coprod_{k\geq 0}\Conf(X,k)\times_{\Sigma_k} L^{k} \right)\right/ \sim
$$
where $\sim$ is the equivalence relation generated by
$$
((x_1,\ldots,x_k),(\ell_1,\ldots\ell_k))\sim
((x_1,\ldots,\widehat{x}_i,\ldots,x_k),(\ell_1,\ldots\widehat{\ell}_i,\ldots\ell_k))
$$
if $\ell_i=\star$. (We agree to set $\Conf(X,0)=L^0=\star$, which gives the base point in $C(X,L)$.) There are canonical maps
$$
\coprod_{0\leq k\leq r}\Conf(X,k)\times_{\Sigma_k} L^{k} \stackrel{\pi_r}\longrightarrow C(X,L)
$$
and we take the largest topology on $C(X,L)$ so that all the maps $\pi_r$ are continuous. 

\medskip
The resulting space is most interesting when $X=M$ is a smooth manifold. Indeed, $H^*(C(M,L))$ can be used to approach three seemingly different geometric objects: To begin with, the Anderson-Trauber spectral sequence,
which converges to the cohomology of the space of based maps from $M$ to $L$, has $E_2$-term given by $H^*(C(M-\mathrm{pt},L))$. On the other hand, for a suitably chosen $L$ (depending only on $\dim(M)$), $H^*(C(M,L))$ gives the $E_2$-term of the Gelfand-Fuks spectral sequence
converging to the continuous cohomology of the Lie algebra of compactly supported C${}^\infty$-vector fields on $M$. Furthermore, as proved by Haefliger
and Bott-Segal,
the latter situation is closely related to D.~McDuff's function spaces $\Gamma_k(M)$ of sections of degree $k$ of the bundle obtained by one-point compactification of each fiber in the tangent bundle of $M$. In such a context, the basic ingredient in a spectral sequence converging to $H^*(\Gamma_k(M))$ is $H^*(C(M,S^0))$. 

\medskip
On purely homotopy theoretic grounds, configuration spaces occupy an important position. To start with, the classical case of $C(\mathbb{R}^m,L)$ leads to a complete and useful theory of homology operations for $m$-fold loop spaces.
Also particularly interesting is the very fruitful connection (via Artin groups) between configuration spaces and  geometric topology: On the one hand, $\Conf(\mathbb{R}^2,k)$ and $B(\mathbb{R}^2,k)$ are aspherical spaces whose fundamental groups are the classical Artin braid groups on $k$ strings (we get the pure braid groups version, in the case of the ordered configuration space). In slightly more general terms, $\Conf(M,k)$ and $B(M,k)$ are Eilenberg-MacLane spaces $K(\pi,1)$ for  other versions of Artin braid groups $\pi$, if $M$ is a surface not homeomorphic to $S^2$ or $\mathbb{R}\mathrm{P}^2$. On the other hand, Artin groups play a role on the homotopy theory of mapping class groups of orientable punctured surfaces with boundary components, and on the stable structure of their classifying spaces.
In addition, there is an appealing connection between the Kervaire invariant problem and a distributivity law for braid groups.

\medskip
It should not be surprising that configuration spaces can be found at the heart of some of the major breakthroughs in homotopy theory. The point starts by noticing that Brown-Gitler spectra, one of the fundamental pieces in homotopy theory, have played a critical role in the solution of central problems such as:
\begin{enumerate} 
\item the Immersion Conjecture for manifolds, 
\item the Segal Conjecture relating the Burnside ring of a finite group $G$ to the stable cohomotopy of the classifying space $BG$, 
\item the Sullivan conjecture on the homotopy nature of the mapping space Map$(BG,X)$ for $G$ a finite group and $X$ a finite cell complex, and 
\item the disproof of the Doomsday Conjecture, i.e.~Mahowald's construction of infinite families of elements having Adams filtration 2 in the stable homotopy groups of spheres.
\end{enumerate}
The punch line then comes from the fact that Brown-Gitler spectra can be realized as Thom spectra of vector bundles over configuration spaces on Euclidean spaces. 
In particular, this leads to a classification up to cobordism of braid group oriented manifolds.

\medskip
Despite of having been studied intensively, except for a few special cases, explicit presentations of the cohomology ring of configuration spaces were largely unknown in the late 80's. The work back then, mostly represented by~\cite{gibe,bct89,bcm89,rbwh,CTconf,CTconf2,Cohen,cinco,clm,fuckshs,fuma,millof,VAINSHTEIN}, focuses mainly on additive descriptions with field coefficients. The work of Kriz and Totaro~(\cite{Kriz,totaro}) in the early 90's, and of Felix-Tanr\'e and Felix-Thomas (\cite{FTun,FT}) in the 2000's settled much of the multiplicative structure (mostly with field coefficients) in the case where $M$ is a projective algebraic variety and, more generally, if $M$ is rationally formal (mostly with characteristic-zero coefficients). But the problem is still largely unsettled for more general manifolds and coefficients. 

\medskip
In this state of affairs, and to the best of the authors' knowledge, the present paper is the first successful attempt to get at a full description of cohomology rings $H^*(\Conf(M,k);R)$ (in terms of generators, relations and explicit additive bases) for a not necessarily orientable manifold $M$ using not necessarily field coefficients $R$ (see the explicit descriptions summarized in the next section).

\medskip
Not reflected by our accomplishments is the fact that this work arose in part from a desire of understanding the collapsibility of the Cohen-Taylor spectral sequence in the case of the real projective spaces. Further indications on this direction are given in the paragraph following Remark~\ref{isomorfismoimpar} in the next section. 
An unexpected bonus of our work is that we get examples of homotopy equivalent manifolds whose configuration spaces have different stable homotopy types and whose loop spaces are not homotopy equivalent. This property is contextualized in the discussion following Corollary~\ref{longisalva} at the end of the next section. 

\medskip
We close this introductory section by briefly mentioning a couple of relatively new areas (in a sense dual to each other) of current research where configuration spaces have played a central role---thus adding to the general point in this section: explicit information on the algebraic topology of configuration spaces is eagerly awaited and needed.


For a smooth algebraic variety $X$, Fulton-MacPherson's compactification\footnote{We follow the notation suggested in~\cite{MR2099074}.} $\Conf[X,n]$ of the configuration space $\Conf(X,n)$ is the result of a series of blow-ups of the naive compactification $\Conf(X,n)\subset X^n$ in order to make the complement of $\Conf(X,n)$ in $\Conf[X,n]$ a divisor with normal crossings. The idea was soon translated to the manifold setting by Axelrod and Singer who defined, for a smooth\footnote{ An alternate approach to $\Conf[M,n]$ through the theory of operads was fully developed and extended to arbitrary manifolds by Markl following pioneering work by Getzler and Jones.} manifold $M$, a corresponding compactification $\Conf[M,n]$ of $\Conf(M,n)$ by using spherical blow-ups. The construction has proven vital in quantum topology, e.g.~in the identification of invariants of three-manifolds coming from Chern-Simons theory. More recently, Sinha (partly motivated by earlier work of Kontsevich) has given an elementary construction of the compactification $\Conf[M,n]$---avoiding the use of blow-ups---with far reaching applications to the study of knot spaces from the point of view of the Goodwillie-Weiss manifold calculus.

Lastly, it should be mentioned that the known homological stability of unordered configuration spaces on open manifolds is one of the phenomena that has motivated the recent development of the concept of topological chiral homology (also known as factorization homology, or higher Hochschild cohomology)---i.e.~the study of homology theories for $n$-manifolds with values in spaces and whose coefficient systems are $n$-disk algebras or $n$-disk stacks.


\section{Main results and their contextualization} \label{seccionintro}
The goal of this work is to give a description of the cohomology ring away from 2 of configuration spaces of pairwise distinct ordered points in (either regular or punctured) real projective spaces.  Our main results are stated next where $ k $ and $ n $ stand for integers greater than 1. 

\begin{thm}[Theorem~\ref{thminvariantesimpar}]
Suppose $R$ is a commutative ring with unit where $2$ is invertible. For $n\geq 2$ odd, there is an $R$-algebra isomorphism
\[
H^*(\Conf(\mathbb{R}\mathrm{P}^n,k);R)\cong \Lambda(\iota_n)\otimes R[\mathcal{C^+}]/\mathcal{K},
\]
where $\iota_n$ has degree $n$ and is the image of the generator in $\mathbb{R}\textnormal{P}^n$ under the projection on the first coordinate $\Conf(\mathbb{R}\textnormal{P}^n,k)\overset{\pi_1}{\longrightarrow} \mathbb{R}\textnormal{P}^n$, the generators in $\mathcal{C}^+$ have degree $n-1$ and are detailed at the beginning of Section $4$, and the relations $\mathcal{K}$ are specified in Theorem~\ref{isoinvariantesimpar}.
\end{thm}

\begin{thm}[Theorem~\ref{isoinvariantespar}]
Let $R$ be a commutative ring with unit where $2$ is invertible. For $n\geq2$ even, there is an $R$-algebra isomorphism
\[
H^*(\Conf(\mathbb{R}\mathrm{P}^n,k);R) \cong \Lambda(\omega_{2n-1})\otimes R[\mathcal{E}]/\mathcal{J}, 
\]
where $\omega_{2n-1}$ is a generator of degree $2n-1$ specified in Theorem \ref{2.3}, the set of generators $\mathcal{E}$  have degree $2n-2$ and are defined just above Lemma \ref{relsmultI}, and $\mathcal{J}$ is the ideal generated by the relations in Lemma \ref{relsmultI}.
\end{thm}

The relations in $\mathcal{K}$ are a reincarnation of those defining the cohomology of the standard ordered configuration spaces on Euclidean spaces (such a fact will be used in Remark~\ref{isomorfismoimpar}, and will be explained in the paragraph following Remark~\ref{isomorfismoimpar}  within the context of the Cohen-Taylor spectral sequence). On the other hand, although the relations defining $\mathcal{J}$ are particularly involved (their explicit description requires a couple of pages at least), we manage to describe an explicit fully-working additive basis for the tensor factor $R[\mathcal{E}]/\mathcal{J}$ (see Theorem~\ref{invariantespar}).

\begin{rem}\label{isomorfismoimpar}
Theorem~\ref{thminvariantesimpar} and the known description of the cohomology ring of configuration spaces on spheres~(\cite[pp.~112--114]{FH} and~\cite[Theorems~2.4, 5.1, and 5.4]{FZ}) imply that there is a ring isomorphism $H^*(\Conf(S^n,k);R)\cong H^*(\Conf(\mathbb{R}\mathrm{P}^n,k);R)$ provided $n$ is odd. Compare with Remark~\ref{4.6}. But there is no such an isomorphism if $n$ is even, in view of Theorem~\ref{isoinvariantespar}. 
\end{rem}

It is illuminating to look at the above results within the context of the Cohen-Taylor spectral sequence for a space $X$, i.e.~the Leray spectral sequence for the open inclusion $\Conf(X,k)\hookrightarrow X^k$. To be precise, in the following considerations we take cohomology with rational coefficients, and we let $M$ stand for a connected $m$-dimensional manifold with $m\geq2$. Recall from~\cite[Theorem~1]{totaro} (see also~\cite[pp.~117--119]{CTconf} and~\cite[Theorem~1.1]{Kriz}) that, if $M$ is  oriented and closed, the initial term $E_2$ and the first non-trivial differential $\delta_m$ in the Cohen-Taylor spectral sequence for $M$ depend only on the cohomology ring and orientation class of $M$. Explicitly,  
\begin{itemize}
\item[(a)] $E_2$ is the $H^*(M)^{\otimes k}$-algebra generated by exterior classes $A'_{i,j}\in E_2^{0,m-1}$, with $1\leq j<i\leq k$, subject to the relations in~(\ref{relacionesusual}) at the beginning of the next section together with the relations $(\pi_i^*(x)-\pi_j^*(x)) A_{i,j}=0$ where $\pi_\ell\colon X^k\to X$ denotes projection on the $\ell$-th coordinate, and
\item[(b)] all differentials $\delta_i$ with $2\leq i<m$ are trivial, while $\delta_m$ vanishes on $H^*(M)^{\otimes k}$ and, therefore, is determined by $\delta_m(A_{i,j})=\pi_{i,j}^*(\Delta_M)$ where $\Delta_M\in H^{m}(X)$ is the diagonal class of $M$, and $\pi_{i,j}\colon X^k\to X\times X$ is the projection $\pi_{i,j}(x_1,\ldots,x_k)=(x_i,x_j)$.
\end{itemize}
Totaro shows in addition that, when $M$ is a complex projective variety, the spectral sequence collapses from its $E_{m+1}$-term yielding an algebra isomorphism $H^*(\Conf(M,k))\cong H^*(E_2,\delta_m)$~(\cite[Theorem~3]{totaro}). More generally, as a consequence of~\cite[Theorem~3.2]{gibe}, Felix and Thomas prove in~\cite[Theorem~1]{FT} that such a collapsing property must hold when $M$ is rationally formal (they also prove that the collapsing fails when $M$ is a simply connected manifold carrying suitable non-trivial Massey products\footnote{We thank Professor Totaro for kindly pointing out that this is dealt with in~\cite{FT}.}). Since real projective spaces are rationally formal, the above collapsibility phenomenon holds in the case of $\mathbb{R}\mathrm{P}^n$ whenever $n$ is odd (so that the orientability requirement holds). This is of course compatible with Theorem~\ref{thminvariantesimpar}. Indeed, since the rational cohomology rings and orientation classes of $S^n$ and $\mathbb{R}\mathrm{P}^n$ agree, the Cohen-Taylor spectral sequences for these two spaces agree up to their $m$ stage, after which both collapse by rational formality. This yields the isomorphism in Remark~\ref{isomorfismoimpar}. Of course, we have noted such a ring isomorphism when coefficients other than the rationals are used (as long as 2 is invertible). Furthermore, in view of Theorem~\ref{isoinvariantespar} (and specially by the complexity of the relations defining $\mathcal{J}$), we would expect that an eventual description of the Cohen-Taylor spectral sequence for a non-orientable real projective space would be more elaborate that the one in items~(a) and~(b) above.

\smallskip
Even though our methods do not make direct use of the Fadell-Neuwirth fibration
\begin{equation}\label{fnf}
\Conf(\RP^n-\star,k)\to\Conf(\RP^n,k+1)\to\RP^n,
\end{equation}
our approach to Theorems~\ref{thminvariantesimpar} and \ref{isoinvariantespar} allows us to get information on the cohomology ring of configuration spaces on punctured real projective spaces.

\begin{thm}[Theorem~\ref{rpnponchadoimpar}]
Let $R$ be a commutative ring with unit where $2$ is invertible. For $n\geq 2$ odd, there is an isomorphism
\[
H^{*}(\Conf(\mathbb{R}\mathrm{P}^n-\star,k);R)\cong  R[\mathcal{C^+}]/\mathcal{K}
\] 
of $R$-algebras.
\end{thm}
\begin{thm}[Theorem~\ref{rpnponchadopar}]
Let $R$ be a commutative ring with unit where $2$ is invertible. For $n\geq 2$ even, there is an $R$-algebra isomorphism
\[
H^{*}(\Conf(\mathbb{R}\mathrm{P}^n-\star,k);R)\cong R[\mathcal{E}']/\mathcal{J'}, 
\]
where the generators $\mathcal{E}'$  and the relations $\mathcal{J}'$ are detailed in Section \ref{section4}.
\end{thm}

Just as in the case of non-punctured spaces, we describe explicit fully-working additive basis for the cohomology rings in Theorems~\ref{rpnponchadoimpar} and~\ref{rpnponchadopar} (see Theorem~\ref{refparbasespult}). In particular, the fact that all these cohomology groups are $R$-free of rank independent of the actual ring $R$, imply:
\begin{cor}\label{longisalva}
There is no odd torsion in the integral cohomology rings of $\Conf(\mathbb{R}\mathrm{P}^n,k)$ and $\Conf(\mathbb{R}\mathrm{P}^n-\star,k)$.
\end{cor} 

Longoni and Salvatore show in \cite[Theorem~2]{SalvatoreLongoni} that there are $3$-dimensional twisted lens spaces having the same homotopy type, but whose $k$-points configuration spaces fail to be homotopy equivalent for  all $ k \geq 2 $. As a consequence of the main results of this paper, we get a new family of examples for which the homotopy invariance of configuration spaces fails. Namely, $\Conf(\mathbb{R}\textrm{P}^{n},k)$ and $\Conf(\mathbb{R}\textrm{P}^{n+1}-\star,k)$ are not homotopy equivalent when $k\geq 3$, even though $\mathbb{R}\textrm{P}^{n}\simeq\mathbb{R}\textrm{P}^{n+1}-\star$. Actually, $\Conf(\mathbb{R}\textrm{P}^{n},k)$ and $\Conf(\mathbb{R}\textrm{P}^{n+1}-\star,k)$ do not have isomorphic cohomology groups. Indeed, for $n$ odd and $k\geq 3$, Theorem \ref{thminvariantesimpar} implies $H^{n-1}(\Conf(\mathbb{R}\mathrm{P}^n,k))\neq 0$ while, by Theorem~\ref{rpnponchadopar}, $H^{n-1}(\Conf(\mathbb{R}\mathrm{P}^{n+1}-\star,k))=0$. Consequently, our results imply that 
\begin{equation}\label{ejemploestructurado}
\mbox{$\Conf(\mathbb{R}\mathrm{P}^n,k)$ and $\Conf(\mathbb{R}\mathrm{P}^{n+1}-\star,k)$ cannot even be {\it stably$$\hspace{.4mm}} homotopy equivalent.}
\end{equation}
Unlike the example by Longoni and Salvatore, the examples in~(\ref{ejemploestructurado}) fail to deal with closed manifolds of a fixed dimension. Yet our examples illustrate the importance of the two additional hypotheses in \cite[Theorem~A]{AoiuinaKlein} where Aouina and Klein prove the {\it stable} homotopy invariance of configuration spaces assuming that the manifolds to which one takes configurations are not only homotopy equivalent but are closed (PL) manifolds of a fixed dimension. Likewise, our results (and a standard argument using the Serre spectral sequence) show that $\Conf(\mathbb{R}\mathrm{P}^n,k)$ and $\Conf(\mathbb{R}\mathrm{P}^{n+1}-\star,k)$ cannot have homotopy equivalent loop spaces. This now illustrates the importance of the (implicit) additional hypotheses in \cite[Theorem~0.1]{Levitt} where Levitt proves (in particular) a homotopy equivalence between the loops spaces of the $k$-point configuration spaces associated to two homotopy equivalent closed manifolds of a fixed dimension.

\smallskip
Note that the configuration spaces $\Conf(\mathbb{R}^n,k)$ and $\Conf(\mathbb{R}^{n+1},k)$ also illustrate the two phenomena discussed in the previous paragraph. But our examples with real projective spaces are more interesting in at least two ways. For one, we use homotopically non-trivial manifolds. Secondly, our examples involve ``really'' different cohomology rings, in the sense that, unlike the examples with Euclidean spaces, the cohomology rings $H^*(\Conf(\mathbb{R}\mathrm{P}^n,k);R)$ and $H^*(\Conf(\mathbb{R}\mathrm{P}^{n+1}-\star,k);R)$ do not just differ by a degree rescaling.

\bigskip
Our proofs of the results described so far use the notion of an orbit configuration space. We describe the explicit spaces we need, as well as their cohomological properties, after recalling the general definition of an orbit configuration space.

For a positive integer $k$ and a topological space $X$ with a (say left) action of a group $G$, the orbit configuration space of $k$ ordered points on $X$, denoted by $\Conf_G(X,k)$, is the subspace of $X^k$ consisting of the $k$-tuples $(x_1,\ldots,x_k)\in X^k$ such that $G\cdot x_i \neq G\cdot x_j$ whenever $i\neq j$. For instance, we recover the definition of an (usual) ordered configuration space when $G$ is the trivial group. Note that the orbit configuration space $\Conf_G(X,k)$ inherits the coordinate-wise action of $G^k$ given by
\begin{equation}\label{coordwisetion}
(g_1,\ldots,g_k)\cdot (x_1,\ldots,x_k)=(g_1\cdot x_1,\ldots,g_k\cdot x_k).
\end{equation}

We are interested in the antipodal action of $\mathbb{Z}_2=\{\pm1\}$ on the sphere $S^n$ and on the double punctured sphere $S^n-\{\pm\star\}$, where $\star$ is some fixed base point of $S^n$. The cohomology rings of the orbit configuration spaces 
\begin{equation}\label{losdosespacios}
\Conf_{\mathbb{Z}_2}(S^n,k) \mbox{ \ \ \ and \ \ \ } \Conf_{\mathbb{Z}_2}(S^n-\{\pm\star\},k)
\end{equation}
and the resulting action of $(\mathbb{Z}_2)^k$ in their cohomology have been considered in~\cite[Theorems~12, 14, and~17, and Lemma~7]{FZ2} and in~\cite[Theorems~1.1 and~5.2 and Tables~1 and~2]{Xico}. Unfortunately both of these works contain a few subtle typos, mistakes, and gaps---see Remarks~\ref{xicoerror}, \ref{otroerrordeFZ}, \ref{remark2.4}, \ref{xicostable2}, and~\ref{3.6} for more precise indications. Our viewpoint corrects and extends the methods and results in~\cite{Xico} (and, indirectly, in~\cite{FZ2}) to prove, in particular, the following result (which encompasses Theorems~\ref{2.1},~\ref{2.2},~\ref{2.3} in this paper) on the cohomology of the orbit configuration spaces in~(\ref{losdosespacios})\footnote{The description of the action of $(\mathbb{Z}_2)^k$ on $H^*(\Conf_{\mathbb{Z}_2}(S^n,k);R)$ and on $H^*(\Conf_{\mathbb{Z}_2}(S^n-\{\pm\star\},k);R)$ requires additional preparatory considerations lying outside the displaying scope of this section. Details are provided in Section~\ref{section3} (Theorems~\ref{accion} and~\ref{accionpar} deal with the case of $\Conf_{\mathbb{Z}_2}(S^n-\{\pm\star\},k)$, while Corollaries~\ref{acciontotalnimpar} and~\ref{acciontotalnpar} deal with the case of $\Conf_{\mathbb{Z}_2}(S^n,k)$).}:

\begin{thm}\label{omnibucito}
Let $R$ denote a commutative ring with unit (where 2 is not necessarily invertible).
\begin{enumerate}
\item For $n\geq2$ odd\footnote{This restriction can be waived if the characteristic of $R$ is 2.}, there is an $R$-algebra isomorphism
$$
H^*(\Conf_{\mathbb{Z}_2}(S^n,k);R)\cong \Lambda[\iota_n]\otimes H^*(\Conf_{\mathbb{Z}_2}(S^n-\{\pm\star\},k-1);R)
$$
where $\iota_n$ has degree $n$ and is the image of a generator in $S^n$ under the projection on the first coordinate $\Conf_{\mathbb{Z}_2}(S^n,k)\to S^n$.
\item Assume that the characteristic of $R$ is either zero or an odd integer. For $n\geq 2$ even, there is an $R$-algebra isomorphism $$H^*(\Conf_{\mathbb{Z}_2}(S^n,k);R)\cong\left(\rule{0mm}{3mm}\Lambda(\iota_n,\omega_{2n-1})/(2\iota_n,\iota_n\hspace{.3mm}\omega_{2n-1})\right)\otimes R[\mathcal{B}]/J$$ where $\iota_n$ and $\omega_{2n-1}$ are as above, and where the generators in $\mathcal{B}$ (all having degree $n-1$) and the relations in $J$ are described in~(\ref{genersbgari}) and its immediate considerations.
\item For $n\geq2$, there is an $R$-algebra isomorphism 
$$
H^*(\Conf_{\mathbb{Z}_2}(S^n-\{\pm\star\},k-1);R)\cong R[\mathcal{A}]/I
$$
where the generators in $\mathcal{A}$ (all having degree $n-1$) and the relations in $I$ are defined in~(\ref{agarigoleada}) and its immediate considerations.
\end{enumerate}
\end{thm}

Just as in the case of the standard configuration spaces, we will describe explicit fully-working additive basis for the cohomology rings of the orbit configuration spaces in Theorem~\ref{omnibucito}.
  
\smallskip
As an application, we study the Lusternik-Schnirelmann category (LScat) and the higher topological complexities (H$\TC$'s) of some of the orbit configuration spaces considered so far. We start with a brief review of these categorical concepts.

\smallskip
The LScat arose from Lusternik-Schnirelmann's study in the late 1920's of the calculus of variations in the large, in particular from their results on the existence of geodesics on topological spheres. In the case of a manifold $M$, LScat gives a lower bound for the number of critical points that {\it any} smooth real-valued function on $M$ can have. With an apparently different motivation, the notion of HTC has been recently motivated and actively studied due to its connection with the sequential motion planning problem in robotics. In both cases, the most espectacular achievements have come from the homotopy facets of these invariants. Despite their differences in time and motivation, both concepts are particular cases of Schwarz's notion of the sectional category (or genus) of a fibration.

\smallskip
The sectional category of a fibration $p\colon E \to B$, denoted by $\secat(p)$, is one less than the minimal number of open sets $U$ covering $B$ so that $p$ admits a local section on each~$U$. If no such finite cover of $B$ exists, we agree to set $\secat(p)=\infty$. In these terms, the LScat of a path connected space $X$, $\cat(X)$, is the sectional category of the evaluation map $e_1\colon P_0(X)\to X$, $\gamma\mapsto\gamma(1)$, where $P_0(X)$ is the space of paths $\gamma\colon[0,1]\to X$ sending $0$ to some chosen\footnote{This definition is independent of the chosen base point.} base point of $X$. Likewise, for an integer $s\geq2$, the $s$-th HTC of a path connected space $X$, $\TC_s(X)$, is the sectional category of the evaluation map $e_s\colon P(X)\to X^s$, $\gamma\mapsto\gamma(0,\frac1{s-1},\ldots,\frac{s-2}{s-1},1)$, where $P(X)$ stands for the space of all paths in $X$. The books~\cite{Cornea,MR2455573} contain a thorough discussion of the meaning, motivation, and general properties of these concepts, including a list of relevant references.

\smallskip
For our purposes, and in order to homogenize the following statements, it will be convenient to write $\TC_1(X)$ as a substitute for $\cat(X)$. We show that the LScat and all the HTC's of the orbit configurations spaces in item 3 of Theorem~\ref{omnibucito} can be computed cohomologically (Theorem~\ref{tceszcls} in this paper). For instance, Corollaries~\ref{forzudo1} and~\ref{highertc} in this paper can be combined into the following assertion:

\begin{cor}\label{s1genezado}
For positive integers $n$ and $s$ with $n>2$, $\TC_s(\Conf_{\mathbb{Z}_2}(S^n-\{\pm\star\},k))=sk-\delta_{n,s}$ where $\delta_{n,s}\in\{0,1\}$ and, in fact, $\delta_{n,s}=0$ if $s=1$ or $n$ is odd.
\end{cor}

The paper ends with evidence (Theorem~\ref{jugoso} and Remark~\ref{jsuydbe}) for the conjecture that Corollary~\ref{s1genezado} can be 
\begin{itemize}
\item extended to consider the case $n=2$ and, at the same time, 
\item made fully precise by setting $\delta_{2,1}=0$ and $\delta_{n,s}=1$ if $n$ is even and $s\geq2$.
\end{itemize}

\section{Proof methodology}
\label{proofmethod}
The five theorems above describing cohomology rings have technically-involved and highly-compu\-tational proofs. As a result, a non-specialist reader might loose track of the global picture while checking all needed details. Thus, in order to increase readability of our work, we now describe our general strategy of proofs. We also take the chance to describe the organization of the paper, and to say a few words about the need of the various hypothesis made on the cohomology coefficients.

\smallskip
We have already noted the coordinate-wise $(\mathbb{Z}_2)^k$-action on  $\Conf_{\mathbb{Z}_2}(S^n,k)$ and, by restriction, on $\Conf_{\mathbb{Z}_2}(S^n-\{\pm\star\},k)$. These actions are clearly free. As noted in~\cite[Proposition~2.1]{Xico}, the corresponding orbit spaces are homeomorphic to $\Conf(\mathbb{R}\mathrm{P}^n,k)$ and $\Conf(\mathbb{R}\mathrm{P}^n-\star,k)$, respectively. We thus get $(\mathbb{Z}_2)^k$-fold covering projection maps
\begin{equation}\label{lascovrgsdosdeellas}
\rho_{n,k}\colon \Conf_{\mathbb{Z}_2}(S^n,k)\to\Conf(\mathbb{R}\mathrm{P}^n,k) \mbox{ \ and \ } \rho'_{n,k}\colon\Conf_{\mathbb{Z}_2}(S^n-\{\pm\star\},k)\to\Conf(\mathbb{R}\mathrm{P}^n-\star,k).
\end{equation}
We compute the cohomology of $\Conf(\mathbb{R}\mathrm{P}^n,k)$ and $\Conf(\mathbb{R}\mathrm{P}^n-\star,k)$
via the spectral sequences associated to the coverings $\rho_{n,k}$ and $\rho'_{n,k}$, that is, the Serre spectral sequences\footnote{In this paper we use without further notice the standard fact that, for a commutative ring with unit $R$, the Serre spectral sequence associated to a fibration $\pi\colon E\to B$ with fiber $F$ has $E_2$ term given by $E_2^{p,q}=H^p(B;H^q(F;R))$, where the system of coefficients can be non-trivial, and converges as an algebra to $H^*(E;R)$.} for 
the maps classifying these two covering projections.

The first needed ingredient is, then, a description of the cohomology rings of the source spaces of $\rho_{n,k}$ and $\rho'_{n,k}$. In the case of $\rho'_{n,k}$, this is attained by means of an inductive process based on the analysis of the Serre spectral sequences of the fibrations
$\Conf_{\mathbb{Z}_2}(S^n-\{\pm\star\},k)\to\Conf_{\mathbb{Z}_2}(S^n-\{\pm\star\},k-1)$ given by projection on the first $k-1$ coordinates. All these spectral sequences have simple systems of coefficients, collapse, and by sparseness cannot have additive extension problems (for any cohomology coefficients). The assembling of the multiplicative structure from that of their $E_\infty$ term has to be worked out carefully, though, and this is where the work in~\cite{Xico} has to be diligently reviewed and corrected. This initial step is completed in Section~\ref{section2}.

\smallskip
The information above allows us to compute the cohomology of the source space of $\rho_{n,k}$ via the Serre spectral sequence of the fibration
\begin{equation}\label{laqueeslafibra}
\Conf_{\mathbb{Z}_2}(S^n-\{\pm\star\},k-1)\to\Conf_{\mathbb{Z}_2}(S^n,k)\to S^n,
\end{equation}
where the second map is projection on the first coordinate. This spectral sequence has a simple system of coefficients (just because $n\geq2$) and collapses when $n$ is odd, but it has a single layer of non-trivial differentials when $n$ is even. Multiplication by 2 plays a key role in such differentials, and this is where we use the additional hypothesis in item 2 of Theorem~\ref{omnibucito}---so that multiplication by 2 yields a monomorphism. In characteristic 2, the spectral sequence collapses for an even $n$ too, yielding of course the output in item 1 of Theorem~\ref{omnibucito}. In either case, there are no (additive or multiplicative) extension problems in the resulting $E_\infty$ term, and this yields on the nose the multiplicative structure in item 2 of Theorem~\ref{omnibucito}. (The lack of extensions problems is forced by sparseness if $n>2$, but an additional argument using Brown representability is needed when $n=2$, which is our novel contribution in this step, as the case $n=2$ is not dealt with in~\cite[Theorem~5.2]{Xico}.) All this is reviewed in more detail in Section~\ref{secdediferenciales}.

\smallskip
Our fully novel work starts in Section~\ref{section3}, where we take the first steps in the determination of the $E_2$-terms of the Serre spectral sequences of the classifying maps of the two coverings $\rho_{n,k}$ and $\rho'_{n,k}$. Indeed, we describe the system of local coefficients in those spectral sequences, that is, the action of $(\mathbb{Z}_2)^k$ on the cohomologies of $\Conf_{\mathbb{Z}_2}(S^n,k)$ and $\Conf_{\mathbb{Z}_2}(S^n-\{\pm\star\},k)$. As explained in detail in Section~\ref{section3}, a very  important point in our approach arises from the observation that the cohomology of the latter space admits in fact an action of $(\mathbb{Z}_2)^{k+1}$. The extended action is better explained (after shifting the index $k$) by recalling that $\Conf_{\mathbb{Z}_2}(S^n-\{\pm\star\},k-1)$ is the fiber in~(\ref{laqueeslafibra}).

\smallskip
The explicit description of the $E_2$-term of the spectral sequence for $\rho_{n,k}$ is currently unknown (and might turn out to be a hard task to accomplish) in the general case, but its analysis is greatly simplified in this paper by assuming in Theorems~\ref{thminvariantesimpar} and~\ref{isoinvariantespar} that 2 is invertible. Indeed, since the cardinality of $(\mathbb{Z}_2)^k$---a power of 2---kills the positive dimensional cohomology of $(\mathbb{Z}_2)^k$ (with any coefficients), the invertibility of 2 implies that the spectral sequence for $\rho_{n,k}$ is concentrated on the ``fiber'' axis. The collapsing of the spectral sequence then comes for free, yielding that the cohomology ring of the target space of $\rho_{n,k}$ can be described as the subring of $(\mathbb{Z}_2)^k$-invariants. The explicit description of this subring of invariants is the central task in Section~\ref{secndinvtes}, where the hypothesis that 2 is invertible is used once again in order to simplify calculations. Indeed, in earlier versions of this work, the computation of $(\mathbb{Z}_2)^k$-invariants was worked out by brute force in several pages of hard calculations. But eventually we realized that the task can be drastically simplified with a suitable change of basis---which only makes sense when 2 is invertible (see~(\ref{2esinvertible1}) and~(\ref{2esinvertible2})).

\smallskip
Lastly, Section~\ref{section4} deals with calculations analogous to those in Section~\ref{secndinvtes}, but now in the ``punctured'' case, and Section~\ref{sectioncatTC} deals with the applications to the LScat and the HTC's.

\section{The cohomology of $\Conf_{\mathbb{Z}_2}(S^n-\{\pm\star\},k)$}\label{section2}

Readers wishing to use~\cite[Theorem~1.1]{Xico} need to be aware that there is a subtle (but critical) typo in the given presentation of the cohomology ring of $\Conf_{\mathbb{Z}_2}(S^n-\{\pm\star\},k)$ (see Remark~\ref{xicoerror}). Realizing the problem is not an easy task, as the descriptive arguments provided in~\cite{Xico} avoid explicit computational details. To mend the situation, in this section we review the methods in Sections~3 and~4 of~\cite{Xico}, going into great details in order to fix the cohomology presentation.

\smallskip
Let $R$ denote a commutative ring with unit where $2$ is not necessarily invertible. All cohomology rings in this section will be considered with coefficients in $R$ unless otherwise stated. Also, throughout this section $n$ will denote an integer greater than $1$. 

\smallskip
Let us start by recalling that the cohomology ring of $\Conf(\mathbb{R}^n,k)$ was first computed in~\cite{cinco}. The following description follows the notation in~\cite[p.~22]{MR1344842}. For $1\leq j<i\leq k$, consider the maps  $p'_{i,j}:\Conf(\mathbb{R}^n,k) \longrightarrow S^{n-1}$ given by
\[
p'_{i,j}(x_1,\ldots,x_k)=\frac{x_i-x_j}{\lVert x_i-x_j \rVert}.
\]
Let $\iota_{n-1}\in H^{n-1}(S^{n-1};R)$ denote the cohomology fundamental class of $S^{n-1}$, and set $A'_{i,j}={p_{i,j}'}^{*}(\iota_{n-1})$.

\begin{lem}\label{cohensdescripcion}
As an $R$-algebra, $H^*(\Conf(\mathbb{R}^n,k);R)$ is generated by the $(n-1)$-dimensional elements $A'_{i,j}$, for $1\leq j< i\leq k$, subject to the relations  
\begin{equation}\label{relacionesusual}
A'_{r,j}A'_{r,i}=A'_{i,j}(A'_{r,i}-A'_{r,j})
\end{equation} 
for $r>i\geq j$.
\end{lem}

Lemma~\ref{cohensdescripcion} can be checked inductively using the Fadell-Neuwirth fibrations $\Conf(\mathbb{R}^n,k)\to\Conf(\mathbb{R}^n,k-1)$ given by projection onto the first $k-1$ coordinates. The argument will be mimicked in the proof of Theorem~\ref{2.1} below---the central result in this section---in order to deal with $\Conf(S^n-\{\pm\star\},k)$. For this to work, we of course need to know that orbit configuration spaces satisfy:

\begin{lem}[{\cite[Lemma~2.3]{Xico}}]\label{prodisact}
Let $X$ be a manifold with a properly discontinuous action of a finite group $G$. If the orbit space $X/G$ is a manifold and $l<k$, then the projection $\Conf_G(X,k)\rightarrow \Conf_G(X,l)$ onto the first $l$ coordinates is a locally trivial bundle with fiber $\Conf_G(X-Q^G_l,k-l)$. Here $Q^G_l$ denotes the union of $l$ disjoint orbits.
\end{lem}

\begin{rem}\label{renotacion}
It will be convenient to work with the following slight reinterpretation of the orbit configuration spaces $\Conf_{\mathbb{Z}_2}(S^n-\{\pm\star\},k)$. Stereographic projection from $\star$ yields a homoeomorphism $S^n-\{\pm\star\}\cong \mathbb{R}^n-\lbrace 0 \rbrace$. In these terms, the $\mathbb{Z}_2$-antipodal action on the punctured sphere takes the form $\tau(x)=-\frac{x}{\lVert x \rVert^2}$ for $x$ in the punctured Euclidean space $\mathbb{R}^n-\{0\}$, where $\tau\in\mathbb{Z}_2$ stands for the generator (see~\cite[p.~4]{Xico}). This is the action to be considered on the left hand side of the resulting homeomorphism $\Conf_{\mathbb{Z}_2}(\mathbb{R}^n-\lbrace 0\rbrace,k) \cong \Conf_{\mathbb{Z}_2}(S^n-\{\pm\star\},k)$.
\end{rem}

\begin{rem}\label{notacionrecorrida}
As explained in the paragraph containing~(\ref{laqueeslafibra}), a central point in this paper is the determination of the cohomology of $\Conf_{\mathbb{Z}_2}(S^n,k)$ from knowledge of the cohomology of $\Conf_{\mathbb{Z}_2}(\mathbb{R}^n-\{0\},k-1)$. With this in mind, and in order to avoid readjusting notation later in the paper, in this section we set notation to study the orbit configuration space $\Conf_{\mathbb{Z}_2}(\mathbb{R}^n-\{0\},k-1)$---rather than the orbit configuration space in the title of this section.
\end{rem}

With this preparation, the last ingredient before stating the main result in this section is the definition of elements $A_{i,j}\in H^{n-1}(\Conf_{\mathbb{Z}_2}(\mathbb{R}^n-\{0\},k-1);R)$ playing a role analogous to that played by the elements~$A'_{i,j}$. Explicitly,  for $0\leq |j| <i<k$, we set $A_{i,j}=p^*_{i,j}(\iota_{n-1})$ where, as above, $\iota_{n-1}$ is the cohomology fundamental class of $S^{n-1}$, and the maps $p_{i,j}:\Conf_{\mathbb{Z}_2}(\mathbb{R}^n-\lbrace 0 \rbrace,k-1)\rightarrow S^{n-1}$ are given by
\begin{eqnarray}
p_{i,0}(x_1,\ldots,x_{k-1})&=&\frac{x_i}{\lVert x_i\rVert};\nonumber\\
p_{i,j}(x_1,\ldots,x_{k-1})&=&\frac{x_i-x_j}{\lVert x_i-x_j\rVert},\qquad\;\;\mbox{if $j>0$;}\label{laspes}\\
p_{i,-j}(x_1,\ldots,x_{k-1})&=&\frac{x_i-\tau x_j}{\lVert x_i-\tau x_j\rVert}\nonumber,\qquad\mbox{if $j>0$.}
\end{eqnarray}
Lastly, we put 
\begin{equation}\label{agarigoleada}
\mathcal{A}=\lbrace A_{i,j},\,|\,1\leq j<i<k \rbrace\cup \lbrace A_{i,-j} \,|\,1\leq j<i<k \rbrace\cup \lbrace A_{i,0} \,|\,1\leq i<k \rbrace.
\end{equation}

\begin{thm}[{\cite[Theorem~1.1]{Xico}}]\label{2.1}
For $n,k\geq 2$, there is a graded $R$-algebra isomorphism
\begin{equation}\label{descrixico}
H^*(\Conf_{\mathbb{Z}_2}(\mathbb{R}^n-\lbrace 0 \rbrace,k-1);R)\cong  R[\mathcal{A}]/I
\end{equation}
where $I$ denotes the ideal generated by the following relations: 
\begin{enumerate}[(a)] 
\item \label{relaciona}For $0\leq j<i<k$,
\[
A^2_{i,j}=A^2_{i,-j}=0.
\]
\item \label{relacionb}For $1\leq i< r<k$,
\begin{eqnarray*}
A_{r,0}A_{r,i}&=&A_{i,0}(A_{r,i}-A_{r,0}),\\
A_{r,0}A_{r,-i}&=&(-1)^n A_{i,0}(A_{r,-i}-A_{r,0}),\\
A_{r,i}A_{r,-i}&=&(-1)^n A_{i,0}(A_{r,-i}-A_{r,i}).
\end{eqnarray*}
\item \label{relacionc}For $1\leq j<i<r< k$,
\begin{eqnarray*}
A_{r,j}A_{r,i}&=&A_{i,j}(A_{r,i}-A_{r,j}),\\
A_{r,j}A_{r,-i}&=&(-1)^n (A_{j,0}+A_{i,0}-A_{i,-j})(A_{r,-i}-A_{r,j}),\\
A_{r,i}A_{r,-j}&=&(-1)^nA_{i,-j}(A_{r,-j}-A_{r,i}),\\
A_{r,-j}A_{r,-i}&=&(-1)^n (A_{i,0}-A_{i,j}+(-1)^nA_{j,0})(A_{r,-i}-A_{r,-j}).
\end{eqnarray*}
\end{enumerate}
Further,~(\ref{descrixico}) is $R$-free and an additive basis is given by the monomials $A_{i_1,j_1} \cdots A_{i_r,j_r}$ with $i_\ell<i_{\ell'}$ whenever $\ell<\ell'$.
\end{thm}

\begin{rem}\label{xicoerror}
The relations given in \cite[Theorem~1.1]{Xico} for the products of the generators $A_{i,j}$ contain a (small but critical) typo. Namely, the product $A_{r,j}A_{r,-i}$ is not equal to $(-1)^n (A_{j,0}+A_{i,0}-A_{i,j})(A_{r,-i}-A_{r,j})$, as it is claimed there, but rather to the expression in (c) above.
\end{rem}

Before going into full details, we describe the plan of proof for Thereom~\ref{2.1}. Since the case $k=2$ is obvious, we can safely assume $k\geq3$. We first work inductively with the fibration 
\begin{equation}
\label{fib2}
(\mathbb{R}^n-\lbrace 0 \rbrace)-Q^{\mathbb{Z}_2}_{k-2}  \rightarrow \Conf_{\mathbb{Z}_2}(\mathbb{R}^n-\lbrace 0\rbrace,k-1) \rightarrow \Conf_{\mathbb{Z}_2}(\mathbb{R}^n-\lbrace 0\rbrace,k-2)
\end{equation}
given by projection onto the first $k-2$ coordinates (Lemma~\ref{prodisact}). We show that the corresponding Serre spectral sequence has a trivial system of coefficients and collapses from its $E_2$ term (Lemma~\ref{locsyscoe}). This yields the additive structure in~(\ref{descrixico}). The determination of the multiplicative structure requires dealing, in a term-by-term basis, with each relation defining $I$. The method is the same in all cases, and we only give full details in a representative situation, namely the one correcting the typo in~\cite[Theorem~1.1]{Xico} observed in Remark~\ref{xicoerror}. The general idea is to reduce the problem to the case $k=4$ by using a naturality argument (based on Lemma~\ref{estnatarg}); a suitable map $\Conf_{\mathbb{Z}_2}(\mathbb{R}^n-\{0\},3)\to\Conf(\mathbb{R}^n,3)$ (defined in Lemma~\ref{mapalphaparti}) then allows us to ``import'' the desired relation from the known multiplicative structure in Lemma~\ref{cohensdescripcion}.

Completing the proof details for the strategy just sketched requires an additional key ingredient. Namely, we consider maps $f_{i,j}:S^{n-1}\longrightarrow \Conf_{\mathbb{Z}_2}(\mathbb{R}^n-\lbrace 0 \rbrace,3)$, $0\leq |j|<i<4$, defined by 
\begin{align}
f_{1,0}(x)&= (x,2e,3e), & f_{2,0}(x) &= (e,\frac{x}{2},3e),& f_{2,1}(x)&= (e,e+\frac{x}{2},3e),	\nonumber	\\[.5em]
f_{2,-1}(x) &= (e,-e+\frac{x}{2},3e),&
f_{3,0}(x)&=(e,\frac{3}{2}e,\frac{x}{2}),&
f_{3,1}(x)&=(e,\frac{3}{2}e,e+\frac{x}{3}),\label{lasefes}\\[.5em]
f_{3,2}(x)&=(e,\frac{3}{2}e,\frac{3}{2}e+\frac{x}{4}),&
f_{3,-1}(x)&=(e,\frac{3}{2}e,-e+\frac{x}{4}),&
f_{3,-2}(x)&=(e,\frac{3}{2}e,-\frac{2}{3}e+\frac{x}{4}),\nonumber
\end{align} 
where $e=(1,0\ldots,0)\in\mathbb{R}^n$. The straightforward verification of the following result is left as an exercise for the reader.
\begin{lem}\label{elauxi1}
Let $0\leq |j|<i<4$ and $0\leq |s|<r<4$. The composite $p_{r,s}f_{i,j}$ is nullhomotopic unless $r=i$ and $s=j$ in which case $p_{r,s}f_{r,s}\simeq\mbox{identity}$.
\end{lem}
With all the pieces in place, we finally start with the detailed arguments leading to the planned proof of Theorem~\ref{2.1}. To begin with, note that the fiber in~(\ref{fib2}) has the homotopy type of a wedge of $2k-3$ copies of the $(n-1)$-dimensional sphere:
\begin{equation}\label{hotopytype}
(\mathbb{R}^n-\lbrace 0 \rbrace)-Q^{\mathbb{Z}_2}_{k-2}\simeq\bigvee_{2k-3}S^{n-1}.
\end{equation}
\begin{lem}\label{natargyque4}
The classes $A_{k-1,j}\in H^{n-1}(\Conf_{\mathbb{Z}_2}(\mathbb{R}^n-\{0\},k-1);R)$, $1-k<j<k-1$, restrict under the fiber inclusion of~(\ref{fib2}) to the standard basis of $H^{n-1}(\bigvee_{2k-3}S^{n-1})$.
\end{lem}
\begin{proof}
Think of the indicated fiber as $\left\{q_1\right\}\times\cdots\times\{q_{k-2}\}\times\left((\mathbb{R}^n-\{0\})-Q_{k-2}^{\mathbb{Z}_2}\right)$, i.e.~as lying over a fixed point $\left(q_1,\ldots,q_{k-2}\right)\in\Conf_{\mathbb{Z}_2}(\mathbb{R}^n-\{0\},k-2)$ where $Q_{k-2}^{\mathbb{Z}_2}=\{q_1,\tau\cdot q_1,\ldots,q_{k-2},\tau\cdot q_{k-2}\}$. In these terms, the inclusions of the wedge-summand spheres appearing in the homotopy equivalence~(\ref{hotopytype}) can be indexed by points in $\mathcal{Q}:=\{0\}\cup Q_{k-2}^{\mathbb{Z}_2}$. Namely, for $q\in\mathcal{Q}$, the inclusion $\iota_q$ of the $q$-th wedge summand sphere into the fiber is realized up to homotopy by the obvious degree-1 map which sends $S^n$ to $\left\{q_1\right\}\times\cdots\times\{q_{k-2}\}\times\Sigma_q$, where $\Sigma_q$ is the sphere of radius $r$ centered at $q$, and $r$ is small enough so that no point in $\mathcal{Q}-\{q\}$ lies in the convex hull determined by $\Sigma_q$. If we set $q_0=0$ and $q_j=\tau\cdot q_{-j}$ for $2-k\leq j<0$, so that $\mathcal{Q}=\{q_j\}_{2-k\leq j\leq k-2}$, then we see directly from the definition of the maps $p_{k-1,j}$ that the composite $p_{k-1,j'}\circ i_{q_j}$ is nullhomotopic if $j\neq j'$, but has degree $\pm1$ if $j=j'$. Consequently, $A_{k-1,j}$ restricts under the fiber inclusion to the generator corresponding to the $q_j$-th sphere in~(\ref{hotopytype}). 
\end{proof}

\begin{lem}\label{locsyscoe}
The three spaces in~(\ref{fib2}) are $(n-2)$-connected and have cohomology concentrated in dimensions divisible by $n-1$. Furthermore, the system of local coefficients in~(\ref{fib2}) is trivial, and the corresponding Serre spectral sequence collapses.
\end{lem}
\begin{proof}
The base space in~(\ref{fib2}) has the homotopy type of $S^{n-1}$ when $k=3$. Therefore, the assertion about the connectivity of the spaces follows from an inductive argument (on $k$) using~(\ref{hotopytype}) and the long exact sequence in homotopy groups of~(\ref{fib2}).

Since the cohomology assertion is obvious for $n=2$, it suffices to prove it for $n>2$. In such a case, all spaces in~(\ref{fib2}) are simply connected, so that the local system of coefficients is forced to be trivial. Assume inductively that, just as for the fiber, the cohomology of the base space in~(\ref{fib2}) is concentrated in dimensions divisible by $n-1$. Then the spectral sequence under consideration collapses by sparseness, and the cohomology of the total space is forced to be concentrated in dimensions divisible by $n-1$ too.

The previous argument takes care of the assertions about the behavior of the spectral sequence when $n>2$ (trivial local coefficients and collapsibility). In any case, the assertion for $n\geq2$ follows from Lemma~\ref{natargyque4} and the fact that, if $F\to E\to B$ is a fibration for which $H^*(E;R)\to H^*(F;R)$ is surjective, then the corresponding Serre spectral sequence collapses and has a trivial system of local coefficients~(see~\cite[Theorem~14.1]{MR0221507} or~\cite[Theorem~4.4]{MR1122592}).
\end{proof}

\begin{rem}\label{otroerrordeFZ}
The final task in the proof above, namely the surjectivity of in cohomology of the fiber inclusion in~(\ref{fib2}) when $n=2$, is addressed in~\cite[Remark~10]{FZ2}. However, that argument is flawed as it uses~\cite[Lemma~7(iv)]{FZ2} which, as explained in Remark~\ref{3.6}, leads to inconsistencies.
\end{rem}
\begin{proof}[Proof of the additive assertion in Theorem~\ref{2.1}]
We need to assemble the cohomological information coming from the spectral sequences for each of the (vertical) fibrations in the tower
\[\xymatrix{
\Conf_{\mathbb{Z}_2}(\mathbb{R}^n-\{0\},k-1) \ar[d] & 
\bigvee_{2k-3}S^{n-1} \ar[l]\\
\Conf_{\mathbb{Z}_2}(\mathbb{R}^n-\{0\},k-2) \ar[d] & 
\bigvee_{2k-5}S^{n-1} \ar[l]\\
\vdots\ar[d]\\
\Conf_{\mathbb{Z}_2}(\mathbb{R}^n-\{0\},2) \ar[d] & 
\bigvee_{3}S^{n-1} \ar[l]\\
\mathbb{R}^n-\{0\}\cong S^{n-1}.}
\]
For the fibration in the bottom of the tower, we know that $E_2=E_\infty$, which evidently is a (doubly graded) free $R$-module. In particular, all possible extensions are trivial when recovering $H^*(\Conf_{\mathbb{Z}_2}(\mathbb{R}^n-\{0\},2);R)$ from its ($R$-free) graded associated $E_\infty$. Consequently, $H^*(\Conf_{\mathbb{Z}_2}(\mathbb{R}^n-\{0\},2);R)$ is $R$-free too. Working our way upwards in the tower, the argument above repeats,  and we deduce that, just as in the case of the bottom fibration, 
none of the spectral sequences for the fibrations in the tower has non-trivial extensions, and 
$H^*(\Conf_{\mathbb{Z}_2}(\mathbb{R}^n-\{0\},k-1))$ is $R$-free.
In fact, we get an $R$-module isomorphism
\begin{equation}\label{isomodulos}
H^*(\Conf_{\mathbb{Z}_2}(\mathbb{R}^n-\lbrace 0 \rbrace, k-1);R)\cong M_1 \otimes M_2 \otimes \cdots \otimes M_{k-1}
\end{equation}
where $M_i$ is the $R$-free module generated by the zero dimensional class $1$ and by the $(n-1)$-dimensional spherical classes $\lbrace A_{i,0} \rbrace \cup \lbrace A_{i,j},A_{i,-j} \rbrace_{1\leq j < i} $. This gives the assertion about the additive structure in Theorem~\ref{2.1}.
\end{proof}
The remaining of the section deals with the ideas leading to the multiplicative structure in Theorem~\ref{2.1}, thus, we assume $k\geq4$ (the multiplicative structure is evidently trivial for $k=3$). The tensor products in~(\ref{isomodulos}) reflects part of the structure, but this does not account for products of elements in a given $M_i$. Namely, although a product $A_{i,j_1}A_{i,j_2}$ is trivial at the level of the spectral sequence, it is usually non-zero at the cohomology level, and we need to give its expression in terms of the additive basis.  As already noted, we will determine in full detail such expression only in the representative case of the relation fixing the typo in~\cite[Theorem~1.1]{Xico}. Namely:
\begin{prop}\label{casodecorreccion}
Let $0<j<i<r<k$. Then $A_{r,j}A_{r,-i}=(-1)^n (A_{j,0}+A_{i,0}-A_{i,-j})(A_{r,-i}-A_{r,j})$. 
\end{prop}
Proposition~\ref{casodecorreccion} follows immediately from Lemmas~\ref{estnatarg} and~\ref{mapalphaparti} below using~(\ref{relacionesusual}) with $(r,i,j)=(3,2,1)$.
\begin{lem}\label{estnatarg}
Let $0<j<i<r<k$. The map $\pi_{r,i,j}:\Conf_{\mathbb{Z}_2}(\mathbb{R}^n - 0,k-1)\to \Conf_{\mathbb{Z}_2}(\mathbb{R}^n - 0,3)$ given by $\pi_{r,i,j}(x_1,\ldots,x_{k-1})=(x_j,x_i,x_r)$ satisfies
\[\begin{array}{rclrclrcl}
\pi_{r,i,j}^*(A_{1,0})&=&A_{j,0},\;\;\;\;\;\;\;&
\pi_{r,i,j}^*(A_{2,0})&=&A_{i,0},&
\pi_{r,i,j}^*(A_{2,1})&=&A_{i,j},\\
\pi_{r,i,j}^*(A_{2,-1})&=&A_{i,-j},&
\pi_{r,i,j}^*(A_{3,0})&=&A_{r,0},&
\pi_{r,i,j}^*(A_{3,1})&=&A_{r,j},\\
\pi_{r,i,j}^*(A_{3,2})&=&A_{r,i},&
\pi_{r,i,j}^*(A_{3,-1})&=&A_{r,-j},\;\;\;\;&
\pi_{r,i,j}^*(A_{3,-2})&=&A_{r,-i}.\\
\end{array}\]
\end{lem}
\begin{lem}\label{mapalphaparti}
The map $\alpha:\Conf_{\mathbb{Z}_2}(\mathbb{R}^n-\lbrace 0 \rbrace,3)\longrightarrow \Conf(\mathbb{R}^n,3)$ given by $\alpha(x,y,z)=(x,\tau y ,z)$ satisfies $\alpha^*(A'_{3,2})=A_{3,-2}$,
$\;\;\alpha^*(A'_{3,1})=A_{3,1},\;$ and
$\;\;\alpha^*(A'_{2,1})=(-1)^{n}(A_{1,0}+A_{2,0}-A_{2,-1})$.
\end{lem}
All equalities in Lemmas~\ref{estnatarg} and~\ref{mapalphaparti} follow from the definitions, except for the formula $\alpha^*(A'_{2,1})=(-1)^{n}(A_{1,0}+A_{2,0}-A_{2,-1})$ which follows from Lemma~\ref{elauxi1} and:
\begin{lem}\label{descgrados}
For $0\leq|j|<i<4$, let $d_{i,j}$ stand for the degree of the composition $p'_{2,1}\alpha f_{i,j}:S^{n-1}\to S^{n-1}$. Then $d_{i,j}=0$ except for $d_{1,0}=d_{2,0}=(-1)^n$ and $d_{2,-1}=(-1)^{n+1}$.
\end{lem}
\begin{proof}A straightforward calculation gives
\begin{align}
p'_{2,1}\alpha f_{1,0}(x)&&&=&&N(\frac{-e}{2}-x)&&=&&-N(\frac{e}{2}+x),\label{ecp21alfaf10}\\[.5em]
p'_{2,1}\alpha f_{2,0}(x)&&&=&&N(-2x-e)&&=&&-N(2x+e),\label{ecp21alfaf20}\\[.5em]
p'_{2,1}\alpha f_{2,1}(x)&&&=&&N(\tau(e+\frac{x}{2})-e),\label{ecp21alfaf21}\\[.5em]
p'_{2,1}\alpha f_{2,-1}(x)&&&=&&N(\tau(-e+\frac{x}{2})-e)&&=&&-N(-e+\frac{x}{2}+\lVert  -e+\frac{x}{2}\rVert^2 e),\label{ecp21alfaf2-1}\\[.5em]
p'_{2,1}\alpha f_{3,0}(x)&&&=&&p'_{2,1}\alpha f_{3,1}(x)&&=&&p'_{2,1}\alpha f_{3,2}(x)\label{ecp21alfaf30}\\[.5em]
&&&=&&p'_{2,1}\alpha f_{3,-1}(x)&&=&&p'_{2,1}\alpha f_{3,-2}(x),\nonumber
\end{align}
where $N:\mathbb{R}^n-\lbrace 0 \rbrace\longrightarrow S^{n-1}$ is the normalization map and, as in~(\ref{lasefes}), $e=(1,0,\ldots,0)\in\mathbb{R}^n$. The maps in (\ref{ecp21alfaf10}) and (\ref{ecp21alfaf20}) are obviously homotopic to the antipodal map so that $d_{1,0}=d_{2,0}=(-1)^n$. On the other hand, $d_{2,1}=0$ since the map in (\ref{ecp21alfaf21}) is homotopic to the constant map. Actually, $p'_{2,1}\alpha f_{2,1}$ is not a surjective map: $e$ is not in the image because $e$ is not enclosed by the image of $\tau(e+\frac{x}{2})$. All the maps in~(\ref{ecp21alfaf30}) are obviously constant, so the corresponding degrees $d_{3,j}$ are trivial. Lastly and most interesting is the identification of the degree of the map in (\ref{ecp21alfaf2-1}), the bulk of this proof.
\newline\indent
Let $F:\mathbb{R}\times \mathbb{R}^n \longrightarrow \mathbb{R}^n$ be the map given by $F(t,(t_1,t_2,\ldots,t_n))=(tt_1,t_2,\ldots,t_n)$. Note that $F(1,x)=x$ and $F(-1,x)$ is $x$ reflected across the hyperplane $t_1=0$. As a map $S^{n-1}\longrightarrow S^{n-1}$,~(\ref{ecp21alfaf2-1}) becomes
\begin{eqnarray*}
-N\left(-e+\frac{x}{2}+\left\lVert  -e+\frac{x}{2}\right\rVert^2 e\right)&=&-N\left(F\left(1,\frac{x}{2}\right)+\left(-1+\left\lVert  -e+\frac{x}{2}\right\rVert ^2\right)e\right)\\ &\simeq&-N\left(F\left(-1,\frac{x}{2}\right)+\left(-1+\left\lVert  -e+\frac{x}{2}\right\rVert^2\right) e\right).
\end{eqnarray*}
The homotopy is given by $$-N\left(F\left(t,\frac{x}{2}\right)+\left(-1+\left\lVert  -e+\frac{x}{2}\right\rVert ^2\right)e\right),\quad \mbox{for $t\in[-1,1]$,}$$ and we next check it is well defined: Suppose there exist $t\in [-1,1]$ and $x=(t_1,\ldots,t_n )\in S^{n-1}$ such that $F(t,\frac{x}{2})+(-1+\lVert  -e+\frac{x}{2}\rVert^2) e=0$. Then $F(t,\frac{x}{2})=(1-\lVert -e+\frac{x}{2} \rVert^2)e$ and so we have $\frac{tt_1}{2}=1-\lVert -e+\frac{x}{2} \rVert^2$ and $t_i=0$ for $i>1$. The latter condition, in turn, implies $t_1=\pm 1$.
\begin{itemize}
\item[] \;\; If $t_1=1$, then $\frac{t}{2}=1-\lVert -e+\frac{e}{2}  \rVert^2 =1-\lVert -\frac{e}{2} \rVert^2=\frac{3}{4}$, so $t=\frac{3}{2}>1$.
\item[] \;\; If $t_1=-1$, then $-\frac{t}{2}=1-\lVert -e-\frac{e}{2}  \rVert^2 =1-\lVert -\frac{3e}{2} \rVert^2=-\frac{5}{4}$, so $t=\frac{5}{2}>1$.
\end{itemize}
Both assumptions lead to a contradiction, so the homotopy is well defined.
\newline\indent
Next we prove that, as maps $S^{n-1}\to S^{n-1}$,
\[
-N\left(F\left(-1,\frac{x}{2}\right)+\left(-1+\left\lVert  -e+\frac{x}{2}\right\rVert^2\right) e\right)\simeq -N\left(F\left(-1,\frac{x}{2}\right)\right).
\]
This time the homotopy is $$-N\left(F\left(-1,\frac{x}{2}\right)+t\left(-1+\left\lVert  -e+\frac{x}{2}\right\rVert^2\right) e\right),\quad\mbox{ for } t\in[0,1],$$ which is well defined: Suppose $F(-1,\frac{x}{2})+t(-1+\lVert  -e+\frac{x}{2}\rVert^2) e=0$ for some $x=(t_1,\ldots,t_n )\in S^{n-1}$ and some $t\in [0,1]$. Then $F(-1,\frac{x}{2})=t(1-\lVert  -e+\frac{x}{2}\rVert^2) e$ and so we have $-\frac{t_1}{2}=t(1-\lVert  -e+\frac{x}{2}\rVert^2) $ and $t_i=0$ for $i>1$. The latter condition, in turn, implies $t_1=\pm 1$.
\begin{itemize}
\item[ \;\; ] If $t_1=1$, then $-\frac{1}{2}=t(1-\lVert -e+\frac{e}{2}  \rVert^2) =t(1-\lVert -\frac{e}{2} \rVert^2)=t\frac{3}{4}$, so $t=-\frac{2}{3}<0$.
\item[ \;\; ] If $t_1=-1$, then $\frac{1}{2}=t(1-\lVert -e-\frac{e}{2}  \rVert^2) =t(1-\lVert -\frac{3e}{2} \rVert^2)=-t\frac{5}{4}$, so  $t=-\frac{2}{5}<0$.
\end{itemize}
Both assumptions lead to a contradiction, so the homotopy is well defined.
We conclude
\[
p'_{2,1}\alpha f_{2,-1}(x)\simeq -N\left(F\left(-1,\frac{x}{2}\right)\right)=-N\left(F\left(-1,x\right)\right)
\]
which is clearly a map of degree $(-1)^{n+1}$.
\end{proof}

\section{The cohomology of $\Conf_{\mathbb{Z}_2}(S^n,k)$}\label{secdediferenciales}
In this section we review the third named author's work in Sections~4 and~5 of~\cite{Xico} where the cohomology algebra $H^*(\Conf_{\mathbb{Z}_2}(S^n,k);R)$ is computed for $n>2$. There is nothing to correct this time, so we omit proof details. Yet, the global picture is reviewed in enough detail so to be in position of extending the computation for $n\geq2$. This section's strategy is rather straightforward, and has been described in a fairly detailed way in Section~\ref{proofmethod}. We therefore get into business right away.

\smallskip
Xicotencatl's method is to look at the Serre spectral sequence associated to the fibration in~(\ref{laqueeslafibra}) which, in the current notation, takes the form
\begin{equation}
\label{fib1}
\Conf_{\mathbb{Z}_2}(\mathbb{R}^n-\lbrace 0\rbrace,k-1) \rightarrow \Conf_{\mathbb{Z}_2}(S^n,k)\rightarrow S^n.
\end{equation}
Since $n\geq 2$,  $S^n$ is simply connected, so the corresponding system of local coefficients is trivial. Also, $S^n$ has torsion-free cohomology, so that the spectral sequence starts with 
\[
E_2^{p,q}\cong H^p(S^n;H^q(\Conf_{\mathbb{Z}_2}(\mathbb{R}^n-\lbrace 0 \rbrace,k-1);R)) \cong H^p(S^n;R)\otimes H^q(\Conf_{\mathbb{Z}_2}(\mathbb{R}^n-\lbrace 0 \rbrace,k-1);R).
\]
For $n$ odd, a nowhere vanishing vector field on $S^n$ easily yields a section for (\ref{fib1}), consequently the spectral sequence collapses (cf.~\cite[Proposition 13]{FZ2} and~\cite[Proposition~4.1]{Xico}), and we have:
\begin{thm}[{\cite[Proposition 14]{FZ2}},{\cite[Proposition 5.2(a)]{Xico}}]\label{2.2}
For $n\geq2$ odd, there is an $R$-algebra isomorphism $H^*(\Conf_{\mathbb{Z}_2}(S^n,k);R)\cong H^*(S^n;R)\otimes H^*(\Conf_{\mathbb{Z}_2}(\mathbb{R}^n-\lbrace 0 \rbrace,k-1);R)$.
\end{thm}
\begin{rem}\label{sparsenessreference}
Part of the assertion in Theorem~\ref{2.2} is, of course, that the multiplicative structure in $E_\infty=E_2$, which is the tensor product of the multiplicative structures for the base and fiber, gives the multiplicative structure of $H^*(\Conf_{\mathbb{Z}_2}(S^n,k);R)$. Such a fact can be seen by sparseness. Namely, by Lemma~\ref{locsyscoe}, the non-zero groups in the spectral sequence are located in spots indicated by bullets in the picture

\smallskip
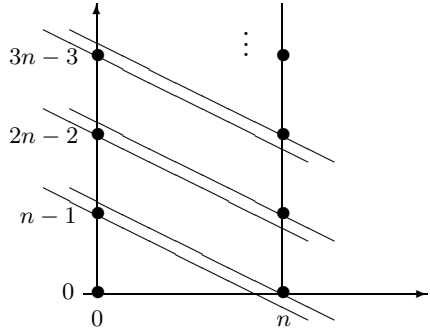
\begin{picture}(10,140)(-180,-15)
\put(0,0){\vector(0,0){110}}
\put(-5,0){\vector(1,0){130}}
\put(70,0){\line(0,0){110}}
\put(-2.4,-2.4){$\bullet$}
\put(67.6,-2.4){$\bullet$}
\put(-2.4,27.6){$\bullet$}
\put(67.6,27.6){$\bullet$}
\put(-2.4,57.6){$\bullet$}
\put(67.6,57.6){$\bullet$}
\put(-2.4,87.6){$\bullet$}
\put(67.6,87.6){$\bullet$}
\put(0,30){\line(2,-1){80}}
\put(0,30){\line(-2,1){20}}
\put(0,60){\line(2,-1){80}}
\put(0,60){\line(-2,1){20}}
\put(0,90){\line(2,-1){80}}
\put(0,90){\line(-2,1){20}}
\put(70,0){\line(2,-1){20}}
\put(70,0){\line(-2,1){80}}
\put(70,30){\line(2,-1){20}}
\put(70,30){\line(-2,1){80}}
\put(70,60){\line(2,-1){20}}
\put(70,60){\line(-2,1){80}}
\put(55,90){$\vdots$}
\put(-2,-12){\scriptsize$0$}
\put(-13,-2){\scriptsize$0$}
\put(68,-12){\scriptsize$n$}
\put(-29,27){\scriptsize$n-1$}
\put(-33,57){\scriptsize$2n-2$}
\put(-33,87){\scriptsize$3n-3$}
\end{picture}

\noindent Here slanted lines indicate families of groups in a fixed total degree (i.e.~slope $-1$). But since $n\geq3$, we see $n-1<n<2n-2<2n-1<3n-3<3n-2<4n-4<\cdots$, so that there is a single non-trivial group in each slanted line. Consequently, a product that is trivial in the $E_\infty$ term, has to be trivial also in the cohomology of the total space.\end{rem}

\smallskip
For $n$ even, the only family of possibly non-trivial differentials
$d^{0,i(n-1)}_n:E^{0,i(n-1)}_n\to E^{n,(i-1)(n-1)}_{n}$ is determined by $d_n(A_{i,j})=2\iota_n$ for all $A_{i,j}\in \mathcal{A}$ (cf.~\cite[Proposition 13]{FZ2} and~\cite[Proposition~4.1]{Xico}). In particular, if the characteristic of $R$ is 2, the conclusion of (and argument for) Theorem \ref{2.2} holds also for any even $n$, yielding the output in item 1 of Theorem~\ref{omnibucito} (but the sparseness argument given in Remark~\ref{sparsenessreference} for the multiplicative structure does not apply when $n=2$; instead we use the argument based on Brown representability given in the proof of Theorem \ref{2.3} below). 

We close the section with a description of the $R$-cohomology algebra of $\Conf_{\mathbb{Z}_2}(S^n,k)$ for $n$ even under the additional hypothesis that the characteristic of $R$ is either zero (e.g.~$R=\mathbb{Z}$ or $R=\mathbb{Q}$) or an odd integer (e.g.~$R=\mathbb{Z}_t$, odd $t$), so that the map $R\to R$ given by multiplication by 2 is injective. Both hypothesis, on $n$ and on $R$, will be in force throughout the rest of this section.

It will be convenient to make a change of basis by defining $B_{i,j}=A_{i,j}-A_{1,0}$, for $|j|<i<k$, and
\begin{equation}\label{genersbgari}
\mathcal{B}=\lbrace B_{i,j}\,|\,|j|<i<k \textnormal{ and }1<i \rbrace.
\end{equation}
A straightforward computation shows that a product of two given elements in 
$\mathcal{B}$ satisfies the exact same relation holding for the product of the corresponding two elements in $\mathcal{A}$ (keeping in mind that, by definition, $B_{1,0}=0$). For instance, the case $i=1$ in the first relation in item (b) of Theorem~\ref{2.1} becomes $B_{r,0}B_{r,1}=0$ for $1<r<k$. Let us denote by $J$ the resulting set of relations among the $B_{i,j}$'s. In particular, it is clear that, starting with the basis described in Theorem~\ref{2.1}, namely the monomials $A_{i_1,j_1}\cdots A_{i_r,j_r}$ which are ordered in the sense that $i_\ell<i_{\ell'}$ whenever $\ell<\ell'$, a new basis for $H^*(\Conf_{\mathbb{Z}_2}(\mathbb{R}^n-\lbrace 0 \rbrace,k-1);R)$ is obtained by replacing each factor $A_{i,j}$ with $i>1$ by the corresponding $B_{i,j}$ (but factors $A_{1,0}$ remain unchanged). There result two types of new basis elements depending on whether or not $A_{1,0}$ appears in the monomial. It is also clear from the hypothesis on the characteristic of $R$, and from the fact that the differential $d_n$ sends every $A_{i,j}$ to $2\iota_n$ that, in the new basis, monomials not including $A_{1,0}$ as a factor are permanent cycles in the spectral squence, whereas the rest of the new basis elements inject under $d_n$ onto $2\iota_n\cdot\overline{\mathcal{B}}$ where $\overline{\mathcal{B}}$ is multiplicatively generated by $\mathcal{B}$. This situation is best organized as follows: Let $\mathbb{K}=\mbox{ker } d_n^{0,n-1}$, the (free) $R$-module with basis $\mathcal{B}$, and let $\mathbb{K}^j$ denote the $R$-module generated by products of $j$ factors in $\mathbb{K}$, where $\mathbb{K}^0$ and $\mathbb{K}^{-1}$ are set to be  $R$ and $0$ respectively. Then a basis for $\mathbb{K}^j$ is given by the elements of degree $j(n-1)$ in the above modified basis for $H^*(\Conf_{\mathbb{Z}_2}(\mathbb{R}^n -\lbrace 0 \rbrace ,k-1);R)$ which do not contain the factor $A_{1,0}$, i.e.~by the ordered monomials in the $B$'s with $j$ factors. Note in particular that $\mathbb{K}^{k-1}=0$. 

It is then clear that the only non-trivial terms in the $(n+1)$-stage of the spectral sequence are given by
\[
\begin{array}{lll}
E^{0,j(n-1)}_{n+1}=\mathbb{K}^j,&&\textnormal{for }0\leq j \leq k-2;\\
E^{n,j(n-1)}_{n+1}=\iota_n A_{1,0}\mathbb{K}^{j-1}\oplus (\iota_n \mathbb{K}^{j})_2,&&\textnormal{for }0\leq j\leq k-1;
\end{array}
\]
where $(-)_2$ denotes the mod 2 reduction of the given module (that is, tensoring with $\mathbb{Z}_2$). The spectral sequence obviously collapses from this point on and, since the groups $E_\infty^{0,q}$ are $R$-free, there are no extension problems when assembling the cohomology of $\Conf_{\mathbb{Z}_2}$ out of $E_\infty$. Lastly, just as in Remark~\ref{sparsenessreference}, sparseness implies that, for $n>2$, the multiplicative structure of the cohomology of the total space agrees with that in the $E_\infty$-term of the spectral sequence. We thus get the $n>2$ case of:
\begin{thm}\label{2.3}
Assume that the characteristic of $R$ is either zero or an odd integer. For even $n\geq 2$ there is an isomorphism of graded $R$-algebras $$H^*(\Conf_{\mathbb{Z}_2}(S^n,k);R)\cong R[\mathcal{B}]/J\otimes\Lambda(\lambda,\omega)/(2\lambda,\lambda\omega)$$ where $\lambda$ and $\omega$ are represented in the spectral sequence by $\iota_n$ and $\iota_nA_{1,0}$ respectively (thus the degrees of $\lambda$ and $\omega$ are $n$ and $2n-1$ respectively).
\end{thm}
\begin{proof}
It only remains to argue the assertion about the multiplicative structure when $n=2$. (The issue is mentioned without an explanation by Feichtner and Ziegler on the first half of page 100 in \cite{FZ2}.) So, assuming $n=2$, we have to show that the square of any $B_{i,j}$ is still zero {\it as a class in the cohomology of the total space}. In other words, we have to argue the impossibility that
\begin{equation}\label{parlacontrtn}
\mbox{some element $B_{i,j}^2$ is represented in the spectral sequence by the (non-trivial!) class of $\iota_2$.}
\end{equation}
(Note that such a possibility cannot be overruled just by anticommutativity: Since each $B_{i,j}$ is a 1-dimensional class, the equality $2B_{i,j}^2=0$ holds even in the cohomology of the total space. But of course we know $2\iota_2=0$ in the cohomology of the total space.) The fact that~(\ref{parlacontrtn}) does not hold comes from Brown representability when the coefficients are $\mathbb{Z}$. That is, 1-dimensional cohomology classes with coefficients in $\mathbb{Z}$ are spherical, and so their squares are forced to be trivial. For other coefficients $R$ the assertion holds since the definition of the classes $B_{i,j}$ is natural with respect to the canonical ring morphism $\mathbb{Z}\to R$.
\end{proof}

Note that $R[\mathcal{B}]/J=\bigoplus_{0\leq j\leq k-2}\mathbb{K}^j$, a basis of which has already been described. In the $E_\infty$ term of the spectral sequence, this $R$-subalgebra corresponds to the left hand side tower supported by 1. Besides, two additional ``copies'' of this tower show up: one copy (tensored with $\mathbb{Z}_2$) is supported by $\lambda$, and a second copy  (shifted one level up) is supported by $\omega$, as shown in Figure~\ref{fig1}. 

\smallskip\begin{figure}[h]
\centering
\begin{tikzpicture}[scale=1,font=\scriptsize,line width=1pt,>=stealth]
\draw [very thin,->,black] (0,0) -- (0,5);
\draw [very thin,->] (0,0) -- (5,0);
\draw[very thick] (0,0) -- (0,4);
\draw[very thick] (1,0) -- (1,4);
\draw[very thick] (1.2,1) -- (1.2,5);
\draw (0,1) node[left] {$n-1$};
\draw (0,2) node[left] {$2(n-1)$};
\draw (0,4) node[left] {$(k-2)(n-1)$};
\draw (0,5) node[left] {$(k-1)(n-1)$};
\draw (-.1,3) node[left] {$\vdots$};
\draw (.9,3) node[left] {$\vdots$};
\draw (1.3,3) node[right] {$\vdots$};
\draw[black] (0,-.1) node[below] {$1$};
\draw[black] (6,5);
\draw (1,-.1) node[below] {$\lambda$};
\draw (1.3,1) node[right] {$\omega$};
\draw[fill] (0,0) circle (2pt) ;
\draw[fill] (0,1) circle (2pt) ;
\draw[fill] (0,2) circle (2pt) ;
\draw[fill] (0,4) circle (2pt) ;
\draw[fill] (1,0) circle (2pt) ;
\draw[fill] (1,1) circle (2pt) ;
\draw[fill] (1,2) circle (2pt) ;
\draw[fill] (1,4) circle (2pt) ;
\draw[fill] (1.2,1) circle (2pt) ;
\draw[fill] (1.2,2) circle (2pt) ;
\draw[fill] (1.2,4) circle (2pt) ;
\draw[fill] (1.2,5) circle (2pt) ;
\end{tikzpicture}
\caption{$H^*(\Conf_{\mathbb{Z}_2}(S^n,k);R)$ for even $n$.}
\label{fig1}
\end{figure}
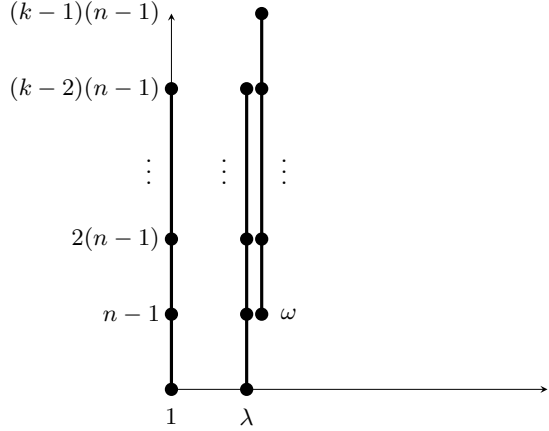

\begin{rem}\label{remark2.4}
The additive version of Theorem \ref{2.3} is obtained in \cite[Theorem 5.2, items (b) and (c)]{Xico} assuming implicitly $n>2$. On the other hand, for $n>2$, the multiplicative relations among generators in Theorems \ref{2.1} and \ref{2.3} (indirectly) correct those found in \cite[Propositions~11 and~16, and Theorem~17]{FZ2}. In fact, the multiplicative relations described by Feichtner and Ziegler in \cite[Proposition 11]{FZ2} for their generators in the cohomology of $\Conf_{\mathbb{Z}_2}(\mathbb{R}^n-\{0\},k-1))$ lead to inconsistencies. We illustrate the problem using Feichtner-Ziegler's notation, which the reader is assumed to be familiar with. (In particular, the notation for the fiber in (\ref{fib1}) will momentarily change to $F_{\langle\varphi\rangle}(\mathbb{R}^k\setminus\{0\},n)$). Take $1\leq i < j\leq n$, and let $k$ be odd (so that the generators $c_i,  c^+_{i,j},  c^-_{i,j}$ are even dimensional and, therefore, commute without introducing signs). Then Lemma~7 and Proposition~11 in \cite{FZ2} imply $$c^-_{i,j}c^+_{i,j}+c_i(c^+_{i,j}+c^-_{i,j})=0=A_i(0)=A_i\left(c^-_{i,j}c^+_{i,j}+c_i(c^+_{i,j}+c^-_{i,j})\right)=c^-_{i,j}c^+_{i,j}-c_i(c^+_{i,j}+c^-_{i,j}).$$
This yields $c_ic^+_{i,j}+c_ic^-_{i,j}=0$, if we work with integral coefficients. However the latter relation contradicts the second item in Proposition~8 of~\cite{FZ2}. 
\end{rem}

\section{($\mathbb{Z}_2)^k$-action}\label{section3}
As in previous sections, $R$ stands for a commutative ring with unit. Let $(\mathbb{Z}_2)^k = \langle \epsilon_1,\epsilon_2,\ldots, \epsilon_k \rangle$ where $\epsilon_i$ is the generator in the $i$-th coordinate, and recall the coordinate-wise $(\mathbb{Z}_2)^k$-action~(\ref{coordwisetion}) on $\Conf_{\mathbb{Z}_2}(S^n,k)$. By abuse of notation we also use the notations
$$
\epsilon_i:\Conf_{\mathbb{Z}_2}(S^n,k)\to\Conf_{\mathbb{Z}_2}(S^n,k)\;\;\;\;\; \mbox{ and }\;\;\;\;\;\epsilon_i:H^*(\Conf_{\mathbb{Z}_2}(S^n,k);R)\to H^*(\Conf_{\mathbb{Z}_2}(S^n,k);R)
$$
for the corresponding induced maps. A formula for the resulting $(\mathbb{Z}_2)^k$-action on the cohomology ring $H^*(\Conf_{\mathbb{Z}_2}(S^n,k);R)$ was stated for $k\leq 3$ in~\cite[Tables~1 and~2]{Xico} with most details omitted. Here we generalize Xicot\'encatl's result for all~$k$, providing full details, and correcting his description for $k=3$. The task will be attained by working, once again, with the Serre spectral sequence of ($\ref{fib1}$). In short, we determine the action of each $\epsilon_i$ on the cohomology of the total space of~(\ref{fib1}) by first understanding its action on the cohomology of the base and the fiber spaces.

The fact that $(\mathbb{Z}_2)^k$ acts on the cohomology of the base and fiber of~(\ref{fib1}) (and, for that matter, on the corresponding spectral sequence) deserves some explanation. The situation for $\epsilon_i$ with $i>1$ is elementary as, then, $\epsilon_i\colon\Conf_{\mathbb{Z}_2}(S^n,k)\to\Conf_{\mathbb{Z}_2}(S^n,k)$ lies over the identity on $S^n$ and so preserves all the fibers in~(\ref{fib1}). By abuse of notation we also write $\epsilon_i$ for the identity on $S^n$, for any of the restricted maps to fibers, as well as for any of the corresponding induced maps in cohomology. In fact, by functoriality, $\epsilon_i\colon\Conf_{\mathbb{Z}_2}(S^n,k)\to\Conf_{\mathbb{Z}_2}(S^n,k)$ determines a $\mathbb{Z}_2$-action, also referred to as $\epsilon_i$, on the whole spectral sequence of~(\ref{fib1}). For instance, as observed above, this spectral sequence $\mathbb{Z}_2$-action is the identity on classes coming from the base space. The important fact to note is that, again by functoriality, the $\epsilon_i$-action on the spectral sequence converges to the $\epsilon_i$-action on the cohomology of the total space.

The above simple situation does not hold for $\epsilon_1$, and we next review the way this issue is dealt with in \cite[Section~6]{Xico}. Start by noticing that $\epsilon_1$ covers the antipodal map. Consider the rotation 
\[R=\left( \begin{matrix}
I_{n-1}&0\\
0&-I_2
\end{matrix}\right)  \in SO(n+1) 
\]
that interchanges the north and south poles $N=(0,\ldots,0,1)$ and $S=(0,\ldots,0,-1)\in S^n$. The restriction of $R$ to $S^n$ is $\mathbb{Z}_2$-equivariant and it is $\mathbb{Z}_2$-equivariantly isotopic to the identity. Therefore it induces a map $R^{\times k}:\Conf_{\mathbb{Z}_2}(S^n,k)\longrightarrow\Conf_{\mathbb{Z}_2}(S^n,k)$ homotopic to the identity fitting in the commutative diagram
\[
\xymatrix{
\Conf_{\mathbb{Z}_2}(S^n,k) \ar[d]_{\pi_{1}}\ar[r]^{R^{\times k}\circ \epsilon_1} & \Conf_{\mathbb{Z}_2}(S^n,k)\ar[d]^{\pi_{1}}\\
S^n\ar[r]^{-R}&S^n.
}
\]
Since $-R$ fixes the north pole $N$, $R^{\times k}\circ\epsilon_1$ restricts to a self map $\epsilon'_1$ on the corresponding fiber. Then we are in a situation similar to the one in the previous paragraph. Namely, $R^{\times k}\circ\epsilon_1\colon\Conf_{\mathbb{Z}_2}(S^n,k)\to\Conf_{\mathbb{Z}_2}(S^n,k)$ induces a self map of the spectral sequence of~(\ref{fib1}) which, at the level of the base space is induced by $-R$ and, at the level of the fiber space is induced by $\epsilon'_1$. As above, funtoriality implies that this self map of spectral sequences converges to the self map induced by $R^{\times k}\circ\epsilon_1$. But the latter self map is simply the one induced by $\epsilon_1\colon\Conf_{\mathbb{Z}_2}(S^n,k)\to\Conf_{\mathbb{Z}_2}(S^n,k)$, as $R^{\times k}$ is homotopic to the identity. Consequently, the action of $\epsilon_1$ on the cohomology of the total space of~(\ref{fib1}) can be traced back in the spectral sequence via the actions of $\epsilon'_1$ on the cohomology of the fiber, and of $-R$ on the cohomology of the base space of~(\ref{fib1}). By (the now standard) abuse of notation, we will use the notation $\epsilon_1$ for the above two $\mathbb{Z}_2$-actions on the cohomologies of the base and fiber spaces of~(\ref{fib1}), and for the spectral sequence $\mathbb{Z}_2$-action described in this paragraph.

\smallskip
Note that $\epsilon'_1$---i.e.~$R^{\times k}\circ \epsilon_1$ restricted to the fiber $\Conf_{\mathbb{Z}_2}(S^n-\{\pm\star\},k-1)$---is given by $R^{\times (k-1)}$. An easy check shows that, after removing the poles and taking into account the stereographic projection, the map $R$ induces the map $\tilde{R}:\mathbb{R}^n-\lbrace 0 \rbrace \longrightarrow \mathbb{R}^n-\lbrace 0 \rbrace$ given by $\tilde{R}(x)=\bar{x}\left/\,\lVert x \rVert^2\right.$ where $\bar{x}=(t_1,\ldots,t_{n-1},-t_n)$ for $x=(t_1,\ldots,t_n)$. Thus, in terms of the homeomorphism explained at the end of Remark~\ref{renotacion}, the self map $\epsilon'_1\colon\Conf_{\mathbb{Z}_2}(\mathbb{R}^n-\{0\},k-1)\to\Conf_{\mathbb{Z}_2}(\mathbb{R}^n-\{0\},k-1)$ takes the form $$\epsilon'_1(x_1,\ldots,x_{k-1})=(\tilde{R}(x_1),\ldots,\tilde{R}(x_{k-1}))$$ which, by (a final) abuse of notation, we denote by $\epsilon_1$ from now on. Likewise, from the comments at the beginning of Remark~\ref{renotacion}, we see that, for $2\leq i\leq k$, the restricted map $\epsilon_i\colon\Conf_{\mathbb{Z}_2}(\mathbb{R}^n-\{0\},k-1)\to\Conf_{\mathbb{Z}_2}(\mathbb{R}^n-\{0\},k-1)$ is given by $$\epsilon_i(x_1,\ldots,x_{k-1})=(x_1,\ldots,x_{i-2},\tau x_{i-1},x_{i+1},\ldots,x_{k-1}).$$
\indent
For future reference we record the following immediate consequence of the discussion above:
\begin{cor}\label{accionenlabaseref}
Let $1\leq i\leq k$. The spectral sequence $\mathbb{Z}_2$-action $\epsilon_i$ is trivial on the cohomology of the base space of~(\ref{fib1}) unless $i=1$ in which case it is multiplication by $(-1)^{n+1}$.
\end{cor}
Our fully novel work starts at this point. As a first step we state our cohomological description of the several topological $(\mathbb{Z}_2)^k$-actions settled above. Proofs appear later in the section.

\begin{thm}\label{accion}
For $n\geq 2$, the $(\mathbb{Z}_2)^k$-action on $H^*(\Conf_{\mathbb{Z}_2}(\mathbb{R}^n-\lbrace 0 \rbrace,k-1);R)$ is given by
\begin{equation}\label{eqaccion}
\epsilon_l A_{i,j}=
\begin{cases} 
(-1)^{n-1}A_{j,0}-A_{i,0}+A_{i,j}, & \textnormal{ if $l=1$, $j>0$};\\
-A_{|j|,0}-A_{i,0}+A_{i,j}, & \textnormal{ if $l=1$, $j<0$};\\
-A_{i,0}, &\textnormal{ if $l=1$, $j=0$, $i \geq 1$};\\
A_{i,-j}, & \textnormal{ if $l>1$, $|j|=l-1$};\\
(-1)^nA_{i,0}, & \textnormal{ if $l>1$, $i=l-1$, $j=0$};\\
(-1)^nA_{j,0}+(-1)^nA_{i,0}+(-1)^{n-1}A_{i,-j}, & \textnormal{ if $l>2$, $i=l-1$, $j>0$};\\
A_{|j|,0}+(-1)^nA_{i,0}+(-1)^{n-1}A_{i,|j|}, & \textnormal{ if $l>2$, $i=l-1$, $j<0$};\\
A_{i,j}, & \textnormal{ otherwise}.
\end{cases}
\end{equation}
\end{thm}

\begin{thm}\label{accionpar}
For $n\geq 2$ even, the $(\mathbb{Z}_2)^k$-action on the permanent cycles $\mathbb{K}^*\subseteq H^*(\Conf_{\mathbb{Z}_2}(\mathbb{R}^n-\lbrace 0 \rbrace,k-1);R)$ is given by
\begin{equation}\label{eqaccionpar}
\epsilon_l B_{i,j}=
\begin{cases} 
-B_{|j|,0}-B_{i,0}+B_{i,j}, & \textnormal{if $l=1$, $|j|>0$}; \\
-B_{i,0}, &\textnormal{if $l=1$, $j=0$, $i>1$};\\
B_{i,-j}, & \textnormal{if $l>1$, $|j|=l-1$};\\
B_{|j|,0}+B_{i,0}-B_{i,-j}, & \textnormal{if $l>2$, $i=l-1$, $|j|>0$};\\
B_{i,j}, &\textnormal{otherwise.}
\end{cases}
\end{equation}
\end{thm}
 Note that $B_{1,0}=0$ in (\ref{eqaccionpar}), and that the formulas in~(\ref{eqaccionpar}) agree with those in~(\ref{eqaccion}) for $n$ even and replacing each $A$ with $B$.

\begin{cor}\label{acciontotalnimpar}
For $n\geq2$ odd, the action of $(\mathbb{Z}_2)^k$ on 
\[H^*(\Conf_{\mathbb{Z}_2}(S^n,k);R)\cong H^*(S^n;R)\otimes H^*(\Conf_{\mathbb{Z}_2}(\mathbb{R}^n-\lbrace 0 \rbrace,k-1);R)= \Lambda(\iota_{n})\otimes R[\mathcal{A}]/I\] is the tensor product of the corresponding actions on each factor of the tensor product.
\end{cor}
\begin{cor}\label{acciontotalnpar}
Assume that the characteristic of $R$ is either zero or an odd integer. For $n\geq 2$ even, the action of $(\mathbb{Z}_2)^k$ on 
$$
H^*(\Conf_{\mathbb{Z}_2}(S^n,k);R)\cong R[\mathcal{B}]/J\otimes \Lambda(\lambda,\omega)/(2\lambda,\lambda\omega)
$$ 
is determined by
\[
\epsilon_l(\lambda)=\begin{cases}
-\lambda, & \textnormal{ if $l=1$; }\\
\lambda, & \textnormal{ if $l>1$},
\end{cases}
\]
\[
\epsilon_l(\omega)=\omega,\textnormal{$\qquad$ $\forall$ $l\geq 1$,}
\]
and the fact that it restricts to the action of $(\mathbb{Z}_2)^k$ on $\mathcal{B}$ stated in Theorem \ref{accionpar}.
\end{cor}

As explained in the discussion preparing the grounds for this section, the use of the Serre spectral sequence of~(\ref{fib1}), Corollary~\ref{accionenlabaseref}, and Theorems~\ref{accion} and~\ref{accionpar} yield Corollaries~\ref{acciontotalnimpar} and~\ref{acciontotalnpar}. Indeed, the asserted $(\mathbb{Z}_2)^k$-action on cohomology classes of the total space {\it coming} from the base space follows from standard considerations with the edge morphisms. On the other hand, since in total dimension $n-1$ the spectral sequence is concentrated in the ``fiber axis'' (see the chart in Remark~\ref{sparsenessreference}), it follows that the $(\mathbb{Z}_2)^k$-action on cohomology classes of the total space {\it mapping} (non-trivially) to the fiber can be read from the corresponding action on their images. 

In addition, Theorem~\ref{accionpar} is a straightforward consequence (whose verification is left to the reader) of the definitions and Theorem~\ref{accion}. So, the only remaining fact to prove in this section is Theorem~\ref{accion}. The case $k=2$ in Theorem~\ref{accion} is elementary; we next prove that the case $k>3$ follows from the case $k=3$.
\begin{prop}\label{ultredsec6}
The formula in~(\ref{eqaccion}) for $k>3$ is a formal consequence of the corresponding formula for $k=3$.
\end{prop}
\begin{proof}
The maps $\pi_{i,j} \colon \Conf_{\mathbb{Z}_2}(\mathbb{R}^n-\{0\},k-1) \to \Conf_{\mathbb{Z}_2}(\mathbb{R}^n-\{0\},2)$ given by $\pi_{i,j}(x_1,\ldots,x_{k-1}) = (x_j,x_i)$ (for $ 0 < j < i < k $) can be used to ``import'' the $(\mathbb{Z}_2)^k$-action on $\Conf_{\mathbb{Z}_2}(\mathbb{R}^n-\{0\},k-1)$ from the $(\mathbb{Z}_2)^3$-action on $\Conf_{\mathbb{Z}_2}(\mathbb{R}^n-\{0\},2)$ because of the following two easily-checked facts: Firstly, $\pi_{i,j}^*$ sends $A_{1,0}$, $A_{2,0}$, $A_{2,1}$, and $A_{2,-1}$ respectively to $A_{j,0}$, $A_{i,0}$, $A_{i,j}$, and $A_{i,-j}$. Secondly, for $1 \leq \ell \leq k$, $\pi_{i,j}$ fits in the commutative diagram 
 \[
\xymatrix{
\Conf_{\mathbb{Z}_2}(\mathbb{R}^n - \{0\},k-1) \ar[d]^{\pi_{i,j}}\ar[r]^{\epsilon_\ell} & \Conf_{\mathbb{Z}_2}(\mathbb{R}^n - \{0\},k-1)\ar[d]^{\pi_{i,j}}\\
\Conf_{\mathbb{Z}_2}(\mathbb{R}^n - \{0\},2)\ar[r]^{\bar{\epsilon}}&\Conf_{\mathbb{Z}_2}(\mathbb{R}^n - \{0\},2),
}
\]
where 
\[\bar{\epsilon}(x,y)=\begin{cases} \epsilon_3(x,y), & \text{ if $i= \ell-1$};\\
\epsilon_2( x,y),& \text{ if $j=\ell-1$};\\
\epsilon_1(x,y), & \text{ if $\ell=1$};\\
(x,y), & \text{ otherwise}.
\end{cases}
\]
\end{proof}

The rest of this section is devoted to proving Theorem \ref{accion} in the remaining case $k=3$, i.e.~to the proof of the following set of equalities: 
\begin{eqnarray}
\epsilon_1 A_{1,0} &=& -A_{1,0},\label{r110}\\ 
\epsilon_1 A_{2,0} &=& -A_{2,0},\label{r120}\\ 
\epsilon_1 A_{2,1} &=& (-1)^{n-1}A_{1,0} -A_{2,0}+A_{2,1},\label{r121}\\
\epsilon_1 A_{2,-1}&=& -A_{1,0}  -A_{2,0} + A_{2,-1},\label{r12m1}\\
\epsilon_2 A_{1,0} &=& (-1)^{n}A_{1,0},\label{r210}\\ 
\epsilon_2 A_{2,0} &=& A_{2,0},\label{r220}\\ 
\epsilon_2 A_{2,1} &=& A_{2,-1},\label{r221}\\
\epsilon_2 A_{2,-1}&=& A_{2,1},\label{r22m1}\\
\epsilon_3 A_{1,0} &=& A_{1,0},\label{r310}\\ 
\epsilon_3 A_{2,0} &=& (-1)^{n}A_{2,0},\label{r320}\\ 
\epsilon_3 A_{2,1} &=& (-1)^{n}A_{1,0}+(-1)^nA_{2,0}+ (-1)^{n-1}A_{2,-1},\label{r321}\\
\epsilon_3 A_{2,-1}&=& A_{1,0} + (-1)^nA_{2,0} + (-1)^{n-1}A_{2,1}.\label{r32m1}
\end{eqnarray}
\begin{rem}\label{xicostable2}
The above description corrects the action reported in~\cite[Table~2]{Xico}.
\end{rem}

The proof strategy for the set of relations in~(\ref{r110})--(\ref{r32m1}) is similar to that in the proof of Lemma~\ref{descgrados}: Recall the maps $p_{i,j}$ and $f_{r,s}$ introduced in~(\ref{laspes}) and~(\ref{lasefes}). By abuse of notation, for $| j | < i \leq 2$,  we will denote by $f_{i,j}$  the composition $\pi_{2,1}f_{i,j}:S^{n-1}\longrightarrow \Conf_{\mathbb{Z}_2}(\mathbb{R}^n-\lbrace 0 \rbrace,2)$. It is easy to check that these maps, together with the corresponding maps $p_{r,s}$, detect the generators for $\Conf_{\mathbb{Z}_2}(\mathbb{R}^n-\lbrace 0 \rbrace,2)$ in the sense of Lemma~\ref{elauxi1}. Then, in order to prove the above set of relations, we only need to compute the degree of the compositions
\[
\xymatrix{
S^{n-1}\ar[r]^-{f_{ij}} & \Conf_{\mathbb{Z}_2}(\mathbb{R}^n - \{0\},2)\ar[r]^{\epsilon_\ell} &\Conf_{\mathbb{Z}_2}(\mathbb{R}^n - \{0\},2)\ar[r]^-{p_{r,s}}&S^{n-1}
}
\]
for $1\leq\ell\leq 3$, $|j|<i\leq 2$, and $|s|<r\leq 2$.
\begin{proof}[Proof of relations~(\ref{r110})--(\ref{r32m1})]
We start by computing the action of $\epsilon_1$, i.e.~relations~(\ref{r110})--(\ref{r12m1}).
\begin{enumerate} 
\item $\epsilon_1 A_{1,0}$: We have
\[\begin{array}{lcllcllcllcl}
p_{1,0}\epsilon_1 f_{1,0}(x)&=&\bar{x} ,&
p_{1,0}\epsilon_1 f_{2,0}(x)&=&e,  &
p_{1,0}\epsilon_1 f_{2,1}(x)&=&e,  &
p_{1,0}\epsilon_1 f_{2,-1}(x)&=&e. \\
\end{array}\]
The first map is a reflection and the rest are constant maps, therefore
\[\begin{array}{lcllcllcllcl}
\text{deg}(p_{1,0}\epsilon_1 f_{1,0}) = -1,&
 \text{deg}(p_{1,0}\epsilon_1 f_{2,0}) = 0,&
   \text{deg}(p_{1,0}\epsilon_1 f_{2,1}) =0,&
   \text{deg}(p_{1,0}\epsilon_1 f_{2,-1})=0.
\end{array}\]
Thus, $\epsilon_1 A_{1,0}=-A_{1,0}$.
\item $\epsilon_1 A_{2,0}$: Clearly, $p_{2,0}\epsilon_1 f_{1,0}(x)=N(\frac{\bar{e}}{2})$, which implies $\text{deg}(p_{2,0}\epsilon_1 f_{1,0}) = 0$. On the other hand, note that $N(\tilde{R} (y))=N(\bar{y}) $ for all $y\in \mathbb{R}^n- \lbrace 0 \rbrace$, therefore
\[\begin{array}{lclcl}
p_{2,0}\epsilon_1 f_{2,0}(x)&=&N(\bar{x})&=&\bar{x},\\
p_{2,0}\epsilon_1 f_{2,1}(x)&=&N(\tilde{R}(e+\frac{x}{2}))&=&N({e+\frac{\bar{x}}{2}}), \\
p_{2,0}\epsilon_1 f_{2,-1}(x)&=&N(\tilde{R}(-e+\frac{x}{2}) )&=&N({-e+\frac{\bar{x}}{2}});
\end{array}\]
The second and third maps are not surjective, therefore we have
\[\begin{array}{lcllcllcl}
\text{deg}(p_{2,0}\epsilon_1 f_{2,0})=-1,&
\text{deg}(p_{2,0}\epsilon_1 f_{2,1})=0,&
\text{deg}(p_{2,0}\epsilon_1 f_{2,-1})=0.\\
\end{array}\]
Thus $\epsilon_1 A_{2,0}=-A_{2,0}$.
\item $\epsilon_1 A_{2,1}$: We have
\[\begin{array}{lclclcl}
p_{2,1}\epsilon_1 f_{1,0}(x)&=&N(\frac{e}{2}-\bar{x})&=&-N(\bar{x}-\frac{e}{2})&\simeq&-\bar{x},
\end{array}\]
therefore $\text{deg}(p_{2,1}\epsilon_1 f_{1,0}) = (-1)^{n-1}$. We also have
\[\begin{array}{lclcl}
p_{2,1}\epsilon_1 f_{2,0}(x)&=&N(2\bar{x}-e)&=&N(\bar{x}-\frac{e}{2}),
\end{array}\]
so $\text{deg}(p_{2,1}\epsilon_1 f_{2,0}) = -1$.
Further, in terms of the (already used) homotopy $F$ given by $F(t,(t_1,\ldots,t_n))=(tt_1,t_2,\ldots,t_n)$, we have
\begin{eqnarray}
p_{2,1}\epsilon_1 f_{2,1}(x)&=&N\left(\tilde{R}\left(e+\frac{x}{2}\right)-e\right)\nonumber
\\&=&N\left(e+\frac{F(1,\bar{x})}{2}-\left\lVert e+\frac{\bar{x}}{2}\right\rVert^2 e\right)\nonumber\\
&\simeq&N\left(e+\frac{F(-1,\bar{x})}{2}-\left\lVert e+\frac{\bar{x}}{2}\right\rVert^2 e\right)\label{hotounova}\\&\simeq&N\left(\frac{F(-1,\bar{x})}{2}\right).\label{hotodosva}
\end{eqnarray}
The homotopy in~(\ref{hotounova}) is given by $N(e+\frac{F(t,\bar{x})}{2}-\lVert e+\frac{\bar{x}}{2}\rVert^2 e)$ with  $t\in [-1,1]$. As before, we have to check that this homotopy is well defined: Suppose there exist $t\in [-1,1]$ and $x=(t_1,\ldots,t_n )\in S^{n-1}$ such that $F(t,\frac{\bar{x}}{2})+(1-\lVert  e+\frac{\bar{x}}{2}\rVert^2) e=0$. Then $F(t,\frac{\bar{x}}{2})=(-1+\lVert e+\frac{\bar{x}}{2} \rVert^2)e$ and so we have $\frac{tt_1}{2}=-1+\lVert e+\frac{\bar{x}}{2} \rVert^2$ and $t_i=0$ for $i>1$. This, in turn, implies $t_1=\pm 1$.
\begin{itemize}
\item[] \;\; If $t_1=1$, then $\frac{t}{2}=-1+\lVert e+\frac{e}{2}  \rVert^2 =-1+\lVert \frac{3e}{2} \rVert^2=\frac{5}{4}$, so $t=\frac{5}{2}>1$.
\item[] \;\; If $t_1=-1$, then $-\frac{t}{2}=-1+\lVert e-\frac{e}{2}  \rVert^2 =-1+\lVert \frac{e}{2} \rVert^2=-\frac{3}{4}$, so $t=\frac{3}{2}>1$.
\end{itemize}\medskip

\noindent
Both assumptions lead to a contradiction, so the homotopy is well defined. The homotopy in~(\ref{hotodosva}) is $N(\frac{F(-1,\bar{x})}{2}+t(1-\lVert e+\frac{\bar{x}}{2}\rVert^2)e)$ with $t\in [0,1].$ Let us verify that this homotopy is well defined: Suppose there exist $t\in [0,1]$ and $x=(t_1,\ldots,t_n )\in S^{n-1}$ such that $\frac{F(-1,\bar{x})}{2}+t(1-\lVert  e+\frac{\bar{x}}{2}\rVert^2) e=0$. Then $\frac{F(-1,\bar{x})}{2}=t(-1+\lVert  e+\frac{\bar{x}}{2}\rVert^2) e$ and so we have $-\frac{t_1}{2}=t(-1+\lVert  e+\frac{\bar{x}}{2}\rVert^2) $ and $t_i=0$ for $i>1$. This, in turn, implies $t_1=\pm 1$.
\begin{itemize}
\item[] \;\; If $t_1=1$, then $-\frac{1}{2}=t(-1+\lVert e+\frac{e}{2}  \rVert^2) =t(-1+\lVert \frac{3e}{2} \rVert^2)=t\frac{5}{4}$, so $t=-\frac{2}{5}<0$.
\item[] \;\; If $t_1=-1$, then $\frac{1}{2}=t(-1+\lVert e-\frac{e}{2}  \rVert^2) =t(-1+\lVert \frac{e}{2} \rVert^2)=-t\frac{3}{4}$, so  $t=-\frac{2}{3}<0$.
\end{itemize}\medskip

\noindent Both assumptions lead to a contradiction, consequently the homotopy is well defined. Therefore $p_{2,1}\epsilon_1 f_{2,1}$ is homotopic to a composition of two reflections. Thus
\[\begin{array}{lcl}
\text{deg}(p_{2,1}\epsilon_1 f_{2,1}) = 1.
\end{array}\]
Lastly, since $e$ is not enclosed by the image of $\tilde{R}(-e+\frac{x}{2})$, we see that the map
$p_{2,1}\epsilon_1 f_{2,-1}(x)=N(\tilde{R}(-e+\frac{x}{2})-e)$ is not surjective. Therefore
\[\begin{array}{lcl}
\text{deg}(p_{2,1}\epsilon_1 f_{2,-1}) =0,
\end{array}\]
and we conclude that $\epsilon_1 A_{2,1}=(-1)^{n-1}A_{1,0}-A_{2,0}+A_{2,1}$.

\item $\epsilon_1 A_{2,-1}$: We clearly have 
\[\begin{array}{lclclcl}
p_{2,-1}\epsilon_1 f_{1,0}(x)&=&N(\frac{e}{2}-\tau(\bar{x}))&=&N(\bar{x}+\frac{e}{2})&\simeq&\bar{x}
\end{array}\]
and
\[\begin{array}{lclclcl}
p_{2,-1}\epsilon_1 f_{2,0}(x)&=&N(2\bar{x}+e)&=&N(\bar{x}+\frac{e}{2})&\simeq& \bar{x}.
\end{array}\]
Therefore
\[\begin{array}{lcl}
\text{deg}(p_{2,-1}\epsilon_1 f_{2,0}) = \text{deg}(p_{2,-1}\epsilon_1 f_{1,0}) = -1.
\end{array}\]
On the other hand, since $-e$ is not enclosed by the image of $\tilde{R}(e+\frac{x}{2})$, $p_{2,-1}\epsilon_1 f_{2,1}(x)=N(\tilde{R}(e+\frac{x}{2})+e)$ is not surjective. Therefore
\[\begin{array}{lcl}
\text{deg}(p_{2,-1}\epsilon_1 f_{2,1}) = 0.
\end{array}\]
Lastly, 
\[\begin{array}{lclcl}
p_{2,-1}\epsilon_1 f_{2,-1}(x)&=&N(\tilde{R}(-e+\frac{x}{2})+e)&\simeq&N(-e+\frac{F(-1,\bar{x})}{2}+\lVert -e+\frac{\bar{x}}{2}\rVert^2 e)\\
&\simeq&N(\frac{F(-1,\bar{x})}{2})&\simeq&x,
\end{array}\]
where the first homotopy is given by $N(-e+\frac{F(t,\bar{x})}{2}+\lVert -e+\frac{\bar{x}}{2}\rVert^2 e)$ with $t\in [-1,1]$, and the second one by $N(\frac{F(-1,\bar{x})}{2}+t(-1+\lVert -e+\frac{\bar{x}}{2}\lVert^2)e)$, with $t\in [0,1].$ We can show that these homotopies are well defined in a similar fashion to the previous case.
Therefore, 
\[\begin{array}{lcl}
\text{deg}(p_{2,-1}\epsilon_1 f_{2,-1}) = 1,
\end{array}\]
And we conclude that $\epsilon_1 A_{2,-1}=-A_{1,0}-A_{2,0}+A_{2,-1}.$
\end{enumerate}

The analysis for $\epsilon_2$ and $\epsilon_3$ below is entirely analogous to the computation of the action of $\epsilon_1$ that we just have done in full detail. Therefore we will just record the results of the computations, without writing out details.

Next we consider $\epsilon_2$, i.e.~relations~(\ref{r210})--(\ref{r22m1}).
\begin{enumerate}
\item $\epsilon_2 A_{1,0}:$ We have
\[\begin{array}{lcllcllcllcl}
p_{1,0}\epsilon_2 f_{1,0}(x)&\!\!=\!\!&-x,  &
p_{1,0}\epsilon_2 f_{2,0}(x)&\!\!=\!\!&-e,  &
p_{1,0}\epsilon_2 f_{2,1}(x)&\!\!=\!\!&-e,  &
p_{1,0}\epsilon_2 f_{2,-1}(x)&\!\!=\!\!&-e, \\
\end{array}\]
therefore
\[\begin{array}{lcllcllcllcl}
\text{deg}(p_{1,0}\epsilon_2 f_{1,0}) = (-1)^n, &
 \text{deg}(p_{1,0}\epsilon_2 f_{2,0}) = 0,&
   \text{deg}(p_{1,0}\epsilon_2 f_{2,1}) =0, &
   \text{deg}(p_{1,0}\epsilon_2 f_{2,-1})=0.
\end{array}\]
Thus, $\epsilon_2 A_{1,0}=(-1)^n A_{1,0}.$

\item $\epsilon_2 A_{2,0}:$ We have
\[\begin{array}{lclcl}
p_{2,0}\epsilon_2 f_{1,0}(x)&=&e,  \\
p_{2,0}\epsilon_2 f_{2,0}(x)&=&N(\frac{x}{2}) &=& x,\\
p_{2,0}\epsilon_2 f_{2,1}(x)&=&N(e+\frac{x}{2})&\simeq&0, \\
p_{2,0}\epsilon_2 f_{2,-1}(x)&=&N(-e+\frac{x}{2})&\simeq&0; \\
\end{array}\]
therefore
\[\begin{array}{lcllcllcllcl}
\text{deg}(p_{2,0}\epsilon_2 f_{1,0}) = 0, &
 \text{deg}(p_{2,0}\epsilon_2 f_{2,0}) = 1,&
   \text{deg}(p_{2,0}\epsilon_2 f_{2,1}) =0, &
   \text{deg}(p_{2,0}\epsilon_2 f_{2,-1})=0.
\end{array}\]
Thus, $\epsilon_2 A_{2,0}= A_{2,0}$.

\item $\epsilon_2 A_{2,1}:$ We have
\[\begin{array}{lclclcl}
p_{2,1}\epsilon_2 f_{1,0}(x)&=&N(2e+x)&\simeq&0,  \\
p_{2,1}\epsilon_2 f_{2,0}(x)&=&N(\frac{x}{2}+e) &\simeq& 0,\\
p_{2,1}\epsilon_2 f_{2,1}(x)&=&N(e+\frac{x}{2}+e)&\simeq&0, \\
p_{2,1}\epsilon_2 f_{2,-1}(x)&=&N(-e+\frac{x}{2}+e)&=&N(\frac{x}{2})&=&x; \\
\end{array}\]
therefore
\[\begin{array}{lcllcllcllcl}
\text{deg}(p_{2,1}\epsilon_2 f_{1,0}) = 0, &
 \text{deg}(p_{2,1}\epsilon_2 f_{2,0}) = 0,&
   \text{deg}(p_{2,1}\epsilon_2 f_{2,1}) =0, &
   \text{deg}(p_{2,1}\epsilon_2 f_{2,-1})=1.
\end{array}\]
Thus $\epsilon_2 A_{2,1}=A_{2,-1}.$

\item $\epsilon_2 A_{2,-1}:$ Note that $\epsilon_2^2=\text{identity}$. Application of the previous case yields $\epsilon_2 A_{2,-1}=A_{2,1}$.
\end{enumerate}

Lastly, we consider $\epsilon_3$, i.e.~relations~(\ref{r310})--(\ref{r32m1}).
\begin{enumerate}
\item $\epsilon_3 A_{1,0}:$ We have
\[\begin{array}{lcllcllcllcl}
p_{1,0}\epsilon_3 f_{1,0}(x)&=&x,  &
p_{1,0}\epsilon_3 f_{2,0}(x)&=&e,  &
p_{1,0}\epsilon_3 f_{2,1}(x)&=&e,  &
p_{1,0}\epsilon_3 f_{2,-1}(x)&=&e; \\
\end{array}\]
therefore
\[\begin{array}{lcllcllcllcl}
\text{deg}(p_{1,0}\epsilon_3 f_{1,0}) = 1,&
 \text{deg}(p_{1,0}\epsilon_3 f_{2,0}) = 0,&
   \text{deg}(p_{1,0}\epsilon_3 f_{2,1}) =0,&
   \text{deg}(p_{1,0}\epsilon_3 f_{2,-1})=0.
\end{array}\]
Thus, $\epsilon_3 A_{1,0}=A_{1,0}.$

\item $\epsilon_3 A_{2,0}:$ We have
\[
\begin{array}{lcllcllcl}
p_{2,0}\epsilon_3 f_{1,0}(x)&=&N(-\frac{e}{2}),  \\
p_{2,0}\epsilon_3 f_{2,0}(x)&=&N(\tau(\frac{x}{2}))&=& -x, \\
p_{2,0}\epsilon_3 f_{2,1}(x)&=&N(\tau(e+\frac{x}{2}))&\simeq&0,  \\
p_{2,0}\epsilon_3 f_{2,-1}(x)&=&N(\tau(-e+\frac{x}{2}))&\simeq&0;
\end{array}
\]
therefore
\[\begin{array}{lcllcl}
\text{deg}(p_{2,0}\epsilon_3 f_{1,0})&\!\!=\!\!& 0,&
 \text{deg}(p_{2,0}\epsilon_3 f_{2,0}) &\!\!=\!\!& (-1)^n,
\end{array}\]
\[\begin{array}{lcllcl}
   \text{deg}(p_{2,0}\epsilon_3 f_{2,1}) &\!\!=\!\!&0, &
   \text{deg}(p_{2,0}\epsilon_3 f_{2,-1})&\!\!=\!\!&0.
\end{array}\]
Thus, $\epsilon_3 A_{2,0}=(-1)^n A_{2,0}.$

\item $\epsilon_3 A_{2,1}:$
We have
\[\begin{array}{lclclclcl}
p_{2,1}\epsilon_3 f_{1,0}(x)&=&N(-\frac{e}{2}-x)&=&-N(\frac{e}{2}+x)&\simeq&-x,  \\
p_{2,1}\epsilon_3 f_{2,0}(x)&=&N(-2x-e)&=&-N(2x+e)&\simeq&-x,  \\
p_{2,1}\epsilon_3 f_{2,1}(x)&=&N(\tau(e+\frac{x}{2})-e)&\simeq&0,
\end{array}\]
therefore
\[\begin{array}{lcllcllcl}
\text{deg}(p_{2,1}\epsilon_3 f_{1,0}) = (-1)^n,&
\text{deg}(p_{2,1}\epsilon_3 f_{2,0}) = (-1)^n,&
 \text{deg}(p_{2,1}\epsilon_3 f_{2,1}) = 0.
\end{array}\]
On the other hand,
\[\begin{array}{lclclclcl}
p_{2,1}\epsilon_3 f_{2,-1}(x)&=&N(\tau(-e+\frac{x}{2})-e)&\simeq&N(e+\frac{F(-1,-x)}{2}-\lVert -e+\frac{x}{2}\rVert^2 e)\\
&\simeq&N(\frac{F(-1,-x)}{2})&=&-F(-1,x),
\end{array}\]
where the first homotopy is given by $N(e+\frac{F(t,-x)}{2}-\lVert -e+\frac{x}{2}\rVert^2 e)$ with $t\in [-1,1]$, and the second one is given by $N(\frac{F(-1,-x)}{2}+t(1-\lVert -e+\frac{x}{2}\rVert^2)e)$, with $t\in [0,1].$ Therefore
\[\begin{array}{lclcl}
 \text{deg}(p_{2,1}\epsilon_3 f_{2,-1}) = (-1)^{n-1}.
 \end{array}\]
And we conclude $\epsilon_3 A_{2,1} = (-1)^nA_{1,0}+(-1)^n A_{2,0}+ (-1)^{n-1} A_{2,-1}.$

\item $\epsilon_3 A_{2,-1}:$ Note that $\epsilon_3^2 =\text{identity}$. By our previous computations,
\[
\begin{array}{lclclclcl}
A_{2,1}=\epsilon_3(\epsilon_3 A_{2,1}) &=& \epsilon_3((-1)^nA_{1,0}+(-1)^n A_{2,0}+ (-1)^{n-1} A_{2,-1})\\
&=&(-1)^nA_{1,0}+A_{2,0}+(-1)^{n-1}\epsilon_3 A_{2,-1}.
\end{array}
\]
Therefore $\epsilon_3 A_{2,-1} = A_{1,0}+(-1)^n A_{2,0}+ (-1)^{n-1} A_{2,1}.$
\end{enumerate}
This finishes the verification of relations~(\ref{r110})--(\ref{r32m1}).
\end{proof}
\begin{rem}\label{3.6}
Theorem~\ref{accion} corrects~\cite[Lemma~7]{FZ2}. The situation is closely related to our discussion, in Remark~\ref{remark2.4}, of the existence of inconsistencies in Feichtner-Ziegler's determination of a presentation for the cohomology ring of the fiber and base spaces in (\ref{fib1}). As described next, the problem can be traced back to the description in \cite[Lemma 7(iv)]{FZ2} of the action of the various $\epsilon_i$ on cohomology rings. To simplify the explanation, once again we adopt momentarily Feichtner-Ziegler's notation in \cite{FZ2}---which the reader is assumed to be familiar with. The proof of Lemma~7(iv) in \cite{FZ2} is based on the asserted equality $(A_2\circ A_1)^*(c^+_{1,2})=(-1)^kc^+_{1,2}$ whose proof, in turn, is reduced to showing that the obvious map
\begin{equation}\label{auxmap}
A_2\circ A_1\colon\mathcal{M}(\{U_1,U_2,U^+_{1,2}\})\to\mathcal{M}(\{U_1,U_2,U^+_{1,2}\})
\end{equation}
satisfies
\begin{equation}\label{aux}
(A_2\circ A_1)^*(\widetilde{c}_{1,2})=(-1)^k\widetilde{c}_{1,2}.
\end{equation}
(Note that (\ref{auxmap}) is not to be understood as a composition of maps from $\mathcal{M}(\{U_1,U_2,U^+_{1,2}\})$ to itself.) Feichtner-Ziegler's argument for (\ref{aux}) then proceeds by considering the central sphere $S$ (of radius $\sqrt{2}$) in ${}^\perp U^+_{1,2}\setminus\{0\}$ which retracts from $\mathcal{M}(\{U_1,U_2,U^+_{1,2}\})$ (with retraction $p$). It is observed that 
\begin{equation}\label{flawed}
\mbox{\it$(\ref{auxmap})$ restricts on $S$ as the antipodal map}
\end{equation}
and, from this, (\ref{aux}) is concluded. But such a conclusion cannot hold: The assertion in~(\ref{flawed}) is right, and gives the (strict) commutativity of the diagram
$$
\xymatrix{
 S \ar@{^{(}->}[r] \ar[d]_{\mathrm{antipodal}} 
 & \mathcal{M}(\{U_1,U_2,U^+_{1,2}\})  \ar[d]^{A_2\circ A_1}  \\
 S \ar@{^{(}->}[r] & \mathcal{M}(\{U_1,U_2,U^+_{1,2}\}).
 }
$$
But (\ref{aux}) cannot be drawn from this, since the map induced in cohomology by the inclusion $S\hookrightarrow\mathcal{M}(\{U_1,U_2,U^+_{1,2}\})$ has a nontrivial kernel. Indeed, $H^{k-1}(\mathcal{M}(\{U_1,U_2,U^+_{1,2}\}))$ is free of rank 3, while $H^{k-1}(S)$ is free of rank 1. Instead, what would certainly give (\ref{aux}) is the existence of a commutative diagram (at least up to homotopy)
$$
\xymatrix{
 \mathcal{M}(\{U_1,U_2,U^+_{1,2}\})  \ar[d]_{A_2\circ A_1}  \ar[r]^{\hspace{1.5cm}p} & S
 \ar[d]^{\mathrm{antipodal}} \\
 \mathcal{M}(\{U_1,U_2,U^+_{1,2}\})    \ar[r]^{\hspace{1.5cm}p} & S.
 }
$$
But (\ref{aux}) is false according to Theorem \ref{accion}, so that such a diagram is impossible. 
\end{rem}

We close the section with a technical result that will be used latter in the paper. Namely, the maps $\epsilon_1, \epsilon_2, \cdots,\epsilon_k \colon \Conf_{\mathbb{Z}_2}(\mathbb{R}^n-\{0\},k-1) \to \Conf_{\mathbb{Z}_2}(\mathbb{R}^n-\{0\},k-1)$ are related as follows:

\begin{lem}\label{lema3.3}
For $n$ odd, $\epsilon_1 \simeq  \epsilon_2\cdots \epsilon_{k}$. For $n$ even, $\epsilon_1 \simeq    h^{\times (k-1)}\epsilon_2\cdots \epsilon_{k}$, with $h:\mathbb{R}^n-\lbrace 0 \rbrace\to \mathbb{R}^n-\lbrace 0 \rbrace$ given by $h(x)=\bar{x}$.
\end{lem}
\begin{proof}
Let $g,f:\mathbb{R}^n-\lbrace 0 \rbrace\longrightarrow \mathbb{R}^n-\lbrace 0 \rbrace$ be the maps $f(x)=\dfrac{x}{\lVert x \rVert^2}$ and $g(x)=-x$. We have that any $T\in O(n)$ is $\mathbb{Z}_2$-equivariant (with $\mathbb{Z}_2=\langle \tau \rangle$):
\[
T(\tau x)=T\left(\dfrac{-x}{\lVert x\rVert^2}\right)=\dfrac{-T(x)}{\lVert x\rVert^2}=\dfrac{-T(x)}{\lVert T(x)\rVert^2}=\tau T(x).
\]
This, coupled with the injectivity of $T$, implies that $T^{\times (k-1)}$ sends orbit configurations to orbit configurations. Therefore $h^{\times (k-1)},g^{\times (k-1)}$ are maps of orbit configurations spaces. We also have that $f$ is injective, and it is also $\mathbb{Z}_2$-equivariant:
\[
f(\tau x)=f\left(\dfrac{-x}{\lVert x \rVert^2}\right)=\dfrac{\dfrac{-x}{\lVert x \rVert^2}}{\left\lVert \dfrac{-x}{\lVert x \rVert^2}\right\rVert^2}=\dfrac{-\left(\dfrac{x}{\lVert x \rVert^2}\right)}{\left\lVert \dfrac{x}{\lVert x \rVert^2}\right\rVert^2}=\tau\left( \dfrac{x}{\lVert x \rVert^2}\right)=\tau f(x).
\]
Thus $f^{\times (k-1)}$ is a map between orbit configurations spaces.
Note that $\tau=gf$ and $\tilde{R}=hf$. Therefore we have $\epsilon_1=\tilde{R}^{\times (k-1)}=(hf)^{\times (k-1)}=h^{\times (k-1)}f^{\times (k-1)}.$ 

For $n$ odd, it is known that there is a homotopy through $O(n)$ between $g$ and $h$, so we have $g^{\times (k-1)}\simeq h^{\times (k-1)}$ as maps of orbit configurations spaces, therefore 
\[
\epsilon_1 = h^{\times (k-1)}f^{\times (k-1)}\simeq g^{\times (k-1)}f^{\times (k-1)}=\tau^{\times (k-1)}=\epsilon_2\cdots\epsilon_k
\]
as maps of orbit configuration spaces.
For $n$ even, there is a homotopy through $O(n)$ between $g$ and the identity. Therefore 
\[
\epsilon_1=h^{\times (k-1)} f^{\times (k-1)}\simeq h^{\times (k-1)} g^{\times (k-1)} f^{\times (k-1)}=h^{\times (k-1)}\tau^{\times (k-1)}=h^{\times (k-1)}\epsilon_2\cdots \epsilon_k
\]
as maps of orbit configurations spaces.
\end{proof}
Of course, Theorem \ref{accion} can be used to give a description of the effect in cohomology of the map $h^{\times (k-1)} \colon \Conf_{\mathbb{Z}_2}(\mathbb{R}^n-\{0\},k-1) \to \Conf_{\mathbb{Z}_2}(\mathbb{R}^n-\{0\},k-1)$ that arises in Lemma \ref{lema3.3} for $n$ even. We omit the details as we will not have occasion of using such information. Yet, in the next section we will need to describe the behavior of the map $h^{\times(k-1)}$ on the permanent cycles $\mathbb{K}^*$ of the previous section.

\section{$(\mathbb{Z}_2)^k$-Invariants}\label{secndinvtes}
We continue with the assumptions that $n\geq2$ and that $R$ denotes a commutative ring with unit, but we now assume in addition that $2$ is invertible. The reason for the new hypothesis is two fold. For one, as explaind in Section~\ref{proofmethod}, the invertibility of 2 implies that the spectral sequence associated to the covering projection $\rho_{n,k}$ is concentrated on the fiber axis, thus allowing us to recover the cohomology of $\Conf(\mathbb{R}\mathrm{P}^n,k)$ as the $(\mathbb{Z}_2)^k$-invariants of the $(\mathbb{Z}_2)^k$-action on $\Conf_{\mathbb{Z}_2}(S^n,k)$. In addition, the invertibility of 2 is needed in our actual calculation of invariants. Namely, it is not difficult to identify critical invariants, and we prove that these are all the invariants using a change of basis that requires the invertibility of 2 in an essential way---see~(\ref{2esinvertible1}) and~(\ref{2esinvertible2}) below. We do not know the full structure of invariants if 2 is not invertible, e.g.~if $\mathrm{char}(R)=2$. Although such a problem could be accessible and interesting, a subtler issue to solve might be to understand the corresponding spectral sequence associated to the covering projection $\rho_{n,k}$. Even the determination of the $E_2$ term---the cohomology of $\Conf_{\mathbb{Z}_2}(\mathbb{R}^n-\{0\},k-1)$ with twisted $(\mathbb{Z}_2)^k$-coefficients---could be a hard task.

\medskip
We start by computing the $(\mathbb{Z}_2)^{k}$-invariants in $H^*(\Conf_{\mathbb{Z}_2}(\mathbb{R}^n-\lbrace 0 \rbrace,k-1);R)$ (Theorems~\ref{generadoresinvariantesimpar} and~\ref{isoinvariantesimpar}) assuming $n$ is odd, hypothesis that will be in force until Theorem \ref{thminvariantesimpar} where the invariants in $H^*(\Conf_{\mathbb{Z}_2}(S^n,k);R)$ are described for odd $n$.

For $0<i<k$, we let $C_{i,0}$ stand for $A_{i,0}$ and, for $0< j <i<k$, we define $C^{+}_{i,j}=A_{i,j}+A_{i,-j}-A_{i,0}$ and $C^{-}_{i,j}=-A_{i,j}+A_{i,-j}-A_{j,0}$. For ease of notation, for a positive $j$ we will also use the notation $C_{i,j}$ and $C_{i,-j}$ to stand respectively for $C^+_{i,j}$ and $C^-_{i,j}$. Put $\mathcal{C}^+=\lbrace C^+_{i,j}\,|\,1\leq j < i<k \rbrace$,  $\mathcal{C}^-=\lbrace C^-_{i,j}\,|\,1\leq j < i<k \rbrace$, $\mathcal{C}_0= \lbrace C_{i,0}\,|\,1\leq i< k \rbrace$ and $\mathcal{C}=\mathcal{C}^+\cup\mathcal{C}^-  \cup \mathcal{C}_0$. Clearly, $\mathcal{C}$ is a basis for $H^{n-1}(\Conf_{\mathbb{Z}_2}(\mathbb{R}^n-\lbrace 0 \rbrace,k-1))$ with inverse change of basis map given by
\begin{eqnarray}
A_{i,j}&=&\frac{C^+_{i,j}-C^-_{i,j}+C_{i,0}-C_{j,0}}{2},\nonumber\\
A_{i,-j}&=&\frac{C^+_{i,j}+C^-_{i,j}+C_{i,0}+C_{j,0}}{2},\label{2esinvertible1}\\
A_{i,0}&=&C_{i,0},\nonumber
\end{eqnarray}
for $0<j<i<k$ (the assumption on the invertibility of 2 is needed here). The formulas in~(\ref{2esinvertible1}) make it clear that $H^*(\Conf_{\mathbb{Z}_2}(\mathbb{R}^n-\{0\},k-1);R)$ is additively generated by the products 
\begin{equation}\label{productosdeCs}
C_{i_1,j_1} \cdots C_{i_r,j_r}
\end{equation}
with $ | j_l | < i_l < k $ for $l=1, \ldots, r$. Our first goal is to show that, in fact, an additive basis is formed by such products that satisfy in addition
\begin{equation}\label{condiciondeorden}
i_l  < i_{l'}\textnormal{ \, whenever \, } l < l'.
\end{equation}
\begin{ex}\label{ejemplo1}
For $n\geq 2$ odd, the multiplicative relations among the $A_{i,j}$'s yield
\begin{equation}\label{ejemplo1de1}
C^-_{3,2}C_{3,0}=-A_{2, 0} A_{3, 0} + A_{3, -2} A_{3, 0} - A_{3, 0} A_{3, 2}=-A_{2,0} A_{3,-2}+A_{2,0} A_{3,0}-A_{2,0} A_{3,2}=-C^+_{3, 2} C_{2, 0}
\end{equation}
so that
%
\begin{eqnarray*}
C^+_{4,3}C^-_{4,2}&=&-A_{2,0} A_{4,-3}+A_{4,-3} A_{4,-2}+A_{2,0} A_{4,0}-A_{4,-2} A_{4,0}-A_{4,-3} A_{4,2}+A_{4,0} A_{4,2}\\
&&-A_{2,0} A_{4,3}+A_{4,-2} A_{4,3}-A_{4,2} A_{4,3}\\
&=&A_{2,0} A_{4,-3}-A_{3,-2} A_{4,-3}+A_{3,2} A_{4,-3}-A_{3,-2} A_{4,-2}+A_{3,0} A_{4,-2}-A_{3,2} A_{4,-2}\\
&&-A_{2,0} A_{4,0}+A_{3,-2} A_{4,2}-A_{3,0} A_{4,2}+A_{3,2} A_{4,2}-A_{2,0} A_{4,3}+A_{3,-2} A_{4,3}-A_{3,2} A_{4,3}\\
&=&-C^-_{3,2} C^-_{4,3}-C^-_{4,2} C^+_{3,2}-C^+_{3,2} C_{2,0}-C^-_{3,2} C_{3,0}-C_{2,0} C_{4,0}
\end{eqnarray*}
which, by~(\ref{ejemplo1de1}), becomes \ $C^+_{4,3}C^-_{4,2}=-C^-_{3,2} C^-_{4,3}-C^-_{4,2} C^+_{3,2}-C_{2,0} C_{4,0}$.
\end{ex}

Note that the set of products in (\ref{productosdeCs}) satisfying (\ref{condiciondeorden}) is in bijective correspondence with the basis described just before Theorem \ref{2.1}. Using Nakayama's lemma we see that the former set will be in fact an additive basis of $H^*(\Conf_{\mathbb{Z}_2}(\mathbb{R}^n-\lbrace 0 \rbrace,k-1);R)$ as long as it additively generates. In turn, the latter condition follows directly from the fact that the products $A_{i_1,j_1} \cdots A_{i_r,j_r}$ satisfying the condition (\ref{condiciondeorden}) form an additive basis, from the explicit form of the relations~(\ref{2esinvertible1}) expressing the $A_{i,j}$'s in terms of the $C_{i,j}$'s, and from the relations in item 2 of Lemma \ref{relsmultsalgunasCs} below---which generalizes the calculation illustrated in~(\ref{ejemplo1de1}). The proof of the lemma is straightforward and left to the reader.
\begin{lem}\label{relsmultsalgunasCs}
For $n\geq 2$ odd, the elements of $\mathcal{C}$ satisfy the following multiplicative relations:
$$
\mbox{For $0<j<i<r<k$, \ \ \ }\begin{cases}
&C^+_{r,i} C^+_{r,j}=-C^+_{i,j} C^+_{r,j} + C^+_{i,j} C^+_{r,i},\\
&C^+_{r,i} C^-_{r,j}=-C^-_{i,j} C^-_{r,i} - C^+_{i,j}C^-_{r,j}  - C_{j,0} C_{r,0},\hspace{3cm}\\
&C^-_{r,i} C^+_{r,j}=C^-_{i,j} C^-_{r,j} + C^+_{i,j}C^-_{r,i}  - C_{i,0} C_{r,0},\\
&C^-_{r,i} C^-_{r,j}=C^-_{i,j} C^+_{r,j} - C^-_{i,j} C^+_{r,i} + C_{j,0} C_{i,0}.
\end{cases}
$$
$$
\mbox{For $0<i<r<k,\hspace{.6cm}$ \ \ \ }\begin{cases}
&C^+_{r,i} C_{r,0}=-C_{i,0}C^-_{r,i} ,\\
&C^-_{r,i} C_{r,0}=- C_{i,0}C^+_{r,i}.\hspace{6.25cm}
\end{cases}
$$
$$
\mbox{For $0\leq j<i<k,\hspace{.65cm}$ \ \ \ }\begin{cases}
&(C^+_{i,j})^2=0,\\
&(C^-_{i,j})^2=0,\\
&C^+_{i,j}C^-_{i,j} = -C_{j,0}C_{i,0}.\hspace{6.2cm}
\end{cases}
$$
\end{lem}

The advantage of using $\mathcal{C}$ over $\mathcal{A}$ to compute invariants becomes apparent when describing the action of $(\mathbb{Z}_2)^k$, for a straightforward verification yields that, except for the cases in~(\ref{excep1}) and~(\ref{excep2}) below, each $\epsilon_l$ acts as the identity on the elements of $\mathcal{C}$:
\begin{eqnarray}
\epsilon_l C^-_{ij}&=&-C^-_{ij}, \text{ \; if $i=l-1$ or $j=l-1$}; \label{excep1}\\
\epsilon_l C_{i,0}&=&
-C_{i,0}, \text{ \; if $i=l-1$ or $l=1$}. \label{excep2}
\end{eqnarray}

\begin{thm}\label{generadoresinvariantesimpar}
Suppose $R$ is a commutative ring with unit where $2$ is invertible. For $n\geq 2$ odd,  the $(\mathbb{Z}_2)^k$-invariants in $H^{*}(\Conf_{\mathbb{Z}_2}(\mathbb{R}^n-\lbrace 0 \rbrace,k-1);R)$ are multiplicatively generated by the set $\mathcal{C}^+$. In fact, and additive basis of the invariants is formed by the products (\ref{productosdeCs}) satisfying (\ref{condiciondeorden}) and $j_l>0$ for $l=1,\ldots,r$.
\end{thm}
\begin{proof}
Let $x\in H^{m(n-1)}(\Conf_{\mathbb{Z}_2}(\mathbb{R}^n-\lbrace 0 \rbrace,k-1);R) $ be an invariant. We will show that each of the basis elements appearing with a nontrivial coefficient in the expression of $x$ as a linear combination of the basis of products (\ref{productosdeCs}) satisfying (\ref{condiciondeorden}) has no factors belonging to $\mathcal{C}^-$ or $\mathcal{C}_0$. Write
\begin{equation}\label{combinacionlineal}
x=\sum a_I C^{}_{i_1,j_1}\cdots C^{}_{i_m,j_m},
\end{equation}
where each coefficient $a_I$ is non-zero  and the summation runs over some multi-indices $$I=((i_1,j_1),\ldots,(i_m,j_m))$$ such that $|j_l|<i_l$ and $i_l<i_{l'}$ if $l<l'$. Note that given our description of the $(\mathbb{Z}_2)^k$-action on $\mathcal{C}$, each monomial  $C^{}_{i_1,j_1}\cdots C^{}_{i_m,j_m}$ is sent to a multiple of itself under the action of any element in $(\mathbb{Z}_2)^k$. Since $2$ is invertible, this means that each term $C^{}_{i_1,j_1}\cdots C^{}_{i_m,j_m}$ appearing in (\ref{combinacionlineal}) is invariant. Fix $I$, and consider the corresponding invariant monomial $z=C^{}_{i_1,j_1}\cdots C^{}_{i_m,j_m}$. 
Suppose that the set of integers $i$ such that we have a factor of the form $C^-_{i,j} $ in $z$ is non-empty, and let $i_{0}$ be the greatest element of this set. By applying $\epsilon_{i_{0}+1}$ to $z$ we get that $-z=\epsilon_{i_{0}+1}z=z$, which is a contradiction, so $z$ has no factors belonging to $\mathcal{C}^-$. An entirely analogous argument shows that there are no factors belonging to $\mathcal{C}_0$ in $z$ either.
\end{proof}

\begin{thm}\label{isoinvariantesimpar}
Suppose $R$ is a commutative ring with unit where $2$ is invertible. For $n\geq 2$ odd, there is an $R$-algebra isomorphism
\[
H^{*}(\Conf_{\mathbb{Z}_2}(\mathbb{R}^n-\lbrace 0 \rbrace,k-1);R)^{(\mathbb{Z}_2)^k} \cong R[\mathcal{C^+}]/\mathcal{K}, 
\]
where $\mathcal{K}$ is the ideal generated by the elements ${C^+_{i,j}}^2$ and $C^+_{r,i}C^+_{r,j}-C^+_{i,j}(C^+_{r,i}-C^+_{r,j})$ for $0<j<i<r<k$.
\end{thm}
\begin{proof}
Lemma \ref{relsmultsalgunasCs} gives an obvious ring map (with domain in $R[\mathcal{C}^+] / \mathcal{K}$). This is an isomorphism since it sets a bijective correspondence between the basis described in Theorem \ref{generadoresinvariantesimpar} and the usual basis in the domain.
\end{proof}
\begin{rem}\label{4.6}
Note that the second relation in the preceding theorem is identical to the known relation (\ref{relacionesusual}). In particular, the cohomology ring described in Theorem \ref{isoinvariantesimpar} is isomorphic to the cohomology ring of the standard configuration space of $k-1$ ordered points in $\mathbb{R}^n$.
\end{rem}
Since the canonical projection $S^n \to \mathbb{R}\mathrm{P}^n$ induces a $(\mathbb{Z}_2)^k$ covering space $\Conf_{\mathbb{Z}_2}(S^n,k) \to\Conf(\mathbb{R}\mathrm{P}^n,k)$, Theorem \ref{2.2}, the fact that the $(\mathbb{Z}_2)^k$-action on $H^*(S^n;R)$ is trivial for odd $n$, and the preceding theorem imply the following result:
\begin{thm}\label{thminvariantesimpar}
Let $R$ be a commutative ring with unit where $2$ is invertible. For $n\geq 2$ odd, there is an isomorphism
\[
H^*(\Conf(\mathbb{R}\mathrm{P}^n,k);R)=H^{*}(\Conf_{\mathbb{Z}_2}(S^n,k);R)^{(\mathbb{Z}_2)^k} \cong \Lambda(\iota_n)\otimes R[\mathcal{C^+}]/\mathcal{K}
\]
of $R$-algebras.
\end{thm}
Next goal is to describe all the $(\mathbb{Z}_2)^{k}$-invariant permanent cycles in $\mathbb{K}^*\subseteq H^*(\Conf_{\mathbb{Z}_2}(\mathbb{R}^n-\lbrace 0 \rbrace,k-1);R)$ (Theorem~\ref{invariantespar}) as well as all the $(\mathbb{Z}_2)^{k}$-invariant elements in $H^*(\Conf_{\mathbb{Z}_2}(S^n,k);R)$ (Theorem~\ref{isoinvariantespar}) for $n$ even, assumption that will be in force throughout the rest of the section. For $0<j<i<k$ define
$$
D^{+}_{i,j}=B_{i,j}+B_{i,-j}-B_{i,0}-B_{j,0},\qquad
D^{-}_{i,j}=B_{i,j}-B_{i,-j},\qquad\mbox{and}\qquad
D_{i,0}=B_{i,0},
$$
and, for $0<|j|<i$, let  \[D^{}_{i,j}=
\begin{cases} 
D^+_{i,j}& \text{if $j>0$}; \\
D^-_{i,|j|},&\text{if $j<0$;}
\end{cases}
\]
(Recall $B_{1,0}=0$.) Put $\mathcal{D}=\lbrace D^+_{i,j} \,|\,0<j<i<k\rbrace\cup \lbrace D^-_{i,j} \,|\,0<j<i<k\rbrace\cup \lbrace D_{i,0} \,|\,1<i<k\rbrace$.
With this notation, the action described in (\ref{eqaccionpar}) implies that, except for the cases in~(\ref{laexce1})--(\ref{laexce3}) below, each $\epsilon_l$ acts as the identity on the elements of $\mathcal{D}$:
\begin{eqnarray}
\epsilon_l D^+_{i,j} &=& -D^+_{i,j}, \text{ \; if $i=l-1$}; \label{laexce1}\\
\epsilon_l D^-_{i,j} &=& -D^-_{i,j}, \text{ \; if $j=l-1$}; \label{laexce2}\\
\epsilon_l D_{i,0} &=& -D_{i,0}, \text{ \; if $l=1$} \label{laexce3}.
\end{eqnarray}
Clearly, $\mathcal{D}$ forms a basis for $\mathbb{K}$ with inverse change of basis given by
\begin{eqnarray}
B_{i,j}&=&\frac{D^+_{i,j}+D^-_{i,j}+D_{i,0}+D_{j,0}}{2},\nonumber\\
B_{i,-j}&=&\frac{D^+_{i,j}-D^-_{i,j}+D_{i,0}+D_{j,0}}{2},\label{2esinvertible2}\\
B_{i,0}&=&D_{i,0},\nonumber
\end{eqnarray}
for $0<j<i<k$ (as in~(\ref{2esinvertible1}), the assumption on the invertibility of 2 is needed here). We leave to the reader the verification of the following multiplicative relations among the elements of $\mathcal{D}$:

\begin{lem}\label{relsmultsalgunasDs}
Let $R$ be a commutative ring with unit where $2$ is invertible, and suppose $n\geq 2$ even. The elements of $\mathcal{D}$ satisfy the following multiplicative relations:
$$
\mbox{For $0<j<i<r<k$, \hspace{8mm}}
\begin{cases}
&D^+_{r,i} D^+_{r,j}=D^-_{i,j}  D^-_{r,j} - D^+_{i,j}  D^-_{r,i} - D_{j,0}  D_{i,0} + 
 D_{j,0}  D_{r,0} - D_{i,0}  D_{r,0},\\
&D^+_{r,i} D^-_{r,j}=D^-_{i,j}  (D^+_{r,j} -  D^+_{r,i} ),\\
&D^-_{r,i} D^+_{r,j}=D^+_{i,j} ( D^+_{r,j} -D^+_{r,i} ),\\
&D^-_{r,i} D^-_{r,j}=-D^-_{i,j}  D^-_{r,i} + D^+_{i,j}  D^-_{r,j}.
\end{cases}
$$
$$
\mbox{For $0<i<r<k$, \hspace{14.3mm}}
\begin{cases}
&D^+_{r,i} D_{r,0}=-D_{i,0}D^+_{r,i},\\
&D^-_{r,i} D_{r,0}=-D_{i,0}D^-_{r,i}.\hspace{6.56cm}
\end{cases}
$$
\vspace{.5mm}
$$
\mbox{For $0<j<i<k$, \hspace{19.7mm} 
$(D^+_{i,j})^2\;=\;(D^-_{i,j})^2\;=\;(D_{i,0})^2\;=\;
D^+_{i,j}D^-_{i,j}\;=\;0.$\hspace{2.7cm}}
$$
\end{lem}

By repeated applications of Lemma \ref{relsmultsalgunasDs} we see that $\mathbb{K}^*$ is additively generated by products of the form 
\begin{equation}\label{productosdeDs}
D^{}_{i_1,j_1}\cdots D^{}_{i_r,j_r},
\end{equation} where 
\begin{equation}\label{condiciondeordenD}
i_l<i_{l'}\textnormal{ if }l<l'.
\end{equation} 
The set of these products is in bijective correspondence with the basis consisting of products $B_{i_1,j_1}\cdots B_{i_r,j_r}$ satisfying condition (\ref{condiciondeordenD}), and so by Nakayama's lemma the set of products of the form (\ref{productosdeDs}) satisfying (\ref{condiciondeordenD}) is an additive basis of the permanent cycles in $H^{*}(\Conf_{\mathbb{Z}_2}(\mathbb{R}^n-\lbrace 0 \rbrace,k-1);R)$.

\begin{rem} \label{laacciondelah}
The previous discussion and our description of 
$ \epsilon_{l}(D^{\pm}_{i,j}) $ in~(\ref{laexce1})--(\ref{laexce3})
easily yield that the map 
$ h^{\times(k-1)} $ 
in Lemma 3.3 acts on the permanent cycles in $ \mathbb{K}^m $ as multiplication by 
$ (-1)^m $.
\end{rem} 

Next we define elements which are clearly $(\mathbb{Z}_2)^k$-invariants; in fact we will show in Theorem \ref{invariantespar} below that they are multiplicative generators for all $(\mathbb{Z}_2)^k$-invariants. For $0< j<i<r<k$, put $I^+_{r,i,j}=D^+_{i,j}D^-_{r,i}$, $I^-_{r,i,j}=D^-_{i,j}D^-_{r,j},$ and for $1<j<i<k$, $\,I_{i,j,0}= D_{j,0}D_{i,0}$.
For $j>0$, we will sometimes write $I_{r,i,j}$ and $I_{r,i,-j}$ instead of $I^+_{r,i,j}$ and $I^-_{r,i,j}$ respectively. Accordingly, we will sometimes write $I^+_{i,j,0}$ or even $I^-_{i,j,0}$ as a substitute for $I_{i,j,0}$. Let $\mathcal{E}^+=\lbrace I^+_{r,i,j}\,|\,0<j<i<r<k \rbrace$, $\mathcal{E}^-=\lbrace I^-_{r,i,j}\,|\,0<j<i<r<k \rbrace$, $\mathcal{E}_0=\lbrace I_{i,j,0}\,|\,1<j<i<k \rbrace$ and $\mathcal{E}=\mathcal{E}^+\cup\mathcal{E}^-\cup\mathcal{E}_0$. We leave to the reader the (rather lengthy but straightforward) verification of the following result (which we have verified both by hand as well as by computer):

\begin{lem}\label{relsmultI} Let $R$ be a commutative ring with unit where $2$ is invertible. Suppose $n\geq 2$ even. The elements of $\mathcal{E}$ satisfy the relations listed below. Relations \ref{relacionrelsmultIa} through \ref{relacionrelsmultIc} express a product $I^{\pm}_{r,i,j} I^{\pm}_{s,a,b}$ with 
\begin{equation}\label{condicionesuno}
\mbox{$0\leq j<i<r<k$, \ $0\leq b<a<s<k$, \ $2\leq i$, \ $2\leq a$, \ and \ $r \leq s$}
\end{equation}
as a linear combination of such products satisfying in addition 
\begin{equation}\label{condicionesdos}
\mbox{$r<s$ \ and \ $a \not\in \{ r,i \} $.}
\end{equation}
Those relations are listed according to the several possibilities for the indices $r,i,j,s,a,$ and $b$ when they satisfy~(\ref{condicionesuno}) but not~(\ref{condicionesdos}). 
\begin{enumerate} 
\item \label{relacionrelsmultIa} $r=s$
\begin{enumerate} 
\item $j\neq 0 \neq b$
\begin{enumerate} 
\item $|\{i,j,a,b\}|=4$ (can assume $a<i$)
\begin{enumerate} 

\bigskip\item $b<a<j$: 
\begin{eqnarray*}
I^+_{r,i,j} I^+_{r,a,b}&=&I^+_{j,a,b}(I^-_{r,i,a}-I^+_{r,i,a}+I^+_{r,i,j})+(I_{i,b,0}-I^+_{i,j,a}-I_{i,j,0}-I_{j,b,0})I^+_{r,a,b},\\
I^-_{r,i,j} I^-_{r,a,b}&=&I^-_{j,a,b}I^-_{r,i,j}-I^+_{i,j,b}I^-_{r,a,b},\\
I^+_{r,i,j} I^-_{r,a,b}&=&I^-_{j,a,b}(I^+_{r,i,j}-I^+_{r,i,b})+(I^-_{i,j,b}-I^+_{i,j,b}-I_{j,b,0}+I_{i,b,0}-I_{i,j,0})I^-_{r,a,b},\\
I^-_{r,i,j}I^+_{r,a,b} &=&I^+_{j,a,b}I^-_{r,i,j}-I^+_{i,j,a}I^+_{r,a,b}.
\end{eqnarray*}

\smallskip\item $b<j<a$:
\begin{eqnarray*}
I^+_{r,i,j} I^+_{r,a,b}&=&(I^-_{a,j,b} - I^+_{a,j,b}+ I_{a,b,0} - I_{a,j,0}-I_{j,b,0})(I^+_{r,i,a}-I^-_{r,i,a} - I^+_{r,i,j}) \\
&& {}+ I^+_{i,j,b}(I^+_{r,a,j}-I^+_{r,a,b} )+ (I_{i,b,0} - I_{i,j,0}- I_{j,b,0})I^+_{r,a,b} ,\\ 
I^-_{r,i,j} I^-_{r,a,b}&=&{}-I^-_{a,j,b}I^-_{r,i,j} - I^+_{i,j,b}I^-_{r,a,b},\\
I^+_{r,i,j} I^-_{r,a,b}&=&I^-_{a,j,b}(I^+_{r,i,b} - I^+_{r,i,j}-I^-_{r,i,b} ) +(I_{i,b,0}- I^+_{i,j,b} - I_{j,b,0}  - I_{i,j,0})I^-_{r,a,b},\\
I^-_{r,i,j}I^+_{r,a,b} &=&I^+_{i,j,b}(I^+_{r,a,j}- I^+_{r,a,b} ) + (I_{j,b,0} - I_{a,b,0} + I_{a,j,0}-I^-_{a,j,b} + I^+_{a,j,b})I^-_{r,i,j}.
\end{eqnarray*}

\smallskip\item $j<b<a$:
\begin{eqnarray*}
I^+_{r,i,j} I^+_{r,a,b}&=&(I^+_{a,b,j}-I^-_{a,b,j}-I_{a,j,0}+I_{a,b,0}+I_{b,j,0})(I^+_{r,i,a}- I^-_{r,i,a} - I^+_{r,i,j} )\\
 && {}+I^-_{i,b,j}(I^+_{r,a,j} - I^+_{r,a,b})  +(I_{i,b,0}- I_{i,j,0}+I_{b,j,0})I^+_{r,a,b},\\
I^-_{r,i,j} I^-_{r,a,b}&=&{}-I^-_{i,b,j}I^-_{r,a,b} - I^+_{a,b,j}I^-_{r,i,j},\\
I^+_{r,i,j} I^-_{r,a,b}&=&I^+_{a,b,j}( I^+_{r,i,b}- I^-_{r,i,b} - I^+_{r,i,j}) + (I_{b,j,0} - I_{i,j,0}+ I_{i,b,0}-I^-_{i,b,j})I^-_{r,a,b},\\
I^-_{r,i,j}I^+_{r,a,b} &=& I^-_{i,b,j}(I^+_{r,a,j} - I^+_{r,a,b}) +(I^-_{a,b,j}- I^+_{a,b,j} - I_{b,j,0} + I_{a,j,0} - I_{a,b,0})I^-_{r,i,j}.
\end{eqnarray*}
\end{enumerate}

\bigskip\item $|\{i,j,a,b\}|=3$ (can assume $a\leq i$)
\begin{enumerate} 

\medskip\item $a=i$ (can assume $b<j$): \hspace{-.6mm}
$ 
I^{\pm}_{r,i,j}I^{\pm}_{r,i,b} =0.
$ 

\smallskip\item $a=j$: \hspace{35mm}
$ 
I^{\pm}_{r,i,j}I^{\pm}_{r,j,b} =0.
$ 

\smallskip\item $b=i$ is impossible.

\smallskip\item $b=j$: \hspace{35.4mm}
$ 
I^{\pm}_{r,i,j}I^{\pm}_{r,a,j} =0.
$ 
\end{enumerate}

\bigskip
\item $|\{i,j,a,b\}|=2$: \hspace{26.5mm}
$ 
I^{\pm}_{r,i,j} I^{\pm}_{r,i,j}=0.
$ 
\end{enumerate}

\bigskip\item $j=0\neq b$ (the case $j\neq 0=b$ is symmetric)
\begin{enumerate} 
\item $|\{i,a,b\}|=3$
\begin{enumerate} 

\bigskip\item $i<b<a$:
\begin{eqnarray*}
I^{}_{r,i,0}I^+_{r,a,b} &=&I^{}_{b,i,0}I^+_{r,a,b},\\
I^{}_{r,i,0}I^-_{r,a,b} &=&I^{}_{b,i,0}I^-_{r,a,b}.
\end{eqnarray*}

\bigskip\item $b<i<a$ or $b<a<i$:
\begin{eqnarray*}
I_{r,i,0}I^+_{r,a,b} &=&-I_{i,b,0}I^+_{r,a,b},\\
I_{r,i,0}I^-_{r,a,b}&=&-I_{i,b,0}I^-_{r,a,b}.
\end{eqnarray*}
\end{enumerate}

\smallskip
\item $|\{i,a,b\}|=2$
\begin{enumerate} 

\bigskip\item $a=i$: \hspace{1.5cm}
$ 
I_{r,i,0}I^{\pm}_{r,i,b} =0.
$ 

\item $b=i$:  \hspace{1.5cm}
$ 
I_{r,i,0}I^{\pm}_{r,a,i} =0.
$ 
\end{enumerate}
\end{enumerate}

\bigskip
\item $j=0=b$ (can assume $a\leq i$): \hspace{1mm}
$ 
I_{r,i,0}I_{r,a,0} =0.
$ 
\end{enumerate}

\bigskip\item $a=r<s$
\begin{enumerate} 
\item $j\neq 0 \neq b$
\begin{enumerate} 
\item $|\{i,j,b\}|=3$
\begin{enumerate} 

\medskip\item $j<i<b$:
\begin{eqnarray*}
I^+_{r,i,j}I^+_{s,r,b} &=&I^+_{b,i,j}(I^+_{s,r,b}-I^+_{s,r,i}),\\
I^-_{r,i,j}I^-_{s,r,b} &=&I^-_{b,i,j}I^-_{s,r,b} + I^-_{r,i,j}I^+_{s,b,j},\\
I^+_{r,i,j}I^-_{s,r,b} &=&I^+_{b,i,j}I^-_{s,r,b} + I^+_{r,i,j}I^+_{s,b,i},\\
I^-_{r,i,j}I^+_{s,r,b} &=&I^-_{b,i,j}(I^+_{s,r,b}-I^+_{s,r,j}).
\end{eqnarray*}
\item $j<b<i$:
\begin{eqnarray*}
I^+_{r,i,j}I^+_{s,r,b} &=&(I^-_{i,b,j} - I^+_{i,b,j} - I_{b,j,0} + I_{i,j,0} - I_{i,b,0})(I^+_{s,r,i}-I^+_{s,r,b}),\\
I^-_{r,i,j}I^-_{s,r,b} &=& I^-_{r,i,j}I^+_{s,b,j}-I^-_{i,b,j}I^-_{s,r,b} ,\\
I^+_{r,i,j}I^-_{s,r,b} &=&(I^+_{r,i,j} - I^+_{r,i,b})I^+_{s,b,j} + (-I^-_{i,b,j} + I^+_{i,b,j} + I_{b,j,0} - I_{i,j,0} + I_{i,b,0})I^-_{s,r,b},\\
I^-_{r,i,j}I^+_{s,r,b} &=&I^-_{i,b,j}(I^+_{s,r,j} - I^+_{s,r,b}).
\end{eqnarray*}

\medskip\item $b<j<i$:
\begin{eqnarray*}
I^+_{r,i,j}I^+_{s,r,b} &=&(I^-_{i,j,b} - I^+_{i,j,b} - I_{j,b,0} + I_{i,b,0} - I_{i,j,0})(I^+_{s,r,b} - I^+_{s,r,i}),\\
I^-_{r,i,j}I^-_{s,r,b} &=&I^-_{r,i,j}I^-_{s,j,b} - I^+_{i,j,b}I^-_{s,r,b},\\
I^+_{r,i,j}I^-_{s,r,b} &=&(I^+_{r,i,j}-I^+_{r,i,b} )I^-_{s,j,b} + (I^-_{i,j,b} - I^+_{i,j,b} - I_{j,b,0} + I_{i,b,0} - I_{i,j,0})I^-_{s,r,b},\\
I^-_{r,i,j}I^+_{s,r,b} &=&I^+_{i,j,b}(I^+_{s,r,j}-I^+_{s,r,b} ).
\end{eqnarray*}

\end{enumerate}

\bigskip\item $|\{i,j,b\}|=2$
\begin{enumerate} 

\medskip\item $b=i$: \hspace{1.5cm}
$ 
I^{\pm}_{r,i,j}I^{\pm}_{s,r, i} =0.
$ 

\medskip\item $b=j$: \hspace{1.5cm}
$ 
I^{\pm}_{r,i,j}I^{\pm}_{s,r,j} =0.
$ 
\end{enumerate}
\end{enumerate}

\bigskip\item $j= 0 \neq b$
\begin{enumerate} 
\item $|\{i,b\}|=2$
\begin{enumerate} 

\medskip\item $i<b$:
\begin{eqnarray*}
I_{r,i,0}I^+_{s,r,b} &=&I_{b,i,0}I^+_{s,r,b},\\
I_{r,i,0}I^-_{s,r,b} &=&I_{b,i,0}I^-_{s,r,b}.
\end{eqnarray*}
\item $b<i$:
\begin{eqnarray*}
I_{r,i,0}I^+_{s,r,b} &=&-I_{i,b,0}I^+_{s,r,b},\\
I_{r,i,0}I^-_{s,r,b} &=&-I_{i,b,0}I^-_{s,r,b}.
\end{eqnarray*}
\end{enumerate}
\item $|\{i,b\}|=1$: \hspace{.7cm}
$ 
I_{r,i,0}I^{\pm}_{s,r,i} \;=\;0.
$ 
\end{enumerate}
\bigskip

\item $j\neq 0 =b$:
\begin{eqnarray*}
\hspace{9.7mm}I^+_{r,i,j}I_{s,r,0} &=&I^+_{r,i,j}I_{s,j,0},\\
I^-_{r,i,j}I_{s,r,0} &=&I^-_{r,i,j}I_{s,j,0}.
\end{eqnarray*}

\bigskip
\item $j=0=b$: \hspace{2.6cm}
$ 
I_{r,i,0}I_{s,r,0} =0.
$ 
\end{enumerate}

\bigskip\item\label{relacionrelsmultIc} $a=i<r<s$
\begin{enumerate} 
\item $j\neq 0 \neq b$
\begin{enumerate} 
\item $|\{j,b\}|=2$
\begin{enumerate} 

\medskip\item $j<b$:
\begin{eqnarray*}
I^+_{r,i,j}I^+_{s,i,b} &=&(I^-_{i,b,j} - I_{b,j,0} + I_{i,j,0} - I_{i,b,0} - I^+_{i,b,j})I^-_{s,r,i},\\
I^-_{r,i,j}I^-_{s,i,b} &=&I^-_{r,b,j}I^-_{s,i,b} + I^-_{r,i,j}I^+_{s,b,j},\\
I^+_{r,i,j}I^-_{s,i,b} &=&(I^+_{r,i,j} - I^+_{r,i,b})I^+_{s,b,j},\\
I^-_{r,i,j}I^+_{s,i,b} &=&I^-_{r,b,j}(I^+_{s,i,b}-I^+_{s,i,j}).
\end{eqnarray*}

\medskip\item $b<j$:
\begin{eqnarray*}
I^+_{r,i,j}I^+_{s,i,b} &=&(-I^-_{i,j,b} + I^+_{i,j,b} + I_{j,b,0} - I_{i,b,0} + I_{i,j,0})I^-_{s,r,i},\\
I^-_{r,i,j}I^-_{s,i,b} &=&I^-_{r,i,j}I^-_{s,j,b} + I^+_{r,j,b}I^-_{s,i,b},\\
I^+_{r,i,j}I^-_{s,i,b} &=&(I^+_{r,i,j}-I^+_{r,i,b})I^-_{s,j,b},\\
I^-_{r,i,j}I^+_{s,i,b} &=&I^+_{r,j,b}(I^+_{s,i,b} - I^+_{s,i,j}).
\end{eqnarray*}
\end{enumerate}

\medskip\item $|\{j,b\}|=1$: \hspace{.5cm}
$ 
I^{\pm}_{r,i,j}I^{\pm}_{s,i,j} =0.
$ 
\end{enumerate}
\item $j=0 \neq b$:
\begin{eqnarray*}
I_{r,i,0}I^+_{s,i,b} &=&I_{r,b,0}I^+_{s,i,b},\\
I_{r,i,0}I^-_{s,i,b} &=&I_{r,b,0}I^-_{s,i,b}.
\end{eqnarray*}
\item $j\neq 0=b$:
\begin{eqnarray*}
I^+_{r,i,j}I_{s,i,0} &=&I^+_{r,i,j}I_{s,j,0},\\
I^-_{r,i,j}I_{s,i,0} &=&I^-_{r,i,j}I_{s,j,0}.
\end{eqnarray*}
\item $j=0=b$: \hspace{3.95cm}
$ 
I_{r,i,0}I_{s,i,0} =0.
$ 
\end{enumerate}

\medskip\item\label{relmultI1} $0\leq j<t<s<i<r<k$:
\begin{eqnarray*}
 I^-_{i,s,j}I^-_{r,t,j}&=&I^-_{s,t,j}I^-_{r,i,j},\\
I^-_{i,t,j} I^-_{r,s,j} &=&-I^-_{s,t,j}I^-_{r,i,j}.
\end{eqnarray*}
\item\label{relmultI2} $0<j<i<t<s<r<k$:
\begin{eqnarray*}
I^-_{s,t,i}I^+_{r,i,j}&=&I^+_{t,i,j}I^-_{r,s,i},\\
I^+_{s,i,j} I^-_{r,t,i} &=&-I^+_{t,i,j}I^-_{r,s,i}.
\end{eqnarray*}
\end{enumerate}
\end{lem}

\begin{rem}
Relations \ref{relmultI1} and \ref{relmultI2} in the previous lemma are  not a consequence of the multiplicative relations among the elements in $\mathcal{D}$, but rather a consequence of the fact that, in some cases, there are different alternatives for associating four $D$'s to form a product of two $I$'s. 
\end{rem}
The relations in Lemma \ref{relsmultI} imply that every product 
\begin{equation}\label{productosdeIs}
\begin{array}{cccc}
&I_{r_1,i_1,j_1}\cdots I_{r_m,i_m,j_m} &\textnormal{ with }|j_l|<i_l<r_l<k\textnormal{ and }1<i_l\textnormal{ for }l=1,\dots,m
\end{array}
\end{equation}
can be written as a linear combination of products of the form (\ref{productosdeIs}) satisfying in addition
\begin{eqnarray}
&r_l < r_{l'} \mbox{ if } l < l',&\label{condiciondeordendeIs} \\
&\mbox{the sets } \{i_l,r_l\}, \mbox{ with } 1\leq l\leq m,
\mbox{ are pairwise disjoint,} & 
\label{condiciondeinterseccion} \\
&\mbox{if } j_a=j_b\leq0, \mbox{ say with } r_a < r_b, \mbox{ then in fact } r_a < i_b,& \label{condicionmenosmenos} \\
&\mbox{if } j_a>0 \mbox{ and } i_a=-j_b, \mbox{ then in fact } r_a<i_b.& \label{condicionmasmenos}
\end{eqnarray}
\begin{thm}\label{invariantespar}
Suppose $R$ is a commutative ring with unit where $2$ is invertible. For $n\geq 2$ even, the $(\mathbb{Z}_2)^k$-invariants in $\mathbb{K}^*\subseteq H^{*}(\Conf_{\mathbb{Z}_2}(\mathbb{R}^n-\lbrace 0 \rbrace,k-1);R)$ are multiplicatively generated by the set $\mathcal{E}$. In fact, an additive basis for the invariants is given by all products of the form (\ref{productosdeIs}) satisfying (\ref{condiciondeordendeIs})--(\ref{condicionmasmenos}).
\end{thm}
\begin{proof}
Suppose $m$ odd and let $x\in \mathbb{K}^m$ be an invariant. By Remark \ref{laacciondelah}, we have that $x=\epsilon_1 x = -\epsilon_2\cdots \epsilon_k x=-x$, so $x=0$.

Suppose now $m$ is even and, as above, let $x \in \mathbb{K}^*$ be an invariant. Write
\[
x=\sum a_I D^{}_{i_1,j_1}\cdots D^{}_{i_m,j_m},
\]
where each $a_I$ is non-zero  and the summation runs over all multi-indices
$I=((i_1,j_1),\ldots,(i_m,j_m))$ such that $|j_l|<i_l$ for $l=1,\ldots,m$ and $i_l<i_{l'}$ if $l<l'$. Recall that an $\epsilon_{l}$ sends each monomial $D^{}_{i_1,j_1}\cdots D^{}_{i_m,j_m}$ to a multiple of itself, therefore, since $2$ is invertible, each $D^{}_{i_1,j_1}\cdots D^{}_{i_m,j_m}$ is invariant. Fix $I$, and consider the corresponding monomial $z=D^{}_{i_1,j_1}\cdots D^{}_{i_m,j_m}$. Note that the action of $\epsilon_1$ on $z$ implies that an even number of factors in $z$ are of the  form $ D_{i,0} $. Further, such factors can be matched in pairs to yield a product of the form 
\begin{equation}\label{condicioncero}
I_{i_1,j_1,0} I_{i_2,j_2,0} \cdots \mbox{ where } j_1 < i_1 < j_2 < i_2 < \cdots.
\end{equation}
Likewise, for each $l$ between $2$ and $k$, we have two possibilities:
\begin{enumerate}
\item There is no factor $D^+_{l-1,*}$ in $z$ (e.g. if $l=2$). In this case, there is an even number of factors of the form $D^-_{*,l-1}$, because otherwise we would have $z=\epsilon_{l}z=-z$.
\item There is (exactly) one factor $D^+_{l-1,*}$ in $z$. In this case, there is an odd number of factors of the form $D^-_{*,l-1}$ in $z$.
\end{enumerate}
The first case allows us to associate products of the form $D^-_{i,j}D^-_{r,j}$, and the second allows us to associate a product of the form $D^+_{i,j}D^-_{r,i}$ and products of the form $D^-_{i,j}D^-_{r,j}$.
Further, just as with~(\ref{condicioncero}), the new matchings can be done so to yield, together with~(\ref{condicioncero}), a unique expression of $D_{i_1,j_1} \cdots D_{i_m,j_m}$ as a product of the form (\ref{productosdeIs}) satisfying in addition (\ref{condiciondeordendeIs})--(\ref{condicionmasmenos}). \smallskip

The above analysis shows that the $(\mathbb{Z}_2)^k$-invariants in $\mathbb{K}^*$ are generated by the products of the form (\ref{productosdeIs}) satisfying in addition (\ref{condiciondeordendeIs})--(\ref{condicionmasmenos}). In fact, this is a basis, since such generators are a subset of the additive basis of $\mathbb{K}^*$ given by the products (\ref{productosdeDs}) satisfying (\ref{condiciondeordenD}).
\end{proof}
We arrive at the complete description of the invariants for the case $n$ even.
\begin{thm}\label{isoinvariantespar}
Let $R$ be a commutative ring with unit where $2$ is invertible. For $n$ even, there is an $R$-algebra isomorphism
\[
H^*(\Conf(\mathbb{R}\mathrm{P}^n,k);R)=H^{*}(\Conf_{\mathbb{Z}_2}(S^n,k);R)^{(\mathbb{Z}_2)^k} \cong \Lambda(\omega)\otimes R[\mathcal{E}]/\mathcal{J}, 
\]
where $\mathcal{J}$ is the ideal generated by the relations in Lemma \ref{relsmultI}.
\end{thm}

\begin{proof}
Since we are assuming that $2$ is a unit in $R$, the isomorphism of Theorem \ref{2.3} reduces to 
\[
H^{*}(\Conf_{\mathbb{Z}_2}(S^n,k);R)\cong \Lambda(\omega)\otimes R[\mathcal{B}]/J. 
\]
Recall from Corollary \ref{acciontotalnpar} that $\omega$ is fixed by the action of $(\mathbb{Z}_2)^k$. The result follows from Theorem \ref{invariantespar}, which implies that the subring of $\mathbb{Z}_2$-invariants in the tensor factor $ R [ \mathcal{B} ] / J $ has the presentation 
$R [ \mathcal{E} ] / \mathcal{J} $.
\end{proof}

\section{Punctured real projective spaces}\label{section4}
In this section we keep the assumption that $n$ is an integer greater than or equal to 2, and that $R$ is a commutative ring with unit where $2$ is invertible (the latter hypothesis is needed for the same reasons as in Section~\ref{secndinvtes}). 

In terms of the homeomorphism at the end of Remark~\ref{renotacion}, the $(\mathbb{Z}_2)^k$-fold covering projection $\rho'_{n,k}$ in~(\ref{lascovrgsdosdeellas}) takes the form
$
\Conf_{\mathbb{Z}_2}(\mathbb{R}^n-\{0\},k)
\longrightarrow \Conf(\mathbb{R}\mathrm{P}^n-\star,k) 
$
which, given that $2$ is invertible in $R$, induces an isomorphism
\begin{equation}\label{jduhkcks}
H^{*}(\Conf(\mathbb{R}\mathrm{P}^n-\star,k);R)\cong
H^*(\Conf_{\mathbb{Z}_2}(\mathbb{R}^n-\{0\},k);R)^{(\mathbb{Z}_2)^{k}}.
\end{equation}

We can reuse the notation of the previous section by noticing that the action in~(\ref{jduhkcks}) corresponds to that of the subgroup $ (\mathbb{Z}_2) ^{k}=\langle \epsilon_2,\ldots,\epsilon_{k+1}\rangle < (\mathbb{Z}_2)^{k+1}$. This means that, when computing invariants, we just have to ignore the action of $\epsilon_1$. Lemma~\ref{lema3.3} and Theorem~\ref{isoinvariantesimpar} then yield:
\begin{thm}\label{rpnponchadoimpar}
Let $R$ be a commutative ring with unit where $2$ is invertible. For $n\geq 2$ odd, there is an isomorphism
\[
H^{*}(\Conf(\mathbb{R}\mathrm{P}^n-\star,k);R)\cong H^*(\Conf_{\mathbb{Z}_2}(\mathbb{R}^n-\{0\},k);R)^{(\mathbb{Z}_2)^{k}}  \cong  R[\mathcal{C^+}]/\mathcal{K}. 
\] 
of $R$-algebras.
\end{thm}
Note that the role of the parameter $k$ in Theorem 4.5 changes here to $k+1$. For instance, the generators $C^+_{i,j}$ of $\mathcal{C}^+$ are now defined for $0 < j < i \leq k$.

\medskip
Our next goal is to compute the invariant elements in~(\ref{jduhkcks})
when $n$ is even (Theorem~\ref{refparbasespult}). It is natural to expect more invariants than those found in Theorem~\ref{invariantespar}. For one, as noted above, the currently acting group is smaller. Besides, we have to consider invariants in the whole of $H^*(\Conf_{\mathbb{Z}_2}(\mathbb{R}^n-\{0\},k);R)$ and not just in the permanent cycles $\mathbb{K}^*$.

We start by noticing that the considerations following Remark~\ref{sparsenessreference} and Lemma~\ref{relsmultsalgunasDs} show that $H^*(\Conf_{\mathbb{Z}_2}(\mathbb{R}^n-\{0\},k);R)$ is multiplicatively generated by $A_{1,0}$ and the elements $D_{i,j}$ with $|j|<i\leq k$ subject only to the relations in Lemma~\ref{relsmultsalgunasDs} together with $A_{1,0}^2=0$. Further, an additive basis is given by all products of the form (\ref{productosdeDs}) and products of the form
\begin{equation}
A_{1,0}D^{}_{i_1,j_1}\cdots D^{}_{i_r,j_r}
\end{equation}
satisfying (\ref{condiciondeordenD}). In fact, all the elements in the set
$$
\mathcal{E}'=\{I^+_{r,i,j} \,|\, 0<j<i<r\leq k\}\cup
\{I^-_{r,i,j} \,|\, 0<j<i<r\leq k\}\cup
\{D_{i,0}\,|\,1<i\leq k\}\cup\{A_{1,0}\}
$$
are clearly $(\mathbb{Z}_2)^k$-invariant. Before showing these generate all other invariants, we describe their multiplicative relations. First of all, while all relations in Lemma~\ref{relsmultI} are clearly inherited (with upper bound $k+1$ instead of $k$ for indices $r,i,j$), the relations involving terms of the form $I_{i,j,0}$ are evidently not in the most primitive form. Instead, we have the following easy-to-check relations:
\begin{lem}\label{relsmultIyD0}Let $R$ be a commutative ring with unit where $2$ is invertible, and suppose $n\geq 2$ even. Then
\begin{eqnarray*}
I^+_{r,i,j}D_{i,0} &=&I^+_{r,i,j}D_{j,0},\\
I^+_{r,i,j}D_{r,0} &=&I^+_{r,i,j}D_{j,0},\\
I^-_{r,i,j}D_{i,0} &=&I^-_{r,i,j}D_{j,0},\\
I^-_{r,i,j}D_{r,0} &=&I^-_{r,i,j}D_{j,0},
\end{eqnarray*}
for $0< j<i<r\leq k$.
\end{lem}
Therefore, any product of the form 
\begin{equation}\label{productodeDeIs}
\begin{array}{ccll}
&D_{s_1,0}\cdots D_{s_{m'},0} I_{r_1,i_1,j_1}\cdots I_{r_m,i_m,j_m} &\textnormal{ with }0<|j_l|<i_l<r_l\leq k\textnormal{ for }l=1,\dots,m\\&&\textnormal{ and } 1<s_l\leq k \textnormal{ for }l=1,\ldots,m'
\end{array}
\end{equation}
or
\begin{equation}\label{productodeA10DeIs}
\begin{array}{ccll}
&A_{1,0}D_{s_1,0}\cdots D_{s_{m'},0} I_{r_1,i_1,j_1}\cdots I_{r_m,i_m,j_m} &\textnormal{ with }0<|j_l|<i_l<r_l\leq k\textnormal{ for }l=1,\dots,m\\&&\textnormal{ and } 1<s_l\leq k\textnormal{ for }l=1,\ldots,m'
\end{array}
\end{equation}
can be written as a linear combination of products of the form (\ref{productodeDeIs}) or (\ref{productodeA10DeIs}) satisfying 
\begin{equation}\label{condiciondeordenD01}
s_l<s_{l'}\textnormal{ if }l<l',
\end{equation}
\begin{equation}\label{condiciondeordenD02}
s_{l}\notin  \{r_{1},\ldots,r_{m}\}\cup\{i_1,\ldots,i_m\}\textnormal{ for }l=1,\ldots,m',
\end{equation}
as well as conditions (\ref{condiciondeordendeIs})--(\ref{condicionmasmenos}).
\begin{thm}\label{refparbasespult}
Let $R$ be a commutative ring with unit where $2$ is invertible. For $n\geq 2$ even, the $(\mathbb{Z}_2)^k$-invariants in $H^{*}(\Conf_{\mathbb{Z}_2}(\mathbb{R}^n-\lbrace 0 \rbrace,k);R)$ are multiplicatively generated by the set $\mathcal{E}'$. Moreover, an additive basis is given by products of the form (\ref{productodeDeIs}) together with products of the form (\ref{productodeA10DeIs}) all of which satisfy (\ref{condiciondeordenD01}), (\ref{condiciondeordenD02}), and (\ref{condiciondeordendeIs})--(\ref{condicionmasmenos}).
\end{thm}
\begin{proof}
The proof is almost the same as the proof of Theorem \ref{invariantespar}, except for two differences:
\begin{enumerate}
\item We ignore the action of $\epsilon_1$. This, however, only removes the condition of having an even number of terms $D_{i,0}$, and now we can just associate all terms $D_{i,0}$.
\item We add $A_{1,0}$ as a potential factor to all monomials. This does not affect our previous proof, because $A_{1,0}$ is already an invariant.
\end{enumerate}\parskip-17pt
\end{proof}
We thus get:
\begin{thm}\label{rpnponchadopar}
Let $R$ be a commutative ring with unit where $2$ is invertible. For $n\geq 2$ even, there is an $R$-algebra isomorphism
\[
H^{*}(\Conf(\mathbb{R}\mathrm{P}^n-\star,k);R)\cong H^*(\Conf_{\mathbb{Z}_2}(\mathbb{R}^n-\{0\},k);R)^{(\mathbb{Z}_2)^{k}}  \cong R[\mathcal{E}']/\mathcal{J'}, 
\]
where $\mathcal{J}'$ is the ideal generated by the relations in Lemma \ref{relsmultI} not involving a term $I_{i,j,0}$ together with the relations of Lemma \ref{relsmultIyD0} and the relation $A_{1,0}^2=0$.
\end{thm}

\section{Lusternik-Schnirelmann category and topological complexity}\label{sectioncatTC}
The results in the previous sections are now used to study the category and all the higher topological complexities of the auxiliary orbit configuration spaces $\Conf_{\mathbb{Z}_2}(\mathbb{R}^n-\{0\},k)$. 

\smallskip
Unless explicitly noted otherwise, the hypothesis $n>2$ will be in force throughout this section. Yet, since at the end of the section we include some observations relevant to the more general case $n\geq2$, we will keep a careful track of the usage of the hypothesis $n>2$. Further, in what follows, any reference to cohomology in the form $H^*(X)$ will implicitly use integer coefficients.

\smallskip
We need the following estimates (given in~\cite[Proposition~1.5 and Theorem~1.50]{Cornea} for the Lusternik-Schnirelmann category, and in~\cite[Theorem~3.9]{MR3252053} for the higher topological complexities):

\begin{prop}
\label{lascotas}
Let $X$ be a $(c-1)$-connected space ($c\geq1$) having the homotopy type of a cell complex of dimension $d$. Choose $\star\in X$ and let $\delta_1\colon\star\to X$ denote the inclusion map. For $s\geq2$ let $\delta_s\colon X\to X^s$ stand for the iterated diagonal. Finally, let $M$ be a module over a commutative unitary ring $R$. For $s\geq 1$ assume there are cohomology classes $c_1,\ldots,c_\ell \in H^*(X^s;M)$ such that $\delta_s^*(c_i)=0$ for all $i$, and such that $0\neq c_1\cdots c_\ell\in H^*(X^s;M^{\otimes \ell})$. Then $\ell\leq\TC_s(X)\leq\frac{sd}{c}.$
\end{prop}

There is a version of the previous result where the cohomology groups are allowed to have twisted coefficients. But for our purposes it suffices to work with cohomology groups with trivial coefficients. Note that the condition $\delta^*_1(c_i)=0$ in Proposition~\ref{lascotas} holds precisely when $c_i$ is a positive dimensional class (since $X$ is assumed to be a 0-connected space). On the other hand, for $s\geq2$, 
an element $c_i$ as in Proposition~\ref{lascotas} is called an $s$-th zero-divisor. The largest integer $\ell$ for which there are $s$-th zero-divisors $c_1,\ldots,c_\ell$ as in Proposition~\ref{lascotas} (for all possible $R$-modules $M$ and rings $R$) is called the  $s$-th zero-divisor cup-length of $X$, and is denoted by $\zcl_s(X)$. If we want to refer to the maximal such $\ell$ when the coefficients $M$ are restricted to be some fixed ring $R$, then we will use the notation $\zcl_s^R(X)$.

\medskip
We study these concepts when $X$ is the orbit configuration space $\Conf_{\mathbb{Z}_2}(\mathbb{R}^n-\{0\},k)$. The homotopy exact sequences associated to the fibrations in~(\ref{fib2}) inductively yield that $\Conf_{\mathbb{Z}_2}(\mathbb{R}^n-\{0\}, k)$ is $(n-2)$-connected. Further, from~(\ref{isomodulos}), the integral cohomology of this space is torsion-free, and vanishes above dimension $k(n-1)$. By  \cite[Section 4.C]{Hatcher}, $\Conf_{\mathbb{Z}_2}(\mathbb{R}^n-\{0\}, k)$ is an $(n-2)$-connected space (thus simply connected, in view of our general hypothesis $n>2$) having the homotopy type of a cell complex of dimension at most $k(n-1)$. (For the purposes of Corollaries~\ref{forzudo1} and~\ref{highertc} below, this is the only place where we use the hypothesis $n>2$.) Proposition~\ref{lascotas} then yields 
\begin{equation}\label{losinvariantesdegenero}
\zcl_s(\Conf_{\mathbb{Z}_2}(\mathbb{R}^n-\{0\}, k))\leq\TCs(\Conf_{\mathbb{Z}_2}(\mathbb{R}^n-\{0\}, k))\leq sk
\end{equation}
for $s\geq1$. Both estimates in~(\ref{losinvariantesdegenero}) are in fact sharp for $s=1$, as the product $A_{1,0}\cdots A_{k,0}\in H^*(\Conf_{\mathbb{Z}_2}(\mathbb{R}^n-\{0\}, k))$ is non-zero. We thus get:
\begin{cor}\label{forzudo1}
For $n>2$, $\cat(\Conf_{\mathbb{Z}_2}(\mathbb{R}^n-\{0\}, k))=k$.
\end{cor}

In order to approach the $s$-th zero-divisor cup-length of $\Conf_{\mathbb{Z}_2}(\mathbb{R}^n-\{0\},k)$ for $s\geq2$ we start by noticing that the rule $A'_{i,j}\mapsto A_{i-1,j-1}$, $1\leq j<i\leq k+1$, determines a ring monomorphism 
$$
H^*(\Conf(\mathbb{R}^n,k+1))\hookrightarrow H^*(\Conf_{\mathbb{Z}_2}(\mathbb{R}^n-\{0\},k)).
$$
Since these rings are torsion-free, we also get ring monomorphisms
$$
H^*(\Conf(\mathbb{R}^n,k+1))^{\otimes s}\hookrightarrow H^*(\Conf_{\mathbb{Z}_2}(\mathbb{R}^n-\{0\},k))^{\otimes s}
$$
for any $s\geq2$. Consequently, $\zcl_s(\Conf_{\mathbb{Z}_2}(\mathbb{R}^n-\{0\},k))$ is bounded from below by $\zcl_s^{\mathbb{Z}}(\Conf(\mathbb{R}^n,k+1))$, the latter of which, as shown in~\cite[Proposition~4.2]{GG15}, is given by $sk-1+\delta_n$ where $\delta_n=\mbox{Odd}(n)$, i.e. $\delta_n=0$ if $n$ is even, otherwise $\delta_n=1$. Consequently,~(\ref{losinvariantesdegenero}) specializes to:

\begin{cor}\label{highertc}
Let $n>2$. Then $\TCs(\Conf_{\mathbb{Z}_2}(\mathbb{R}^n-\{0\}, k))=sk$ if $n$ is odd, whereas, if $n$ is even, $\TCs(\Conf_{\mathbb{Z}_2}(\mathbb{R}^n-\{0\}, k))\in\{sk-1,sk\}$.
\end{cor}

At the end of this section we describe our (inconclusive) efforts to resolve the gap in Corollary~\ref{highertc} by one unit for the value of $\TCs(\Conf_{\mathbb{Z}_2}(\mathbb{R}^n-\{0\}, k))$ when $n$ is even. Here we prove that the lower bound in~(\ref{losinvariantesdegenero}) is always sharp:

\begin{thm}\label{tceszcls}
Let $n>2$. Then $\TC_s(\Conf_{\mathbb{Z}_2}(\mathbb{R}^n-\{0\}, k))=\zcl_s(\Conf_{\mathbb{Z}_2}(\mathbb{R}^n-\{0\}, k))$.
\end{thm}

The method of proof of Corollary~\ref{highertc} yields Theorem~\ref{tceszcls} when $n$ is odd. We prove Theorem~\ref{tceszcls} for $n$ even using the main idea in the proof of~\cite[Theorem~4.1]{GG15}. We give full details below for the sake of completeness.

\medskip
For a fibration $p:E\to B$ with fiber $F$, let $p(\ell):E(\ell)\to B$ be the ($\ell+1$)-th fiberwise join power of~$p$. This is a fibration with fiber $F^{\star(\ell+1)}$, the $(\ell+1)$-iterated join of $F$ with itself. It is well known that, for $B$ paracompact, a necessary and sufficient condition for having $\secat(p)\leq \ell$ is that $p(\ell)$ admits a global (continuous) section. Thus, the following result---a direct generalization of~\cite[Theorem~1]{Sch58}---gives a useful cohomological identification of the first obstruction for multi-sectioning~$p$.

\begin{lem}\label{classabstrnew}
Let $p:E\to B$ be a fibration with fiber $F$ whose base $B$ is a CW complex. Assume $p$ admits a section $\phi$ over the $k$-skeleton $B^{(k)}$ of $B$ for some $k\geq1$. If $F$ is $k$-simple and the obstruction cocycle to the extension of $\phi$ to $B^{(k+1)}$ lies in the cohomology class
$
\eta\in H^{k+1}(B;\{\pi_k(F)\}),
$ 
then $p(\ell)$ admits a section over $B^{(k+1)(\ell+1)-1}$ whose obstruction cocycle to extending to $B^{(k+1)(\ell+1)}$ belongs to the cohomology class
\begin{equation*}\label{potencia}
\eta^{\ell+1}\in H^{(k+1)(\ell+1)}(B;\{\pi_{k\ell+k+\ell}(F^{\star(\ell+1)})\}).
\end{equation*}
Here $\eta^{\ell+1}$ denotes the image of the $(\ell+1)$-fold cup power of $\eta$ under the $\pi_1(B)$-homomorphism of coefficients $\pi_k(F)^{\otimes (\ell+1)}\to \pi_{k\ell+k+\ell}(F^{\star(\ell+1)})$ given by the iterated join of homotopy classes.
\end{lem}

Lemma~\ref{classabstrnew} requires, of course, the use of cohomology with possibly twisted coefficients. In our application, both $B$ and $F$ will be simply connected (as a consequence of the second use of our general hypothesis $n>2$), so cohomology will always have simple coefficients, and the simplicity of $F$ will come for free.

\begin{proof}[Proof of Theorem~\ref{tceszcls}]
We prepare the grounds with two simplification hypotheses. First of all, as we have already noted, it remains to consider the case when $n$ is even and $n\geq4$. Further, the discussion leading to Corollary~\ref{highertc} yields
$
sk-1\leq\zcl_s(\Conf_{\mathbb{Z}_2}(\mathbb{R}^n-\{0\},k))\leq\TC_s((\Conf_{\mathbb{Z}_2}(\mathbb{R}^n-\{0\},k))\leq sk,
$
so we can safely assume
\begin{equation}\label{elparo}
\zcl_s(\Conf_{\mathbb{Z}_2}(\mathbb{R}^n-\{0\},k))=sk-1.
\end{equation}
In these terms, for an $(n-2)$-connected cell complex $X$ of dimension at most $k(n-1)$ having the homotopy type of $\Conf_{\mathbb{Z}_2}(\mathbb{R}^n-\{0\},k)$---whose existence is discussed in the paragraph containing~(\ref{losinvariantesdegenero}))---, it suffices\footnote{It is well known that the value of $\TC_s(Z)$ depends only on the homotopy type of the space $Z$.} to show that $\TC_s(X)\leq sk-1$ or, equivalently, that $p(sk-1)$ admits a global (continuous) section. Here $p$ stands for the fibration $e_s\colon P(X)\to X^s$ described at the beginning of this section.

\smallskip
The fiber of $p$, $(\Omega X)^{s-1}$, is $(n-3)$-connected so that, by standard obstruction theory:
\begin{enumerate}
\item there are sections $\phi$ of $p$ on the $(n-2)$-skeleton of $X^s$,
\item the fundamental class of $p$ (also called the primary class of $p$), denoted by $$\eta_p\in H^{n-1}(X^s,\pi_{n-2}((\Omega X)^{s-1})),$$ is the obstruction class for (altering rel.~the $(n-3)$-skeleton and then) extending any such partial section $\phi$ to the $(n-1)$-skeleton of $X^s$, and
\item $p^*(\eta_p)=0$ so that, in our context, $\eta_p$ is an $s$-th zero-divisor.
\end{enumerate}

By Lemma~\ref{classabstrnew} (with $k=n-2$ and $\ell=sk-1$), we get a section of $p(sk-1)$ on the $((n-1)sk-1)$-skeleton of $X^s$ whose obstruction to extending to the $((n-1)sk)$-skeleton is trivial in view of~(\ref{elparo}). The result follows since such an extended section is already defined on all of $X^s$, as $\dim(X)\leq k(n-1)$.
\end{proof}

We close the paper with a series of ideas (inspired by the work in~\cite{FarberGrant} as generalized in~\cite{GG15}) leading to a plausible argument to prove that $\TC_s(\Conf_{\mathbb{Z}_2}(\mathbb{R}^n-\{0\},k))=sk-1$ for $s\geq2$ and for any even $n\,$---including the case $n=2$. First of all, the results in Section~\ref{section2} imply that, for any $R$-module $M$, $$H^*(\Conf_{\mathbb{Z}_2}(\mathbb{R}^n-\{0\}, k)^s;M)\cong H^*(\Conf_{\mathbb{Z}_2}(\mathbb{R}^n-\{0\}, k);R)^{\otimes s}\otimes_R M.$$ Since the rings $H^*(\Conf_{\mathbb{Z}_2}(\mathbb{R}^{n}-\{0\}, k);R)$ and $H^*(\Conf_{\mathbb{Z}_2}(\mathbb{R}^{n+2}-\{0\}, k);R)$ differ only by degree scaling, it follows that, for fixed $s$ and $k$, $\zcl_s(\Conf_{\mathbb{Z}_2}(\mathbb{R}^n-\{0\}, k))$ depends only on the parity of $n$. In particular, in settling the gap by one in Corollary~\ref{highertc} for $\TC_s(\Conf_{\mathbb{Z}_2}(\mathbb{R}^n-\{0\},k))$ when $n$ is even (and $n\geq4$), it is enough to consider a single value of $n$. Of course, the case $n=4$ is the most reasonable instance to explore. However, the situation for $n=2$ not only would seem to be more accessible, but can potentially be more fruitful:

\begin{thm}\label{jugoso}
Assume that $\Conf_{\mathbb{Z}_2}(\mathbb{R}^2-\{0\},k)$ has the homotopy type of a product $S^1\times X$ for some cell complex $X$ of dimension $k-1$. Then $\;\TC_s(\Conf_{\mathbb{Z}_2}(\mathbb{R}^n-\{0\},k))=sk-1\;$ for any $s\geq2$ and any positive even integer $n$.
\end{thm}
\begin{proof}The subadditivity property of $\TC_s$ and Proposition~\ref{lascotas} yield (in that order) the two inequalities in
\begin{equation}\label{cortocircuito}
\TC_s(\Conf_{\mathbb{Z}_2}(\mathbb{R}^2-\{0\},k))\leq\TC_s(S^1)+\TC_s(X)=s-1+\TC_s(X)\leq s-1+s(k-1)=sk-1.
\end{equation}
Another application of Proposition~\ref{lascotas} then yields $\zcl_s(\Conf_{\mathbb{Z}_2}(\mathbb{R}^2-\{0\},k))\leq sk-1$ which, in view of the considerations just before Theorem~\ref{jugoso}, generalizes to
\begin{equation}\label{cuantomas}
\zcl_s(\Conf_{\mathbb{Z}_2}(\mathbb{R}^n-\{0\},k))\leq sk-1,\quad\mbox{for any positive even integer $n$.}
\end{equation}
The inequality in~(\ref{cuantomas}) is now forced to be an equality giving the value of $\TC_s(\Conf_{\mathbb{Z}_2}(\mathbb{R}^n-\{0\},k))$: in view of Corollary~\ref{highertc} and Theorem~\ref{tceszcls} if $n\geq4$, and in view of~(\ref{cortocircuito}) and, again, the considerations just before Theorem~\ref{jugoso} if $n=2$.
\end{proof}

\begin{rem}\label{jsuydbe}
We give evidence supporting the conjecture that the hypothesis in Theorem~\ref{jugoso} holds without restriction. For one, while $\Conf(\mathbb{R}^2,k)$ is known to have the homotopy type of a cell complex of dimension $k-1$, the analogue of Theorem~\ref{jugoso} described in~\cite{FarberGrant} and~\cite{GG15} makes use of the well known and more explicit  homotopy splitting $\Conf(\mathbb{R}^2,k)\simeq Y\times S^1$
with $Y$ a CW complex of dimension $k-2$. But perhaps more interesting is to note that, while the hypothesis in Theorem~\ref{jugoso} is obviously true for $k=1$, it also holds for $k=2$. A direct (but {\it ad-hoc}) argument is given in the final chapter of the Ph.D.~thesis of the second named author.
\end{rem}


\begin{thebibliography}{10}

\bibitem{AoiuinaKlein}
Mokhtar Aouina and John~R. Klein.
\newblock On the homotopy invariance of configuration spaces.
\newblock {\em Algebr. Geom. Topol.}, 4:813--827, 2004.

\bibitem{arnold}
V.~I. Arnol{\cprime}d.
\newblock The cohomology ring of the group of dyed braids.
\newblock {\em Mat. Zametki}, 5:227--231, 1969.

\bibitem{MR3252053}
Ibai Basabe, Jes{\'u}s Gonz{\'a}lez, Yuli~B Rudyak, and Dai Tamaki.
\newblock Higher topological complexity and its symmetrization.
\newblock {\em Algebr. Geom. Topol.}, 14(4):223--244, 2014.

\bibitem{gibe}
Martin Bendersky and Sam Gitler.
\newblock The cohomology of certain function spaces.
\newblock {\em Trans. Amer. Math. Soc.}, 326(1):423--440, 1991.

\bibitem{bct89}
C.-F. B{\"o}digheimer, F.~Cohen, and L.~Taylor.
\newblock On the homology of configuration spaces.
\newblock {\em Topology}, 28(1):111--123, 1989.

\bibitem{bcm89}
C.-F. B{\"o}digheimer, F.~R. Cohen, and R.~J. Milgram.
\newblock Truncated symmetric products and configuration spaces.
\newblock {\em Math. Z.}, 214(2):179--216, 1993.

\bibitem{MR0221507}
Armand Borel.
\newblock {\em Topics in the homology theory of fibre bundles}, volume 1954 of
  {\em Lectures given at the University of Chicago}.
\newblock Springer-Verlag, Berlin-New York, 1967.

\bibitem{rbwh}
Robert~F. Brown and James~H. White.
\newblock Homology and {M}orse theory of third configuration spaces.
\newblock {\em Indiana Univ. Math. J.}, 30(4):501--512, 1981.

\bibitem{MR1344842}
F.~R. Cohen.
\newblock On configuration spaces, their homology, and {L}ie algebras.
\newblock {\em J. Pure Appl. Algebra}, 100(1-3):19--42, 1995.

\bibitem{CTconf}
F.~R. Cohen and L.~R. Taylor.
\newblock Computations of {G}el'fand-{F}uks cohomology, the cohomology of
  function spaces, and the cohomology of configuration spaces.
\newblock In {\em Geometric applications of homotopy theory ({P}roc. {C}onf.,
  {E}vanston, {I}ll., 1977), {I}}, volume 657 of {\em Lecture Notes in Math.},
  pages 106--143. Springer, Berlin, 1978.

\bibitem{CTconf2}
F.~R. Cohen and L.~R. Taylor.
\newblock Configuration spaces: applications to {G}elfand-{F}uks cohomology.
\newblock {\em Bull. Amer. Math. Soc.}, 84(1):134--136, 1978.

\bibitem{Cohen}
Fred Cohen.
\newblock Cohomology of braid spaces.
\newblock {\em Bull. Amer. Math. Soc.}, 79:763--766, 1973.

\bibitem{cinco}
Fred Cohen.
\newblock Homology of {$\Omega ^{(n+1)}\Sigma ^{(n+1)}X$} and
  {$C_{(n+1)}X,\,n>0$}.
\newblock {\em Bull. Amer. Math. Soc.}, 79:1236--1241 (1974), 1973.

\bibitem{clm}
Frederick~R. Cohen, Thomas~J. Lada, and J.~Peter May.
\newblock {\em The homology of iterated loop spaces}.
\newblock Lecture Notes in Mathematics, Vol. 533. Springer-Verlag, Berlin-New
  York, 1976.

\bibitem{Cornea}
Octav Cornea, Gregory Lupton, John Oprea, and Daniel Tanr\'e.
\newblock {\em Lusternik-Schirelmann Category}.
\newblock American Mathematical Society, 2003.

\bibitem{fasolo}
Edward Fadell.
\newblock Homotopy groups of configuration spaces and the string problem of
  {D}irac.
\newblock {\em Duke Math. J.}, 29:231--242, 1962.

\bibitem{fane}
Edward Fadell and Lee Neuwirth.
\newblock Configuration spaces.
\newblock {\em Math. Scand.}, 10:111--118, 1962.

\bibitem{FH}
Edward~R. Fadell and Sufian~Y. Husseini.
\newblock {\em Geometry and topology of configuration spaces}.
\newblock Springer Monographs in Mathematics. Springer-Verlag, Berlin, 2001.

\bibitem{MR2455573}
Michael Farber.
\newblock {\em Invitation to topological robotics}.
\newblock Zurich Lectures in Advanced Mathematics. European Mathematical
  Society (EMS), Z\"urich, 2008.

\bibitem{FarberGrant}
Michael Farber and Mark Grant.
\newblock Topological complexity of configuration spaces.
\newblock {\em Proc. Amer. Math. Soc.}, 137(5):1841--1847, 2009.

\bibitem{FZ}
Eva~Maria Feichtner and G{\"u}nter~M. Ziegler.
\newblock The integral cohomology algebras of ordered configuration spaces of
  spheres.
\newblock {\em Doc. Math.}, 5:115--139 (electronic), 2000.

\bibitem{FZ2}
Eva~Maria Feichtner and G{\"u}nter~M. Ziegler.
\newblock On orbit configuration spaces of spheres.
\newblock {\em Topology Appl.}, 118(1-2):85--102, 2002.
\newblock Arrangements in Boston: a Conference on Hyperplane Arrangements
  (1999).

\bibitem{FTun}
Yves F{\'e}lix and Daniel Tanr{\'e}.
\newblock The cohomology algebra of unordered configuration spaces.
\newblock {\em J. London Math. Soc. (2)}, 72(2):525--544, 2005.

\bibitem{FT}
Yves Felix and Jean-Claude Thomas.
\newblock Configuration spaces and {M}assey products.
\newblock {\em Int. Math. Res. Not.}, (33):1685--1702, 2004.

\bibitem{fuckshs}
D.~B. Fuks.
\newblock Cohomology of the braid group {${\rm mod}\ 2$}.
\newblock {\em Functional Anal. Appl.}, 4:143--151, 1970.

\bibitem{fuma}
William Fulton and Robert MacPherson.
\newblock A compactification of configuration spaces.
\newblock {\em Ann. of Math. (2)}, 139(1):183--225, 1994.

\bibitem{GG15}
Jes\'us Gonz\'alez and Mark Grant.
\newblock Sequential motion planning of non-colliding particles in euclidean
  spaces.
\newblock {\em Accepted for publication in Proc.~Amer.~Math.~Soc.}

\bibitem{Hatcher}
Allen Hatcher.
\newblock {\em Algebraic topology}.
\newblock Cambridge University Press, Cambridge, New York, 2002.

\bibitem{Kriz}
Igor K{\v{r}}{\'{\i}}{\v{z}}.
\newblock On the rational homotopy type of configuration spaces.
\newblock {\em Ann. of Math. (2)}, 139(2):227--237, 1994.

\bibitem{Levitt}
Norman Levitt.
\newblock Spaces of arcs and configuration spaces of manifolds.
\newblock {\em Topology}, 34(1):217--230, 1995.

\bibitem{SalvatoreLongoni}
Riccardo Longoni and Paolo Salvatore.
\newblock Configuration spaces are not homotopy invariant.
\newblock {\em Topology}, 4(2):375--380, 2005.

\bibitem{millof}
R.~James Milgram and Peter L{\"o}ffler.
\newblock The structure of deleted symmetric products.
\newblock In {\em Braids ({S}anta {C}ruz, {CA}, 1986)}, volume~78 of {\em
  Contemp. Math.}, pages 415--424. Amer. Math. Soc., Providence, RI, 1988.

\bibitem{MR1122592}
Mamoru Mimura and Hirosi Toda.
\newblock {\em Topology of {L}ie groups. {I}, {II}}, volume~91 of {\em
  Translations of Mathematical Monographs}.
\newblock American Mathematical Society, Providence, RI, 1991.
\newblock Translated from the 1978 Japanese edition by the authors.

\bibitem{Sch58}
A.~S. Schwarz.
\newblock The genus of a fiber space.
\newblock {\em Amer. Math. Soc. Transl}, 55(2):49--140, 1966.

\bibitem{MR2099074}
Dev~P. Sinha.
\newblock Manifold-theoretic compactifications of configuration spaces.
\newblock {\em Selecta Math. (N.S.)}, 10(3):391--428, 2004.

\bibitem{totaro}
Burt Totaro.
\newblock Configuration spaces of algebraic varieties.
\newblock {\em Topology}, 35(4):1057--1067, 1996.

\bibitem{VAINSHTEIN}
F.~V. Va{\u\i}n{\v{s}}te{\u\i}n.
\newblock The cohomology of braid groups.
\newblock {\em Funktsional. Anal. i Prilozhen.}, 12(2):72--73, 1978.

\bibitem{Xico}
Miguel~A. Xicot{\'e}ncatl.
\newblock On orbit configuration spaces and the rational cohomology of {$F({\bf
  R}{\rm P}^n,k)$}.
\newblock {\em Une d\'egustation topologique: homotopy theory in the {S}wiss
  {A}lps ({A}rolla, 1999)}, 265:233--249, 2000.

\end{thebibliography}
\def\cprime{$'$} \def\cprime{$'$} \def\cprime{$'$} \def\cprime{$'$}
  \def\cprime{$'$}

\end{document}